\documentclass[10pt,twoside]{article}
\usepackage{latexsym}
\usepackage{amsmath}
\usepackage{amssymb}
\usepackage{amscd}
\usepackage{amsthm}
\usepackage[all]{xy}

\newtheorem{thm}{Theorem}[section]
\newtheorem{cor}[thm]{Corollary}
\newtheorem{lem}[thm]{Lemma}
\newtheorem{prop}[thm]{Proposition}

\theoremstyle{definition}
\newtheorem{defn}[thm]{Definition}
\newtheorem{rem}[thm]{Remark}
\newtheorem{exa}[thm]{Example}

\DeclareMathOperator{\GL}{\mathbf{GL}}

\newcommand{\R}{\mathbb R}
\newcommand{\Z}{\mathbb Z}

\newcommand{\C}{\mathbb C}

\def\YEAR{\year}\newcount\VOL\VOL=\YEAR\advance\VOL by-1995
\def\firstpage{1}\def\lastpage{1000}
\def\received{}\def\revised{}
\def\communicated{}

\makeatletter
\def\magnification{\afterassignment\m@g\count@}
\def\m@g{\mag=\count@\hsize6.5truein\vsize8.9truein\dimen\footins8truein}
\makeatother

\font\eightrm=cmr8
\font\caps=cmcsc10                    
\font\Caps=cmcsc10 scaled \magstep1   

\makeatletter
\@specialpagefalse\headheight=8.5pt
\def\DocMath{}
\renewcommand{\@evenhead}{%
    \ifnum\thepage>\lastpage\rlap{\thepage}\hfill%
    \else\rlap{\thepage}\slshape\leftmark\hfill{\caps\SAuthor}\hfill\fi}%
\renewcommand{\@oddhead}{%
    \ifnum\thepage=\firstpage{\DocMath\hfill\llap{\thepage}}%
    \else{\slshape\rightmark}\hfill{\caps\STitle}\hfill\llap{\thepage}\fi}%
\makeatother

\def\TSkip{\bigskip}
\newbox\TheTitle{\obeylines\gdef\GetTitle #1
\ShortTitle  #2
\SubTitle    #3
\Author      #4
\ShortAuthor #5
\EndTitle
{\setbox\TheTitle=\vbox{\baselineskip=20pt\let\par=\cr\obeylines%
\halign{\centerline{\Caps##}\cr\noalign{\medskip}\cr#1\cr}}%
    \copy\TheTitle\TSkip\TSkip%
\def\next{#2}\ifx\next\empty\gdef\STitle{#1}\else\gdef\STitle{#2}\fi%
\def\next{#3}\ifx\next\empty%
    \else\setbox\TheTitle=\vbox{\baselineskip=20pt\let\par=\cr\obeylines%
    \halign{\centerline{\caps##} #3\cr}}\copy\TheTitle\TSkip\TSkip\fi%
\centerline{\caps #4}\TSkip\TSkip%
\def\next{#5}\ifx\next\empty\gdef\SAuthor{#4}\else\gdef\SAuthor{#5}\fi%
\ifx\received\empty\relax
    \else\centerline{\eightrm Received: \received}\fi%
\ifx\revised\empty\TSkip%
    \else\centerline{\eightrm Revised: \revised}\TSkip\fi%
\ifx\communicated\empty\relax
    \else\centerline{\eightrm Communicated by \communicated}\fi\TSkip\TSkip%
\catcode'015=5}}\def\Title{\obeylines\GetTitle}
\def\Abstract{\begingroup\narrower
    \parskip=\medskipamount\parindent=0pt{\caps Abstract. }}
\def\EndAbstract{\par\endgroup\TSkip}

\long\def\MSC#1\EndMSC{\def\arg{#1}\ifx\arg\empty\relax\else
     {\par\narrower\noindent%
     2000 Mathematics Subject Classification: #1\par}\fi}

\long\def\KEY#1\EndKEY{\def\arg{#1}\ifx\arg\empty\relax\else
    {\par\narrower\noindent Keywords and Phrases: #1\par}\fi\TSkip}

\newbox\TheAdd\def\Addresses{\vfill\copy\TheAdd\vfill
    \ifodd\number\lastpage\vfill\eject\phantom{.}\vfill\eject\fi}
{\obeylines\gdef\GetAddress #1
\Address #2
\Address #3
\Address #4
\EndAddress
{\def\xs{4.3truecm}\parindent=0pt
\setbox0=\vtop{{\obeylines\hsize=\xs#1\par}}\def\next{#2}
\ifx\next\empty 
     \setbox\TheAdd=\hbox to\hsize{\hfill\copy0\hfill}
\else\setbox1=\vtop{{\obeylines\hsize=\xs#2\par}}\def\next{#3}
\ifx\next\empty 
     \setbox\TheAdd=\hbox to\hsize{\hfill\copy0\hfill\copy1\hfill}
\else\setbox2=\vtop{{\obeylines\hsize=\xs#3\par}}\def\next{#4}
\ifx\next\empty\ 
     \setbox\TheAdd=\vtop{\hbox to\hsize{\hfill\copy0\hfill\copy1\hfill}
                \vskip20pt\hbox to\hsize{\hfill\copy2\hfill}}
\else\setbox3=\vtop{{\obeylines\hsize=\xs#4\par}}
     \setbox\TheAdd=\vtop{\hbox to\hsize{\hfill\copy0\hfill\copy1\hfill}
            \vskip20pt\hbox to\hsize{\hfill\copy2\hfill\copy3\hfill}}
\fi\fi\fi\catcode'015=5}}\gdef\Address{\obeylines\GetAddress}

\allowdisplaybreaks[1]

\begin{document}
\Title
    A Lefschetz fixed point formula for singular arithmetic
              schemes with smooth generic fibres
\ShortTitle
Singular fixed point formula of Lefschetz type
\SubTitle
\Author
Shun Tang
\ShortAuthor
\EndTitle
\Abstract
In this article, we consider singular equivariant arithmetic
schemes whose generic fibres are smooth. For such schemes, we
prove a relative fixed point formula of Lefschetz type in the
context of Arakelov geometry. This formula is an analog, in the
arithmetic case, of the Lefschetz formula proved by R. W. Thomason
in \cite{Th}. In particular, our result implies a fixed point
formula which was conjectured by V. Maillot and D. R\"{o}ssler in
\cite{MR}.
\EndAbstract
\MSC 14C40, 14G40, 14L30, 58J20, 58J52
\EndMSC
\KEY fixed point formula, singular arithmetic scheme, Arakelov geometry
\EndKEY
\Address
         D\'{e}partement de Math\'{e}matiques
         B\^{a}timent 425
         Universit\'{e} Paris-Sud 11
         91405 Orsay, France
         shun.tang@math.u-psud.fr
\Address
\Address
\Address
\EndAddress
\section{Introduction}
It is the aim of this article to prove a singular Lefschetz fixed
point formula for some schemes which admit the actions of a
diagonalisable group scheme, in the context of Arakelov geometry.
We first roughly describe the history of the study of such
Lefschetz fixed point formulae and relative Lefschetz-Riemann-Roch
problems.

Let $k$ be an algebraically closed field and let $n$ be an integer
which is prime to the characteristic of $k$. A projective
$k-$variety $X$ which admits an automorphism $g$ of order $n$ will
be called an equivariant variety. An equivariant coherent sheaf on
$X$ is a coherent sheaf $F$ on $X$ together with a homomorphism
$\varphi: g^*F\rightarrow F$. It is clear that this homomorphism
induces a family of endomorphisms $H^i(\varphi)$ on cohomology
spaces $H^i(X,F)$.

A classical Lefschetz fixed point formula is to give an expression
of the alternating sum of the traces of $H^i(\varphi)$, as a sum
of the contributions from the components of the fixed point
subvariety $X_g$. On the other hand, roughly speaking, a
Lefschetz-Riemann-Roch theorem is a commutative diagram in
equivariant $K-$theory which can be regarded as a Grothendieck
type generalization of the Lefschetz fixed point formula. Indeed,
when we choose the base variety in such a commutative diagram to
be a point, we will get the ordinary Lefschetz fixed point
formula. If $X$ is nonsingular, P. Donovan has proved such a
theorem in \cite{Do} by using the results and some of the methods
of the paper of A. Borel and J. P. Serre on the
Grothendieck-Riemann-Roch theorem (cf. \cite{BS}). In \cite{BFQ},
P. Baum, W. Fulton and G. Quart generalized Donovan's theorem to
singular varieties, the key step of their proof heavily relies on
an elegant method called the deformation to the normal cone.
Denote by $G_0(X,g)$ (resp. $K_0(X,g)$) the Quillen's algebraic
$K$-group associated to the category of equivariant coherent
sheaves (resp. vector bundles of finite rank) on $X$, then
$K_0({\rm Pt},g)$ is isomorphic to the group ring $\Z[k]$ and
$G_0(X,g)$ (resp. $K_0(X,g)$) has a natural $K_0({\rm
Pt},g)$-module (resp. $K_0({\rm Pt},g)$-algebra) structure. Let
$f$ be an equivariantly projective morphism between two
equivariant varieties $X$ and $Y$, then it is possible to define a
push-forward morphism $f_*$ from $G_0(X,g)$ to $G_0(Y,g)$ in a
rather standard way. Let $\mathcal{R}$ be any flat $K_0({\rm
Pt},g)$-algebra in which $1-\zeta$ is invertible for each
non-trivial $n$-th root of unity $\zeta$ in $k$. The main result
of Baum, Fulton and Quart reads: there exists a family of group
homomorphisms $L.$ between $K$-groups making the following diagram
\begin{displaymath}
\xymatrix{
G_0(X,g) \ar[r]^-{L.} \ar[d]_{f_*} & G_0(X_g,g)\otimes_{\Z[k]}\mathcal{R} \ar[d]^{{f_g}_*} \\
 G_0(Y,g) \ar[r]^-{L.} & G_0(Y_g,g)\otimes_{\Z[k]}\mathcal{R}}
\end{displaymath} commutative. If $Z$ is a nonsingular equivariant variety such
that there exists an equivariant closed immersion from $X$ to $Z$,
then for every equivariant coherent sheaf $E$ on $X$ the
homomorphism $L.$ is exactly given by the formula
\begin{displaymath}
L.(E)=\lambda_{-1}^{-1}(N_{Z/{Z_g}}^\vee)\cdot\sum_j(-1)^j{\rm
Tor}_{\mathcal{O}_Z}^j(i_*E,\mathcal{O}_{Z_g})
\end{displaymath}
where $N_{Z/{Z_g}}$ stands for the normal bundle of $Z_g$ in $Z$
and
$\lambda_{-1}(N_{Z/{Z_g}}^\vee):=\sum(-1)^j\wedge^jN_{Z/{Z_g}}^\vee$.

We would like to indicate that one can use the same method so
called the deformation to the normal cone to extend Baum, Fulton
and Quart's result to general scheme case where $X$ and $Y$ are
Noetherian, separated schemes endowed with projective actions of
the diagonalisable group scheme $\mu_n$ associated to $\Z/{n\Z}$.
Here by a $\mu_n$-action on $X$ we understand a morphism $m_X:
\mu_n\times X\to X$ which statisfies some compatibility
properties. Denote by $p_X$ the projection from $\mu_n\times X$ to
$X$. For a coherent $\mathcal{O}_X$-module $E$ on $X$, a
$\mu_n$-action on $E$ we mean an isomorphism $m_E: p_X^*E\to
m_X^*E$ which satisfies certain associativity properties. We refer
to \cite{Ko} and \cite[Section 2]{KR1} for the group scheme action
theory we are talking about.

In \cite{Th}, R. W. Thomason used another way to generalize Baum,
Fulton and Quart's result to the scheme case, and he removed the
condition of projectivity. The strategy Thomason followed was to
use Quillen's localization sequence for higher equivariant
$K$-groups to prove an algebraic concentration theorem. Let $D$ be
an integral Noetherian ring, and let $\mu_n$ be the diagonalisable
group scheme over $D$ associated to $\Z/{n\Z}$. Denote the ring
$\Z[\Z/{n\Z}]\cong \Z[T]/{(1-T^n)}$ by $R(\mu_n)$. We consider the
prime ideal $\rho$ in $R(\mu_n)$ which is the kernel of the
canonical morphism $\Z[T]/{(1-T^n)}\rightarrow \Z[T]/{(\Phi_n)}$
where $\Phi_n$ stands for the $n$-th cyclotomic polynomial. By
construction the elements $1-T^k$ for $k=1,\ldots,n-1$ are not
contained in $\rho$. Let $X$ be a $\mu_n$-equivariant scheme over
$D$, then $G_0(X,\mu_n)$ (resp. $K_0(X,\mu_n)$) has a natural
$R(\mu_n)$-module (resp. $R(\mu_n)$-algebra) structure because
$K_0(D,\mu_n)$ is isomorphic to $K_0(D)[T]/{(1-T^n)}$. Denote by
$i$ the inclusion from $X_{\mu_n}$ to $X$. The algebraic
concentration theorem reads: there exists a natural group
homomorphism $i_*$ from $G_0(X_{\mu_n},\mu_n)_{\rho}$ to
$G_0(X,\mu_n)_{\rho}$ which is an isomorphism. Moreover, if $X$ is
regular, the inverse map of $i_*$ is given by
$\lambda_{-1}^{-1}(N_{X/{X_{\mu_n}}}^\vee)\cdot i^*$ where
$N_{X/{X_{\mu_n}}}$ is the normal bundle of $X_{\mu_n}$ in $X$.
This concentration theorem can be used to prove a singular
Lefschetz fixed point formula which is an extension of Baum,
Fulton and Quart's result in general scheme case. Thomason's
approach has nothing to do with the construction of the
deformation to the normal cone, and the localization he used is
slightly weaker than Baum, Fulton and Quart's in the sense that
the complement of the ideal $\rho$ in $R(\mu_n)$ is not the
smallest algebra in which the elements $1-T^k$ ($k=1,\ldots,n-1$)
are invertible. If one exactly chooses $\mathcal{R}$ to be the
complement of the ideal $\rho$ in $R(\mu_n)$, then these two
localizations are equal to each other.

In \cite{KR1}, K. K\"{o}hler and D. R\"{o}ssler generalized the
regular case of Baum, Fulton and Quart's result to Arakelov
geometry. To every regular $\mu_n$-equivariant arithmetic scheme
$X$, they associate an equivariant arithmetic $K_0$-group
$\widehat{K_0}(X,\mu_n)$ which contains some smooth form class on
$X_{\mu_n}(\C)$ as analytic datum. Such an equivariant arithmetic
$K_0$-group has a ring structure and moreover it can be made to be
an $R(\mu_n)$-algebra. Let $\overline{N}_{X/{X_{\mu_n}}}$ be the
normal bundle with respect to the regular immersion
$X_{\mu_n}\hookrightarrow X$ which is endowed with the quotient
metric induced by a chosen K\"{a}hler metric of $X(\C)$, then the
main theorem in \cite{KR1} reads: the element
$\lambda_{-1}(\overline{N}_{X/{X_{\mu_n}}}^\vee)$ is invertible in
$\widehat{K_0}(X_{\mu_n},\mu_n)\otimes_{R(\mu_n)}\mathcal{R}$ and
we have the following commutative diagram
\begin{displaymath}
\xymatrix{
\widehat{K_0}(X,\mu_n) \ar[rr]^-{\Lambda_R(f)^{-1}\cdot \tau} \ar[d]_{f_*} && \widehat{K_0}(X_{\mu_n},\mu_n)\otimes_{R(\mu_n)}\mathcal{R} \ar[d]^{{f_{\mu_n}}_*} \\
 \widehat{K_0}(D,\mu_n) \ar[rr]^-{\iota} && \widehat{K_0}(D,\mu_n)\otimes_{R(\mu_n)}\mathcal{R}}
\end{displaymath}
where
$\Lambda_R(f):=\lambda_{-1}(\overline{N}_{X/{X_{\mu_n}}}^\vee)\cdot(1+R_g(N_{X/{X_{\mu_n}}}))$,
$\tau$ stands for the restriction map and $\iota$ is the natural
morphism from a ring or a module to its localization which sends
an element $e$ to $\frac{e}{1}$. Here $R_g(\cdot)$ is the
equivariant $R$-genus, the definition of the two push-forward
morphisms $f_*$ and ${f_{\mu_n}}_*$ involves an important analytic
datum which is called the equivariant analytic torsion. The
strategy K\"{o}hler and R\"{o}ssler followed to prove such an
arithmetic Lefschetz-Riemann-Roch theorem was to use the
construction of the deformation to the normal cone to prove an
analog of this theorem for equivariant closed immersions. After
that, they decompose the morphism $f$ to a closed immersion $h$
from $X$ to some projective space $\mathbb{P}^r_D$ followed by a
smooth morphism $p$ from $\mathbb{P}^r_D$ to ${\rm Spec}(D)$. Then
the theorem in general situation follows from an argument of
investigating the behavior of the error term under the morphisms
$h$ and $p$.

Provided X. Ma's extension of equivariant analytic torsion to
higher equivariant analytic torsion form, it was conjectured by
K\"{o}hler and R\"{o}ssler in \cite{KR2} that an analog of
\cite[Theorem 4.4]{KR1} in relative setting holds. We have already
proved this conjecture in \cite{T1}. Our method is similar to
Thomason's, we first show that there exists an arithmetic
concentration theorem in Arakelov geometry and then deduce from it
the relative Lefschetz fixed point formula. The same as Thomason's
approach, our method has nothing to do with the construction of
the deformation to the normal cone, but unfortunately it only
works for regular arithmetic schemes.

One may naturally asked that whether it is possible to construct a
more general arithmetic $\widehat{G_0}$-theory and prove a
relative Lefschetz fixed point formula for singular arithmetic
schemes which is entirely an analog of Thomason's singular
Lefschetz formula in Arakelov geometry. The answer is Yes, and
this is what we have done in this article. To do this, one needs a
$\widehat{G_0}$-theoretic vanishing theorem which can be viewed as
an extension of K\"{o}hler and R\"{o}ssler's fixed point formula
for closed immersions to the singular case. The proof of such a
vanishing theorem occupies a lot of space in this article. Let $X$
and $Y$ be two singular equivariant arithmetic schemes with smooth
generic fibres, and let $f: X\to Y$ be an equivariant morphism
which is smooth on the complex numbers. Assume that the
$\mu_n$-action on $Y$ is trivial and $f$ can be decomposed to be
$h\circ i$ where $i$ is an equivariant closed immersion from $X$
to some regular arithmetic scheme $Z$ and $h: Z\to Y$ is
equivariant and smooth on the complex numbers. Let
$\overline{\eta}$ be an equivariant hermitian sheaf on $X$.
Referring to Section 6.1 for the explanations of various
notations, we announce that our main theorem in this article is
the following equality which holds in
$\widehat{G_0}(Y,\mu_n,\emptyset)_{\rho}$:
\begin{align*}
f_*(\overline{\eta})=&{f_{\mu_n}^Z}_*(i_{\mu_n}^*(\lambda_{-1}^{-1}(\overline{N}^\vee_{Z/{Z_{\mu_n}}}))\cdot\sum_k(-1)^k{\rm
Tor}_{\mathcal{O}_Z}^k(i_*\overline{\eta},\overline{\mathcal{O}}_{Z_{\mu_n}}))\\
&+\int_{X_g/Y}{\rm Td}(Tf_g,\omega^X){\rm
ch}_g(\overline{\eta})\widetilde{{\rm
Td}}_g(\overline{\mathcal{F}},\omega^X){\rm
Td}_g^{-1}(\overline{F})\\
&-\int_{X_g/Y}{\rm Td}_g(Tf){\rm
ch}_g(\eta)R_g(N_{X/{X_g}})\\
&+\int_{X_g/Y}\widetilde{{\rm Td}}(Tf_g,\omega^X,\omega^Z_X){\rm
ch}_g(\overline{\eta}){\rm Td}_g(\overline{N}_{Z/{Z_g}}){\rm
Td}_g^{-1}(\overline{F}).
\end{align*}

The structure of this article is as follows. In Section 2, we
recall some differential-geometric facts for the convenience of
the reader. In Section 3, we formulate and prove a vanishing
theorem for equivariant closed immersions in a purely analytic
setting. In Section 4, we define the arithmtic $G_0$-groups with
respect to fixed wave front sets which are necessary for our later
arguments. In Section 5 and Section 6, we formulate and prove the
arithmetic concentration theorem and the relative Lefschetz fixed
point formula for singular arithmetic schemes.

\textbf{Acknowledgements.} The author wishes to thank his thesis
advisor Damian R\"{o}ssler for providing such an interesting
topic, also for his constant encouragement and for many fruitful
discussions between them. The author is greatful to Xiaonan Ma,
from whom he gets many meaningful comments and suggestions.
Finally, thanks to the referee, for his/her excellent work.

\section{Differential-geometric preliminaries}

\subsection{Equivariant Chern-Weil theory}
Let $G$ be a compact Lie group and let $M$ be a compact complex
manifold which admits a holomorphic $G$-action. By an equivariant
hermitian vector bundle on $M$, we understand a hermitian vector
bundle on $M$ which admits a $G$-action compatible with the
$G$-structure of $M$ and whose metric is $G$-invariant. Let $g\in
G$ be an automorphism of $M$, we shall denote by $M_g=\{x\in M\mid
g\cdot x=x\}$ the fixed point submanifold. $M_g$ is also a compact
complex manifold.

Now let $\overline{E}$ be an equivariant hermitian vector bundle
on $M$, it is well known that the restriction of $\overline{E}$ to
$M_g$ splits as a direct sum
\begin{displaymath}
\overline{E}\mid_{M_g}=\bigoplus_{\zeta\in
S^1}\overline{E}_{\zeta}
\end{displaymath}
where the equivariant structure $g^E$ of $E$ acts on
$\overline{E}_{\zeta}$ as multiplication by $\zeta$. We often
write $\overline{E}_g$ for $\overline{E}_1$ and call it the
$0$-degree part of $\overline{E}\mid_{M_g}$. As usual,
$A^{p,q}(M)$ stands for the space of $(p,q)$-forms
$\Gamma^{\infty}(M,\Lambda^pT^{*(1,0)}M\wedge\Lambda^qT^{*(0,1)}M)$,
we define
\begin{displaymath}
\widetilde{A}(M)=\bigoplus_{p=0}^{{\rm dim}M}(A^{p,p}(M)/({\rm
Im}\partial+{\rm Im}\overline{\partial})).
\end{displaymath}
We denote by $\Omega^{\overline{E}_\zeta}\in A^{1,1}(M_g)$ the
curvature form associated to $\overline{E}_\zeta$. Let
$(\phi_\zeta)_{\zeta\in S^1}$ be a family of $\GL(\C)$-invariant
formal power series such that $\phi_\zeta\in \C[[\mathbf{gl}_{{\rm
rk}E_\zeta}(\C)]]$ where ${\rm rk}E_\zeta$ stands for the rank of
$E_\zeta$ which is a locally constant function on $M_g$. Moreover,
let $\phi\in \C[[\bigoplus_{\zeta\in S^1}\C]]$ be any formal power
series. We have the following definition.

\begin{defn}\label{s201}
The way to associate a smooth form to an equivariant hermitian
vector bundle $\overline{E}$ by setting
\begin{displaymath}
\phi_g(\overline{E}):=\phi((\phi_{\zeta}(-\frac{\Omega^{\overline{E}_\zeta}}{2\pi
i}))_{\zeta\in S^1})
\end{displaymath}
is called an $g$-equivariant Chern-Weil theory associated to
$(\phi_\zeta)_{\zeta\in S^1}$ and $\phi$. The class of
$\phi_g(\overline{E})$ in $\widetilde{A}(M_g)$ is independent of
the metric.
\end{defn}

Write ${\rm dd}^c$ for the differential operator
$\frac{\overline{\partial}\partial}{2\pi i}$. The theory of
equivariant secondary characteristic classes is described in the
following theorem.

\begin{thm}\label{s202}
To every short exact sequence $\overline{\varepsilon}:
0\rightarrow \overline{E}'\rightarrow \overline{E}\rightarrow
\overline{E}''\rightarrow 0$ of equivariant hermitian vector
bundles on $M$, there is a unique way to attach a class
$\widetilde{\phi}_g(\overline{\varepsilon})\in \widetilde{A}(M_g)$
which satisfies the following three conditions:

(i). $\widetilde{\phi}_g(\overline{\varepsilon})$ satisfies the
differential equation
\begin{displaymath}
{\rm
dd}^c\widetilde{\phi}_g(\overline{\varepsilon})=\phi_g(\overline{E}'\oplus\overline{E}'')-\phi_g(\overline{E});
\end{displaymath}

(ii). for every equivariant holomorphic map $f: M'\rightarrow M$,
$\widetilde{\phi}_g(f^*\overline{\varepsilon})=f_g^*\widetilde{\phi}_g(\overline{\varepsilon})$;

(iii). $\widetilde{\phi}_g(\overline{\varepsilon})=0$ if
$\overline{\varepsilon}$ is equivariantly and orthogonally split.
\end{thm}
\begin{proof}
This is \cite[Theorem 3.4]{KR1}.
\end{proof}

We now give some examples of equivariant character forms and their
corresponding secondary characteristic classes.

\begin{exa}\label{s204}
(i). The equivariant Chern character form ${\rm
ch}_g(\overline{E}):=\sum_{\zeta\in S^1}\zeta{\rm
ch}(\overline{E}_\zeta)$.

(ii). The equivariant Todd form ${\rm
Td}_g(\overline{E}):=\frac{c_{{\rm rk}E_g}(\overline{E}_g)}{{\rm
ch}_g(\sum_{j=0}^{{\rm rk}E}(-1)^j\wedge^j\overline{E}^\vee)}$. As
in \cite[Thm. 10.1.1]{Hi} one can show that
\begin{displaymath}
{\rm Td}_g(\overline{E})={\rm
Td}(\overline{E}_g)\prod_{\zeta\neq1}{\rm
det}(\frac{1}{1-\zeta^{-1}e^{\frac{\Omega^{\overline{E}_\zeta}}{2\pi
i}}}). \end{displaymath}

(iii). Let $\overline{\varepsilon}: 0\rightarrow
\overline{E}'\rightarrow \overline{E}\rightarrow
\overline{E}''\rightarrow 0$ be a short exact sequence of
hermitian vector bundles. The secondary Bott-Chern characteristic
class is given by $\widetilde{{\rm
ch}}_g(\overline{\varepsilon})=\sum_{\zeta\in S^1}\zeta
\widetilde{{\rm ch}}(\overline{\varepsilon}_\zeta)$.

(iv). If the equivariant structure $g^\varepsilon$ has the
eigenvalues $\zeta_1,\cdots,\zeta_m$, then the secondary Todd
class is given by
\begin{displaymath}
\widetilde{{\rm
Td}}_g(\overline{\varepsilon})=\sum_{i=1}^m\prod_{j=1}^{i-1}{\rm
Td}_g(\overline{E}_{\zeta_j})\cdot\widetilde{{\rm
Td}}(\overline{\varepsilon}_{\zeta_i})\cdot\prod_{j=i+1}^m{\rm
Td}_g(\overline{E}'_{\zeta_j}+\overline{E}''_{\zeta_j}).
\end{displaymath}
\end{exa}

\begin{rem}\label{s205}
One can use Theorem~\ref{s202} to give a proof of the statements
(iii) and (iv) in the examples above.
\end{rem}

Let $E$ be an equivariant hermitian vector bundle with two
different hermitian metrics $h_1$ and $h_2$, we shall write
$\widetilde{\phi}_g(E,h_1,h_2)$ for the equivariant secondary
characteristic class associated to the exact sequence
\begin{displaymath}
0\to (E,h_1)\to (E,h_2)\to 0 \to 0
\end{displaymath} where the map
from $(E,h_1)$ to $(E,h_2)$ is the identity map.

\subsection{Equivariant analytic torsion forms}
In \cite{BK}, J.-M. Bismut and K. K\"{o}hler extended the
Ray-Singer analytic torsion to the higher analytic torsion form
$T$ for a holomorphic submersion. The purpose of making such an
extension is that the differential equation on ${\rm dd}^cT$ gives
a refinement of the Grothendieck-Riemann-Roch theorem. Later, in
his article \cite{Ma1}, X. Ma generalized J.-M. Bismut and K.
K\"{o}hler's results to the equivariant case. In this subsection,
we shall briefly recall Ma's construction of the equivariant
analytic torsion form. This construction is not very important for
understanding the rest of this article, but the equivariant
analytic torsion form itself will be used to define a reasonable
push-forward morphism between equivariant arithmetic $G_0$-groups.

We first fix some notations and assumptions. Let $f: M\rightarrow
B$ be a proper holomorphic submersion of complex manifolds, and
let $TM$, $TB$ be the holomorphic tangent bundle on $M$, $B$.
Denote by $J^{Tf}$ the complex structure on the real relative
tangent bundle $T_{\R}f$, and assume that $h^{Tf}$ is a hermitian
metric on $Tf$ which induces a Riemannian metric $g^{Tf}$. Let
$T^HM$ be a vector subbundle of $TM$ such that $TM=T^HM\oplus Tf$,
the following definition of K\"{a}hler fibration was given in
\cite[Def. 1.4]{BGS2}.

\begin{defn}\label{201}
The triple $(f,h^{Tf},T^HM)$ is said to define a K\"{a}hler
fibration if there exists a smooth real $(1,1)-$form $\omega$
which satisfies the following three conditions:

(i). $\omega$ is closed;

(ii). $T_{\R}^HM$ and $T_{\R}f$ are orthogonal with respect to
$\omega$;

(iii). if $X,Y\in T_{\R}f$, then $\omega(X,Y)=\langle
X,J^{Tf}Y\rangle_{g^{Tf}}$.
\end{defn}

It was shown in \cite[Thm. 1.5 and 1.7]{BGS2} that for a given
K\"{a}hler fibration, the form $\omega$ is unique up to addition
of a form $f^*\eta$ where $\eta$ is a real, closed $(1,1)$-form on
$B$. Moreover, for any real, closed $(1,1)$-form $\omega$ on $M$
such that the bilinear map $X,Y\in T_{\R}f\mapsto
\omega(J^{Tf}X,Y)\in \R$ defines a Riemannian metric and hence a
hermitian product $h^{Tf}$ on $Tf$, we can define a K\"{a}hler
fibration whose associated $(1,1)$-form is $\omega$. In
particular, for a given $f$, a K\"{a}hler metric on $M$ defines a
K\"{a}hler fibration if we choose $T^HM$ to be the orthogonal
complement of $Tf$ in $TM$ and $\omega$ to be the K\"{a}hler form
associated to this metric.

We now recall the Bismut superconnection of a K\"{a}hler
fibration. Let $(\xi,h^\xi)$ be a hermitian complex vector bundle
on $M$. Let $\nabla^{Tf}$, $\nabla^\xi$ be the holomorphic
hermitian connections on $(Tf,h^{Tf})$ and $(\xi,h^\xi)$. Let
$\nabla^{\Lambda(T^{*(0,1)}f)}$ be the connection induced by
$\nabla^{Tf}$ on $\Lambda(T^{*(0,1)}f)$. Then we may define a
connection on $\Lambda(T^{*(0,1)}f)\otimes\xi$ by setting
\begin{displaymath}
\nabla^{\Lambda(T^{*(0,1)}f)\otimes\xi}=\nabla^{\Lambda(T^{*(0,1)}f)}\otimes
1+1\otimes\nabla^\xi.
\end{displaymath} Let $E$ be the
infinite-dimensional bundle on $B$ whose fibre at each point $b\in
B$ consists of the $C^\infty$ sections of
$\Lambda(T^{*(0,1)}f)\otimes\xi\mid_{f^{-1}b}$. This bundle $E$ is
a smooth $\Z$-graded bundle. We define a connection $\nabla^E$ on
$E$ as follows. If $U\in T_{\R}B$, let $U^H$ be the lift of $U$ in
$T^H_{\R}M$ so that $f_*U^H=U$. Then for every smooth section $s$
of $E$ over $B$, we set
\begin{displaymath}
\nabla_U^Es=\nabla_{U^H}^{\Lambda(T^{*(0,1)}f)\otimes\xi}s.
\end{displaymath}
For $b\in B$, let $\overline{\partial}^{Z_b}$ be the Dolbeault
operator acting on $E_b$, and let $\overline{\partial}^{Z_b*}$ be
its formal adjoint with respect to the canonical hermitian product
on $E_b$ (cf. \cite[1.2]{Ma1}). Let $C(T_{\R}f)$ be the Clifford
algebra of $(T_{\R}f,h^{Tf})$, then the bundle
$\Lambda(T^{*(0,1)}f)\otimes\xi$ has a natural
$C(T_{\R}f)$-Clifford module structure. Actually, if $U\in Tf$,
let $U'\in T^{*(0,1)}f$ correspond to $U$ defined by
$U'(\cdot)=h^{Tf}(U,\cdot)$, then for $U, V\in Tf$ we set
\begin{displaymath}
c(U)=\sqrt{2}U'\wedge,\quad
c(\overline{V})=-\sqrt{2}i_{\overline{V}}
\end{displaymath} where
$i_{(\cdot)}$ is the contraction operator (cf. \cite[Definition
1.6]{BGV}). Moreover, if $U,V\in T_{\R}B$, we set
$T(U^H,V^H)=-P^{Tf}[U^H,V^H]$ where $P^{Tf}$ stands for the
canonical projection from $TM$ to $Tf$.

\begin{defn}\label{202}
Let $e_1,\ldots,e_{2m}$ be a basis of $T_{\R}B$, and let
$e^1,\ldots,e^{2m}$ be the dual basis of $T_{\R}^*B$. Then the
element
\begin{displaymath}
c(T)=\frac{1}{2}\sum_{1\leq\alpha,\beta\leq 2m}e^\alpha\wedge
e^\beta\widehat{\otimes}c(T(e_\alpha^H,e_\beta^H))
\end{displaymath}
is a section of $(f^*\Lambda(T^*_{\R}B)\widehat{\otimes}{\rm
End}(\Lambda(T^{*(0,1)}f)\otimes\xi))^{\rm odd}$.
\end{defn}

\begin{defn}\label{203}
For $u>0$, the Bismut superconnection on $E$ is the differential
operator
\begin{displaymath}
B_u=\nabla^E+\sqrt{u}(\overline{\partial}^Z+\overline{\partial}^{Z*})-\frac{1}{2\sqrt{2u}}c(T)
\end{displaymath}
on
$f^*(\Lambda(T_{\R}^*B))\widehat{\otimes}(\Lambda(T^{*(0,1)}f)\otimes\xi)$.
\end{defn}

\begin{defn}\label{204}
Let $N_V$ be the number operator on
$\Lambda(T^{*(0,1)}f)\otimes\xi$ and on $E$, namely $N_V$ acts as
multiplication by $p$ on $\Lambda^p(T^{*(0,1)}f)\otimes\xi$. For
$U,V\in T_{\R}B$, set
$\omega^{H\overline{H}}(U,V)=\omega^M(U^H,V^H)$ where $\omega^M$
is the closed form in the definition of K\"{a}hler fibration.
Furthermore, for $u>0$, set
$N_u=N_V+\frac{i\omega^{H\overline{H}}}{u}$.
\end{defn}

We now turn to the equivariant case. Let $G$ be a compact Lie
group, we shall assume that all complex manifolds, hermitian
vector bundles and holomorphic morphisms considered above are
$G-$equivariant and all metrics are $G$-invariant. We will
additionally assume that the direct images $R^kf_*\xi$ are all
locally free so that the $G$-equivariant coherent sheaf $R^\cdot
f_*\xi$ is locally free and hence a $G$-equivariant vector bundle
over $B$. \cite[1.2]{Ma1} gives a $G$-invariant hermitian metric
(the $L^2$-metric) $h^{R^\cdot f_*\xi}$ on the vector bundle
$R^\cdot f_*\xi$.

For $g\in G$, let $M_g=\{x\in M\mid g\cdot x=x\}$ and $B_g=\{b\in
B\mid g\cdot b=b\}$ be the fixed point submanifolds, then $f$
induces a holomorphic submersion $f_g: M_g\rightarrow B_g$. Let
$\Phi$ be the homomorphism $\alpha\mapsto (2i\pi)^{-{\rm
deg}\alpha/2}$ of $\Lambda^{\rm even}(T_{\R}^*B)$ into itself. We
put
\begin{displaymath}
{\rm ch}_g(R^\cdot f_*\xi,h^{R^\cdot f_*\xi})=\sum_{k=0}^{{\rm
dim}M-{\rm dim}B}(-1)^k{\rm ch}_g(R^k f_*\xi,h^{R^k f_*\xi})
\end{displaymath}
and
\begin{displaymath}
{\rm ch}_g'(R^\cdot f_*\xi,h^{R^\cdot f_*\xi})=\sum_{k=0}^{{\rm
dim}M-{\rm dim}B}(-1)^kk{\rm ch}_g(R^k f_*\xi,h^{R^k f_*\xi}).
\end{displaymath}

\begin{defn}\label{205}
For $s\in \C$ with ${\rm Re}(s)>1$, let
\begin{displaymath}
\zeta_1(s)=-\frac{1}{\Gamma(s)}\int_0^1 u^{s-1}(\Phi{\rm
Tr_s}[gN_u{\rm exp}(-B_u^2)]-{\rm ch}_g'(R^\cdot f_*\xi,h^{R^\cdot
f_*\xi})){\rm d}u
\end{displaymath} and similarly for $s\in \C$
with ${\rm Re}(s)<\frac{1}{2}$, let
\begin{displaymath}
\zeta_2(s)=-\frac{1}{\Gamma(s)}\int_1^\infty u^{s-1}(\Phi{\rm
Tr_s}[gN_u{\rm exp}(-B_u^2)]-{\rm ch}_g'(R^\cdot f_*\xi,h^{R^\cdot
f_*\xi})){\rm d}u.
\end{displaymath}
\end{defn}

X. Ma has proved that $\zeta_1(s)$ extends to a holomorphic
function of $s\in \C$ near $s=0$ and $\zeta_2(s)$ is a holomorphic
function of $s$.

\begin{defn}\label{206}
The smooth form $T_g(\omega^M,h^\xi):=\frac{\partial}{\partial
s}(\zeta_1+\zeta_2)(0)$ on $B_g$ is called the equivariant
analytic torsion form.
\end{defn}

\begin{thm}\label{207}
The form $T_g(\omega^M,h^\xi)$ lies in $\bigoplus_{p\geq
0}A^{p,p}(B_g)$ and satisfies the following differential equation
\begin{displaymath}
{\rm dd}^cT_g(\omega^M,h^\xi)={\rm ch}_g(R^\cdot f_*\xi,h^{R^\cdot
f_*\xi})-\int_{M_g/{B_g}}{\rm Td}_g(Tf,h^{Tf}){\rm
ch}_g(\xi,h^\xi).
\end{displaymath} Here $A^{p,p}(B_g)$ stands for
the space of smooth forms on $B_g$ of type $(p,p)$.
\end{thm}
\begin{proof}
This is \cite[Theorem 2.12]{Ma1}.
\end{proof}

We define a secondary characteristic class
\begin{displaymath}
\widetilde{{\rm ch}}_g(R^\cdot f_*\xi,h^{R^\cdot
f_*\xi},h'^{R^\cdot f_*\xi}):=\sum_{k=0}^{{\rm dim}M-{\rm
dim}B}(-1)^k\widetilde{{\rm ch}}_g(R^k f_*\xi,h^{R^k
f_*\xi},h'^{R^k f_*\xi})
\end{displaymath}
such that it satisfies the following differential equation
\begin{displaymath}
{\rm dd}^c\widetilde{{\rm ch}}_g(R^\cdot f_*\xi,h^{R^\cdot
f_*\xi},h'^{R^\cdot f_*\xi})={\rm ch}_g(R^\cdot f_*\xi,h^{R^\cdot
f_*\xi})-{\rm ch}_g(R^\cdot f_*\xi,h'^{R^\cdot f_*\xi}),
\end{displaymath}
then the anomaly formula can be described as follows.

\begin{thm}\label{208}(Anomaly formula)
Let $\omega'$ be the form associated to another K\"{a}hler
fibration for $f: M\rightarrow B$. Let $h'^{Tf}$ be the metric on
$Tf$ in this new fibration and let $h'^\xi$ be another metric on
$\xi$. The following identity holds in
$\widetilde{A}(B_g):=\bigoplus_{p\geq 0}(A^{p,p}(B_g)/{({\rm
Im}\partial+{\rm Im}\overline{\partial})})$:
\begin{align*}
T_g(\omega^M,h^\xi)-T_g(\omega'^M,h'^\xi)=&\widetilde{{\rm
ch}}_g(R^\cdot f_*\xi,h^{R^\cdot f_*\xi},h'^{R^\cdot
f_*\xi})\\
&-\int_{M_g/{B_g}}[\widetilde{{\rm
Td}}_g(Tf,h^{Tf},h'^{Tf}){\rm ch}_g(\xi,h^\xi)\\
&+{\rm Td}_g(Tf,h'^{Tf})\widetilde{{\rm
ch}}_g(\xi,h^\xi,h'^\xi)].
\end{align*} In particular, the
class of $T_g(\omega^M,h^\xi)$ in $\widetilde{A}(B_g)$ only
depends on $(h^{Tf},h^\xi)$.
\end{thm}
\begin{proof}
This is \cite[Theorem 2.13]{Ma1}.
\end{proof}

\subsection{Equivariant Bott-Chern singular currents}
The Bott-Chern singular current was defined by J.-M. Bismut, H.
Gillet and C. Soul\'{e} in \cite{BGS3} in order to generalize the
usual Bott-Chern secondary characteristic class to the case where
one considers the resolutions of hermitian vector bundles
associated to the closed immersions of complex manifolds. In
\cite{Bi}, J.-M. Bismut generalized this topic to the equivariant
case. We shall recall Bismut's construction of the equivariant
Bott-Chern singular current in this subsection. Similar to the
equivariant analytic torsion form, the construction itself is not
very important for understanding our later arguments, we just
recall it for the convenience of the reader. Bismut's construction
was realized via some current valued zeta function which involves
the supertraces of Quillen's superconnections. This is similar to
the non-equivariant case.

As before, let $g$ be the automorphism corresponding to an element
in a compact Lie group $G$. Let $i: Y\rightarrow X$ be an
equivariant closed immersion of $G$-equivariant K\"{a}hler
manifolds, and let $\overline{\eta}$ be an equivariant hermitian
vector bundle on $Y$. Assume that $\overline{\xi}.$ is a complex
(of homological type) of equivariant hermitian vector bundles on
$X$ which provides a resolution of $i_*\overline{\eta}$. We denote
the differential of the complex $\xi.$ by $v$. Note that $\xi.$ is
acyclic outside $Y$ and the homology sheaves of its restriction to
$Y$ are locally free and hence they are all vector bundles. We
write $H_n=\mathcal{H}_n(\xi.\mid_Y)$ and define a $\Z$-graded
bundle $H=\bigoplus_nH_n$. For each $y\in Y$ and $u\in TX_y$, we
denote by $\partial_uv(y)$ the derivative of $v$ at $y$ in the
direction $u$ in any given holomorphic trivialization of $\xi.$
near $y$. Then the map $\partial_uv(y)$ acts on $H_y$ as a chain
map, and this action only depends on the image $z$ of $u$ in $N_y$
where $N$ stands for the normal bundle of $i(Y)$ in $X$. So we get
a chain complex of holomorphic vector bundles $(H,\partial_zv)$.

Let $\pi$ be the projection from the normal bundle $N$ to $Y$,
then we have a canonical identification of $\Z$-graded chain
complexes
\begin{displaymath}
(\pi^*H,\partial_zv)\cong(\pi^*(\wedge N^\vee\otimes\eta),i_z).
\end{displaymath}
For this, one can see \cite{Bi2}, Section I.b. Moreover, such an
identification is an identification of $G$-bundles which induces a
family of canonical isomorphisms $\gamma_n: H_n\cong\wedge^n
N^\vee\otimes\eta$. Another way to describe these canonical
isomorphisms $\gamma_n$ is applying \cite{GBI}, Exp. VII, Lemma
2.4 and Proposition 2.5. These two constructions coincide because
they are both locally, on a suitable open covering $\{U_j\}_{j\in
J}$, determined by any complex morphism over the identity map of
$\eta\mid_{U_j}$ from $(\xi.\mid_{U_j},v)$ to the minimal
resolution of $\eta\mid_{U_j}$ (e.g. the Koszul resolution). The
advantage of using the construction given in \cite{GBI} is that it
remains valid for arithmetic varieties over any base instead of
the complex numbers. Later in \cite{Bi}, for the use of
normalization, J.-M. Bismut considered the automorphism of
$N^\vee$ defined by multiplying a constant $-\sqrt{-1}$, it
induces an isomorphism of chain complexes
\begin{displaymath}
(\pi^*(\wedge N^\vee\otimes\eta),i_z)\cong (\pi^*(\wedge
N^\vee\otimes\eta),\sqrt{-1}i_z)
\end{displaymath}
and hence
\begin{displaymath}
(\pi^*H,\partial_zv)\cong(\pi^*(\wedge
N^\vee\otimes\eta),\sqrt{-1}i_z).
\end{displaymath}
This identification induces a family of isomorphisms
$\widetilde{\gamma_n}: H_n\cong\wedge^n N^\vee\otimes\eta$. By
finite dimensional Hodge theory, for each $y\in Y$, there is a
canonical isomorphism
\begin{displaymath}
H_y\cong\{f\in \xi._y\mid vf=0, v^*f=0\}
\end{displaymath}
where $v^*$ is the dual of $v$ with respect to the metrics on
$\xi.$. This means that $H$ can be regarded as a smooth
$\Z$-graded $G$-equivariant subbundle of $\xi$ so that it carries
an induced $G$-invariant metric. On the other hand, we endow
$\wedge N^\vee\otimes \eta$ with the metric induced from
$\overline{N}$ and $\overline{\eta}$. J.-M. Bismut introduced the
following definition.

\begin{defn}\label{209}
We say that the metrics on the complex of equivariant hermitian
vector bundles $\overline{\xi}.$ satisfy Bismut assumption (A) if
the identification $(\pi^*H,\partial_zv)\cong(\pi^*(\wedge
N^\vee\otimes\eta),\sqrt{-1}i_z)$ also identifies the metrics, it
is equivalent to the condition that the identification
$(\pi^*H,\partial_zv)\cong(\pi^*(\wedge N^\vee\otimes\eta),i_z)$
identifies the metrics.
\end{defn}

\begin{prop}\label{210}
There always exist $G$-invariant metrics on $\xi.$ which satisfy
Bismut assumption (A) with respect to the equivariant hermitian
vector bundles $\overline{N}$ and $\overline{\eta}$.
\end{prop}
\begin{proof}
This is \cite[Proposition 3.5]{Bi}.
\end{proof}

From now on we always suppose that the metrics on a resolution
satisfy Bismut assumption (A). Let $\nabla^{\xi}$ be the canonical
hermitian holomorphic connection on $\xi.$, then for each $u>0$,
we may define a $G$-invariant superconnection
\begin{displaymath}
C_u:=\nabla^\xi+\sqrt{u}(v+v^*)
\end{displaymath} on the
$\Z_2$-graded vector bundle $\xi$. Moreover, let $\Phi$ be the map
$\alpha\in \wedge(T_\R^*X_g)\rightarrow (2\pi i)^{-{\rm
deg}\alpha/2}\alpha\in \wedge(T_\R^*X_g)$ and denote
\begin{displaymath}
({\rm Td}_g^{-1})'(\overline{N}):=\frac{\partial}{\partial
b}\mid_{b=0}({\rm Td}_g(b\cdot {\rm
Id}-\frac{\Omega^{\overline{N}}}{2\pi i})^{-1}) \end{displaymath}
where $\Omega^{\overline{N}}$ is the curvature form associated to
$\overline{N}$. We recall as follows the construction of the
equivariant singular current given in \cite[Section VI]{Bi}.

\begin{lem}\label{211}
Let $N_H$ be the number operator on the complex $\xi.$ i.e. it
acts on $\xi_j$ as multiplication by $j$, then for $s\in \C$ and
$0< {\rm Re}(s)<\frac{1}{2}$, the current valued zeta function
\begin{displaymath}
Z_g(\overline{\xi}.)(s):=\frac{1}{\Gamma(s)}\int_0^\infty
u^{s-1}[\Phi{\rm Tr_s}(N_Hg{\rm exp}(-C_u^2))+({\rm
Td}_g^{-1})'(\overline{N}){\rm
ch}_g(\overline{\eta})\delta_{Y_g}]{\rm d}u \end{displaymath} is
well-defined on $X_g$ and it has a meromorphic continuation to the
complex plane which is holomorphic at $s=0$.
\end{lem}

\begin{defn}\label{212}
The equivariant singular Bott-Chern current on $X_g$ associated to
the resolution $\overline{\xi}.$ is defined as
\begin{displaymath}
T_g(\overline{\xi}.):=\frac{\partial}{\partial
s}\mid_{s=0}Z_g(\overline{\xi}.)(s). \end{displaymath}
\end{defn}

\begin{thm}\label{213}
The current $T_g(\overline{\xi}.)$ is a sum of $(p,p)$-currents
and it satisfies the differential equation
\begin{displaymath}
{\rm dd}^cT_g(\overline{\xi}.)={i_g}_*{\rm
ch}_g(\overline{\eta}){\rm
Td}_g^{-1}(\overline{N})-\sum_k(-1)^k{\rm ch}_g(\overline{\xi}_k).
\end{displaymath} Moreover, the wave front set of
$T_g(\overline{\xi}.)$ is contained in $N^\vee_{g,\R}$ where
$N^\vee_{g,\R}$ stands for the underlying real bundle of the dual
of $N_g$.
\end{thm}
\begin{proof}
This follows from \cite[Theorem 6.7, Remark 6.8]{Bi}.
\end{proof}

Finally, we recall a theorem concerning the relationship of
equivariant Bott-Chern singular currents involved in a double
complex. This theorem will be used to show that our definition of
a general embedding morphism in equivariant arithmetic
$G_0$-theory is reasonable.

\begin{thm}\label{214}
Let \begin{displaymath} \overline{\chi}:\quad 0\rightarrow
\overline{\eta}_n\rightarrow\cdots\rightarrow\overline{\eta}_1\rightarrow\overline{\eta}_0\rightarrow0
\end{displaymath} be an exact sequence of equivariant hermitian
vector bundles on $Y$. Assume that we have the following double
complex consisting of resolutions of $i_*\overline{\chi}$ such
that all rows are exact sequences. \begin{displaymath}
\xymatrix{0
\ar[r] & \overline{\xi}_{n,\cdot} \ar[r] \ar[d] & \cdots \ar[r] &
\overline{\xi}_{1,\cdot} \ar[r] \ar[d] &
\overline{\xi}_{0,\cdot} \ar[r] \ar[d] & 0 \\
0 \ar[r] & i_*\overline{\eta}_n \ar[r] & \cdots \ar[r] &
i_*\overline{\eta}_1 \ar[r] & i_*\overline{\eta}_0 \ar[r] & 0.}
\end{displaymath} For each $k$, we write $\overline{\varepsilon}_k$
for the exact sequence \begin{displaymath} 0\rightarrow
\overline{\xi}_{n,k}\rightarrow \cdots\rightarrow
\overline{\xi}_{1,k}\rightarrow \overline{\xi}_{0,k}\rightarrow 0.
\end{displaymath} Then we have the following equality in
$\widetilde{\mathcal{U}}(X_g):=\bigoplus_{p\geq0}(D^{p,p}(X_g)/({\rm
Im}\partial+{\rm Im}\overline{\partial}))$
\begin{displaymath}
\sum_{j=0}^n(-1)^jT_g(\overline{\xi}_{j,\cdot})={i_g}_*\frac{\widetilde{{\rm
ch}}_g(\overline{\chi})}{{\rm
Td}_g(\overline{N})}-\sum_k(-1)^k\widetilde{{\rm
ch}}_g(\overline{\varepsilon}_k). \end{displaymath} Here
$D^{p,p}(X_g)$ stands for the space of currents on $X_g$ of type
$(p,p)$.
\end{thm}
\begin{proof}
This is \cite[Theorem 3.14]{KR1}.
\end{proof}

\subsection{Bismut-Ma's immersion formula}
In this subsection, we shall recall Bismut-Ma's immersion formula
which reflects the behaviour of the equivariant analytic torsion
forms of a K\"{a}hler fibration under composition of an immersion
and a submersion. By translating to the equivariant arithmetic
$G_0$-theoretic language, such a formula can be used to measure,
in arithmetic $G_0$-theory, the difference between a push-forward
morphism and the composition formed as an embedding morphism
followed by a push-forward morphism. Although Bismut-Ma's
immersion formula plays a very important role in our arguments, we
shall not recall its proof since it is rather long and technical.

Let $i: Y\rightarrow X$ be an equivariant closed immersion of
$G$-equivariant K\"{a}hler manifolds. Let $S$ be a complex
manifold with the trivial $G$-action, and let $f: Y\rightarrow S$,
$l: X\rightarrow S$ be two equivariant holomorphic submersions
such that $f=l\circ i$. Assume that $\overline{\eta}$ is an
equivariant hermitian vector bundle on $Y$ and $\overline{\xi}.$
provides a resolution of $i_*\overline{\eta}$ on $X$ whose metrics
satisfy Bismut assumption (A). Let $\omega^Y$, $\omega^X$ be the
real, closed and $G$-invariant $(1,1)$-forms on $Y$, $X$ which
induce the K\"{a}hler fibrations with respect to $f$ and $l$
respectively. We additionally assume that $\omega^Y$ is the
pull-back of $\omega^X$ so that the K\"{a}hler metric on $Y$ is
induced by the K\"{a}hler metric on $X$. As before, denote by $N$
the normal bundle of $i(Y)$ in $X$. Consider the following exact
sequence
\begin{displaymath}
\overline{\mathcal{N}}:\quad 0\rightarrow \overline{Tf}\rightarrow
\overline{Tl}\mid_Y\rightarrow \overline{N}\rightarrow 0
\end{displaymath} where $N$ is endowed with the quotient metric, we
shall write $\widetilde{{\rm
Td}}_g(\overline{Tf},\overline{Tl}\mid_Y)$ for $\widetilde{{\rm
Td}}_g(\overline{\mathcal{N}})$ the equivariant Bott-Chern
secondary characteristic class associated to
$\overline{\mathcal{N}}$. It satisfies the following differential
equation \begin{displaymath} {\rm dd}^c\widetilde{{\rm
Td}}_g(\overline{Tf},\overline{Tl}\mid_Y)={\rm
Td}_g(Tf,h^{Tf}){\rm Td}_g(\overline{N})-{\rm
Td}_g(Tl\mid_Y,h^{Tl}). \end{displaymath}

For simplicity, we shall suppose that in the resolution $\xi.$,
$\xi_j$ are all $l-$acyclic and moreover $\eta$ is $f-$acyclic. By
an easy argument of long exact sequence, we have the following
exact sequence \begin{displaymath} \Xi:\quad 0\rightarrow
l_*(\xi_m)\rightarrow l_*(\xi_{m-1})\rightarrow\ldots\rightarrow
l_*(\xi_0)\rightarrow f_*\eta\rightarrow 0. \end{displaymath} By
the semi-continuity theorem, all the elements in the exact
sequence above are vector bundles. In this case, we recall the
definition of the $L^2$-metrics on direct images precisely as
follows. We just take $f_*h^\eta$ as an example. Note that the
semi-continuity theorem implies that the natural map
\begin{displaymath}
(R^0f_*\eta)_s\rightarrow H^0(Y_s,\eta\mid_{Y_s})
\end{displaymath} is an isomorphism for every point $s\in S$ where
$Y_s$ stands for the fibre over $s$. We may endow
$H^0(Y_s,\eta\mid_{Y_s})$ with a $L^2$-metric given by the formula
\begin{displaymath}
<u,v>_{L^2}:=\frac{1}{(2\pi)^{d_s}}\int_{Y_s}h^\eta(u,v)\frac{{\omega^Y}^{d_s}}{d_s!}
\end{displaymath} where $d_s$ is the complex dimension of the fibre
$Y_s$. It can be shown that these metrics depend on $s$ in a
$C^\infty$ manner (cf. \cite[p.278]{BGV}) and hence define a
hermitian metric on $f_*\eta$. We shall denote it by $f_*h^\eta$.

In order to understand the statement of Bismut-Ma's immersion
formula, we still have to recall an important concept defined by
J.-M. Bismut, the equivariant $R$-genus. Let $W$ be a
$G$-equivariant complex manifold, and let $\overline{E}$ be an
equivariant hermitian vector bundle on $W$. For $\zeta\in S^1$ and
$s>1$ consider the zeta function
\begin{displaymath}
L(\zeta,s)=\sum_{k=1}^\infty\frac{\zeta^k}{k^s} \end{displaymath}
and its meromorphic continuation to the whole complex plane.
Define the formal power series in $x$
\begin{displaymath}
\widetilde{R}(\zeta,x):=\sum_{n=0}^\infty\big(\frac{\partial
L}{\partial
s}(\zeta,-n)+L(\zeta,-n)\sum_{j=1}^n\frac{1}{2j}\big)\frac{x^n}{n!}.
\end{displaymath}

\begin{defn}\label{215}
The Bismut equivariant $R$-genus of an equivariant hermitian
vector bundle $\overline{E}$ with
$\overline{E}\mid_{X_g}=\sum_\zeta\overline{E}_\zeta$ is defined
as \begin{displaymath} R_g(\overline{E}):=\sum_{\zeta\in
S^1}\big({\rm
Tr}\widetilde{R}(\zeta,-\frac{\Omega^{\overline{E}_\zeta}}{2\pi
i})-{\rm
Tr}\widetilde{R}(1/\zeta,\frac{\Omega^{\overline{E}_\zeta}}{2\pi
i})\big) \end{displaymath} where $\Omega^{\overline{E}_\zeta}$ is
the curvature form associated to $\overline{E}_\zeta$. Actually,
the class of $R_g(\overline{E})$ in $\widetilde{A}(X_g)$ is
independent of the metric and we just write $R_g(E)$ for it.
Furthermore, the class $R_g(\cdot)$ is additive.
\end{defn}

\begin{thm}\label{216}(Immersion formula)
Let notations and assumptions be as above. Then the equality
\begin{align*}
&\sum_{i=0}^m(-1)^iT_g(\omega^X,h^{\xi_i})-T_g(\omega^Y,h^\eta)+\widetilde{{\rm
ch}}_g(\Xi,h^{L^2})\\
=&\int_{X_g/S}{\rm
Td}_g(Tl,h^{Tl})T_g(\overline{\xi}.)+\int_{Y_g/S}\frac{\widetilde{{\rm
Td}}_g(\overline{Tf},\overline{Tl}\mid_Y)}{{\rm
Td}_g(\overline{N})}{\rm ch}_g(\overline{\eta})\\
&+\int_{X_g/S}{\rm
Td}_g(Tl)R_g(Tl)\sum_{i=0}^m(-1)^i{\rm
ch}_g(\xi_i)-\int_{Y_g/S}{\rm Td}_g(Tf)R_g(Tf){\rm ch}_g(\eta)
\end{align*}
holds in $\widetilde{A}(S)$.
\end{thm}
\begin{proof}
This is the combination of \cite[Theorem 0.1 and 0.2]{BM}, the
main theorems in that paper.
\end{proof}

\section{A vanishing theorem for equivariant closed immersions}

\subsection{The statement}
By a projective manifold we shall understand a compact complex
manifold which is projective algebraic, that means a projective
manifold is the complex analytic space $X(\C)$ associated to a
smooth projective variety $X$ over $\C$ (cf. \cite[Appendix
B]{Ha}). Let $\mu_n$ be the diagonalisable group variety over $\C$
associated to $\Z/{n\Z}$. We say $X$ is $\mu_n$-equivariant if it
admits a $\mu_n$-projective action, this means the associated
projective manifold $X(\C)$ admits an action by the group of
complex $n$-th roots of unity. Denote by $X_{\mu_n}$ the fixed
point subscheme of $X$, by GAGA principle, $X_{\mu_n}(\C)$ is
equal to $X(\C)_g$ where $g$ is the automorphism on $X(\C)$
corresponding to a fixed primitive $n$-th root of unity. If no
confusion arises, we shall not distinguish between $X$ and $X(\C)$
as well as $X_{\mu_n}$ and $X_g$. Since the classical arguments of
locally free resolutions may not be compatible with the
equivariant setting, we summarize some crucial facts we need as
follows.

(i). Every equivariant coherent sheaf on an equivariantly
projective scheme is an equivariant quotient of an equivariant
locally free coherent sheaf.

(ii). Every equivariant coherent sheaf on an equivariantly
projective scheme admits an equivariant locally free resolution.
It is finite if the equivariant scheme is regular.

(iii). An exact sequence of equivariant coherent sheaves on an
equivariantly projective scheme admits an exact sequence of
equivariant locally free resolutions.

(iv). Any two equivariant locally free resolutions of an
equivariant coherent sheaf on an equivariantly projective scheme
can be dominated by a third one.

Now let $i: Y \to X$ be a $\mu_n$-equivariant closed immersion of
projective manifolds with normal bundle $N$. Let $S$ be a
projective manifold with the trivial $\mu_n$-action and let $h:
X\to S$ be an equivariant holomorphic submersion whose restriction
$f: Y\to S$ is an equivariant holomorphic submersion. According to
our assumptions, we may define a K\"{a}hler fibration with respect
to $h$ by choosing a $\mu_n(\C)$-invariant K\"{a}hler form
$\omega^X$ on $X$. By restricting $\omega^X$ to $Y$ we obtain a
K\"{a}hler fibration with respect to $f$. The same thing goes to
$h_g: X_g\to S$ and $f_g: Y_g\to S$. Let $\overline{\eta}$ be an
equivariant hermitian holomorphic vector bundle on $Y$, assume
that $(\overline{\xi.},v)$ is a complex of equivariant hermitian
vector bundles on $X$ which provides a resolution of
$i_*\overline{\eta}$, whose metrics satisfy Bismut assumption (A).

Write $N_g$ for the $0$-degree part of $N\mid_{Y_g}$ which is
isomorphic to the normal bundle of $i_g(Y_g)$ in $X_g$ and denote
by $F$ the orthogonal complement of $N_g$. According to \cite[Exp.
VII, Lemma 2.4 and Proposition 2.5]{GBI} we know that there exists
a canonical isomorphism from the homology sheaf
$H(\xi.\mid_{X_g})$ to ${i_g}_*(\wedge^{\cdot}F^\vee\otimes
\eta\mid_{Y_g})$ which is equivariant. Then the restriction of
$(\xi.,v)$ to $X_g$ can always split into a series of short exact
sequences in the following way: \begin{displaymath} (*):\quad 0\to
{\rm Im}\to {\rm Ker}\to
{i_g}_*(\wedge^{\cdot}F^\vee\otimes\eta\mid_{Y_g})\to 0
\end{displaymath} and
\begin{displaymath}
(**):\quad 0\to {\rm Ker}\to \xi.\mid_{X_g} \to {\rm Im} \to 0.
\end{displaymath} Suppose that
$\wedge^{\cdot}F^\vee\otimes\eta\mid_{Y_g}$ and $\xi.\mid_{X_g}$
are all acyclic (higher direct images vanish). Then according to
an easy argument of long exact sequence, these short exact
sequences $(*)$ and $(**)$ induce a series of short exact
sequences of direct images:
\begin{displaymath}
\mathcal{H}(*):\quad 0\to R^0{h_g}_*({\rm Im})\to R^0{h_g}_*({\rm
Ker})\to R^0{f_g}_*(\wedge^{\cdot}F^\vee\otimes\eta\mid_{Y_g})\to
0 \end{displaymath} and \begin{displaymath} \mathcal{H}(**):\quad
0\to R^0{h_g}_*({\rm Ker})\to R^0{h_g}_*(\xi.\mid_{X_g})\to
R^0{h_g}_*({\rm Im})\to 0. \end{displaymath}

By semi-continuity theorem, all elements in the exact sequences
above are vector bundles. We endow $R^0{h_g}_*(\xi.\mid_{X_g})$
and $R^0{f_g}_*(\wedge^{\cdot}F^\vee\otimes\eta\mid_{Y_g})$ with
the $L^2$-metrics which are induced by the metrics on $\xi.$,
$\eta$ and $F$. Here the normal bundle $N$ admits the quotient
metric induced from the exact sequence \begin{displaymath} 0 \to
Tf \to Th\mid_{Y} \to N \to 0 \end{displaymath} and the bundle $F$
admits the metric induced by the metric on $N$. Moreover, we endow
$R^0{h_g}_*({\rm Im})$ and $R^0{h_g}_*({\rm Ker})$ with the
metrics induced by the $L^2$-metrics of
$R^0{h_g}_*(\xi.\mid_{X_g})$ so that $\mathcal{H}(*)$ and
$\mathcal{H}(**)$ become short exact sequences of equivariant
hermitian vector bundles. Denote by $\widetilde{{\rm
ch}}_g(\overline{\xi.},\overline{\eta})$ the alternating sum of
the equivariant secondary Bott-Chern characteristic classes of
$\mathcal{H}(*)$ and $\mathcal{H}(**)$ such that it satisfies the
following differential equation
\begin{align*}
{\rm dd}^c\widetilde{{\rm
ch}}_g(\overline{\xi.},\overline{\eta})=&\sum_{j}(-1)^{j}{\rm
ch}_g(R^0{f_g}_*(\wedge^{j}\overline{F}^\vee\otimes\overline{\eta}\mid_{Y_g}))\\
&\qquad\qquad\qquad\qquad\qquad-\sum_{j}(-1)^{j}{\rm
ch}_g(R^0{h_g}_*(\overline{\xi}_j\mid_{X_g})).
\end{align*}

Now the difference
\begin{align*}
\delta(i,\overline{\eta},\overline{\xi}.):=&\widetilde{{\rm
ch}}_g(\overline{\xi.},\overline{\eta})-\sum_k(-1)^kT_g(\omega^{Y_g},h^{\wedge^kF^\vee\otimes\eta\mid_{Y_g}})\\
&+\sum_k(-1)^kT_g(\omega^{X_g},h^{\xi_k\mid_{X_g}})-\int_{X_g/S}T_g(\overline{\xi.}){\rm
Td}(\overline{Th_g})\\
&-\int_{Y_g/S}{\rm Td}(Tf_g){\rm Td}_g^{-1}(F){\rm
ch}_g(\eta)R(N_g)\\
&-\int_{Y_g/S}{\rm ch}_g(\overline{\eta}){\rm
Td}_g^{-1}(\overline{N})\widetilde{{\rm
Td}}(\overline{Tf_g},\overline{Th_g}\mid_{Y_g})
\end{align*}
makes sense and it is an element in
$\bigoplus_{p\geq0}A^{p,p}(S)/{({\rm Im}\partial+{\rm
Im}\overline{\partial})}$. Here the symbols $T_g(\cdot)$ in the
summations stand for analytic torsion forms introduced in Section
2.1, the symbol $T_g(\overline{\xi.})$ in the integral is the
equivariant Bott-Chern singular current introduced in Section 2.2.

The vanishing theorem for equivariant closed immersions can be
formulated as the following.

\begin{thm}\label{301}
Let $i: Y\to X$ be an equivariant closed immersion of projective
manifolds, and let $S$ be a projective manifold with the trivial
$\mu_n$-action. Assume that we are given two equivariant
holomorphic submersions $f: Y\to S$ and $h: X\to S$ such that
$f=h\circ i$. Then $X$ admits an equivariant hermitian very ample
invertible sheaf $\overline{\mathcal{L}}$ relative to the morphism
$h$, and for any equivariant hermitian resolution $0\to
\overline{\xi}_m\to \cdots\overline{\xi}_1\to \overline{\xi}_0\to
i_*\overline{\eta}\to 0$ we have
\begin{displaymath}
\delta(i,\overline{\eta}\otimes i^*\overline{\mathcal{L}}^{\otimes
n},\overline{\xi.}\otimes \overline{\mathcal{L}}^{\otimes n})=0
\quad\quad \text{for} \quad n\gg0. \end{displaymath} Here the
metrics on the resolution are supposed to satisfy Bismut
assumption (A).
\end{thm}

\subsection{Deformation to the normal cone}
To prove the vanishing theorem for closed immersions, we use a
geometric construction called the deformation to the normal cone
which allows us to deform a resolution of hermitian vector bundle
associated to a closed immersion of projective manifolds to a
simpler one. The $\delta$-difference of this new simpler
resolution is much easier to compute.

Let $i: Y\hookrightarrow X$ be a closed immersion of projective
manifolds with normal bundle $N_{X/Y}$. For a vector bundle $E$ on
$X$ or $Y$, the notation $\mathbb{P}(E)$ will stand for the
projective space bundle ${\rm Proj}({\rm Sym}(E^\vee))$.

\begin{defn}\label{302}
The deformation to the normal cone $W(i)$ of the immersion $i$ is
the blowing up of $X\times \mathbb{P}^1$ along $Y\times
\{\infty\}$. We shall just write $W$ for $W(i)$ if there is no
confusion about the immersion.
\end{defn}

There are too many geometric objects and morphisms appearing in
the construction of the deformation to the normal cone, we have to
fix various notations in a clear way. We denote by $p_X$ (resp.
$p_Y$) the projection $X\times \mathbb{P}^1\rightarrow X$ (resp.
$Y\times \mathbb{P}^1\rightarrow Y$) and by $\pi$ the blow-down
map $W\rightarrow X\times\mathbb{P}^1$. We also denote by $q_X$
(resp. $q_Y$) the projection $X\times \mathbb{P}^1\rightarrow
\mathbb{P}^1$ (resp. $Y\times \mathbb{P}^1\rightarrow
\mathbb{P}^1$) and by $q_W$ the composition $q_X\circ\pi$. It is
well known that the map $q_W$ is flat and for $t\in \mathbb{P}^1$,
we have \begin{displaymath}
q_W^{-1}(t)\cong\left\{
\begin{array}{ll}
    X\times \{t\}, & \hbox{if $t\neq\infty$,} \\
    P\cup \widetilde{X}, & \hbox{if $t=\infty$,} \\
\end{array}%
\right. \end{displaymath} where $\widetilde{X}$ is isomorphic to
the blowing up of $X$ along $Y$ and $P$ is isomorphic to the
projective completion of $N_{X/Y}$ i.e. the projective space
bundle $\mathbb{P}(N_{X/Y}\oplus\mathcal{O}_Y)$. Denote the
canonical projection from $\mathbb{P}(N_{X/Y}\oplus\mathcal{O}_Y)$
to $Y$ by $\pi_P$, then the morphism $\mathcal{O}_Y\rightarrow
N_{X/Y}\oplus\mathcal{O}_Y$ induces a canonical section $i_\infty:
Y\hookrightarrow \mathbb{P}(N_{X/Y}\oplus\mathcal{O}_Y)$ which is
called the zero section embedding. Moreover, let $j: Y\times
\mathbb{P}^1\rightarrow W$ be the canonical closed immersion
induced by $i\times {\rm Id}$, then the component $\widetilde{X}$
doesn't meet $j(Y\times \mathbb{P}^1)$ and the intersection of
$j(Y\times \mathbb{P}^1)$ with $P$ is exactly the image of $Y$
under the section $i_\infty$.

The advantage of the construction of the deformation to the normal
cone is that we may control the rational equivalence class of the
fibres $q_W^{-1}(t)$. More precisely, in the language of line
bundles, we have the isomorphisms $\mathcal{O}(X)\cong
\mathcal{O}(P+\widetilde{X})\cong\mathcal{O}(P)\otimes\mathcal{O}(\widetilde{X})$
which is an immediate consequence of the isomorphism
$\mathcal{O}(0)\cong\mathcal{O}(\infty)$ on $\mathbb{P}^1$.

On $P=\mathbb{P}(N_{X/Y}\oplus\mathcal{O}_Y)$, there exists a
tautological exact sequence \begin{displaymath} 0\rightarrow
\mathcal{O}(-1)\rightarrow
\pi_P^*(N_{X/Y}\oplus\mathcal{O}_Y)\rightarrow Q\rightarrow 0
\end{displaymath} where $Q$ is the tautological quotient bundle.
This exact sequence and the inclusion
$\mathcal{O}_P\rightarrow\pi_P^*(N_{X/Y}\oplus\mathcal{O}_Y)$
induce a section $\sigma: \mathcal{O}_P\rightarrow Q$ which
vanishes along the zero section $i_\infty(Y)$. By duality we get a
morphism $Q^\vee\rightarrow \mathcal{O}_P$, and this morphism
induces the following exact sequence
\begin{displaymath}
0\rightarrow
\wedge^nQ^\vee\rightarrow\cdots\rightarrow\wedge^2Q^\vee\rightarrow
Q^\vee\rightarrow\mathcal{O}_P\rightarrow
{i_\infty}_*\mathcal{O}_Y\rightarrow 0 \end{displaymath} where $n$
is the rank of $Q$. Note that $i_\infty$ is a section of $\pi_P$
i.e. $\pi_P\circ i_\infty={\rm Id}$, the projection formula
implies the following definition.

\begin{defn}\label{303}
For any vector bundle $\eta$ on $Y$, the following complex of
vector bundles \begin{displaymath} 0\rightarrow
\wedge^nQ^\vee\otimes\pi_P^*\eta\rightarrow\cdots\rightarrow\wedge^2Q^\vee\otimes\pi_P^*\eta\rightarrow
Q^\vee\otimes\pi_P^*\eta\rightarrow\pi_P^*\eta\rightarrow 0
\end{displaymath} provides a resolution of ${i_\infty}_*\eta$ on
$P$. This complex is called the Koszul resolution of
${i_\infty}_*\eta$ and will be denoted by $\kappa(\eta,N_{X/Y})$.
If the normal bundle $N_{X/Y}$ admits some hermitian metric, then
the tautological exact sequence induces a hermitian metric on $Q$.
If, moreover, the bundle $\eta$ also admits a hermitian metric,
then the Koszul resolution is a complex of hermitian vector
bundles and will be denoted by
$\overline{\kappa}(\overline{\eta},\overline{N}_{X/Y})$.
\end{defn}

Now, assume that $X$ is a $\mu_n$-equivariant projective manifold
and $E$ is an equivariant locally free sheaf on $X$. Then
according to \cite[(1.4) and (1.5)]{Ko}, $\mathbb{P}(E)$ admits a
canonical $\mu_n$-equivariant structure such that the projection
map $\mathbb{P}(E)\rightarrow X$ is equivariant and the canonical
bundle $\mathcal{O}(1)$ admits an equivariant structure. Moreover,
let $Y\to X$ be an equivariant closed immersion of projective
manifolds, according to \cite[(1.6)]{Ko} the action of $\mu_n$ on
$X$ can be extended to the blowing up ${\rm Bl}_YX$ such that the
blow-down map is equivariant and the canonical bundle
$\mathcal{O}(1)$ admits an equivariant structure. So by endowing
$\mathbb{P}^1$ with the trivial $\mu_n$-action, the construction
of the deformation to the normal cone described above is
compatible with the equivariant setting.

For the use of our later arguments, the K\"{a}hler metric chosen
on $W$ should be well controlled on the fibres of the deformation.
For this purpose, it is necessary to introduce the following
definition.

\begin{defn}\label{304}(R\"{o}ssler)
A metric $h$ on $W$ is said to be normal to the deformation if

(a). it is invariant and K\"{a}hler;

(b). the restriction $h\mid_{{j_g}_*(Y_g\times\mathbb{P}^1)}$ is a
product $h'\times h''$, where $h'$ is a K\"{a}hler metric on $Y_g$
and $h''$ is a K\"{a}hler metric on $\mathbb{P}^1$;

(c). the intersections of $X\times\{0\}$ with $j_*(Y\times
\mathbb{P}^1)$ and of $P$ with $j_*(Y\times\mathbb{P}^1)$ are
orthogonal at the fixed points.
\end{defn}

\begin{lem}\label{305}
For any $\mu_n$-invariant K\"{a}hler metric $h^X$ on $X$ which
induces an invariant K\"{a}hler metric $h^Y$ on $Y$, there exists
a metric $h^W$ on $W$ which is normal to the deformation and the
restriction of $h^W$ to $X\cong X\times\{0\}$ (resp. $Y\cong
Y\times\{\infty\}$) is exactly $h^X$ (resp. $h^Y$). Moreover, we
may require that the hermitian normal bundles
$\overline{N}_{{Y\times\mathbb{P}^1}/{Y\times\{0\}}}$ and
$\overline{N}_{{Y\times\mathbb{P}^1}/{Y\times\{\infty\}}}$ are
both isometric to the trivial bundles with trivial metrics.
\end{lem}
\begin{proof}
The existence of the metric which is normal to the deformation is
the content of \cite[Lemma 6.13]{KR1} and \cite[Lemma 6.14]{Roe},
such a metric is constructed via the Grassmannian graph
construction. Roughly speaking, according to another description
of the deformation to the normal cone via the Grassmannian graph
construction, we have an embedding $W\to X\times
\mathbb{P}^r\times\mathbb{P}^1$ and the metric $h^W$ is the
$\mu_n$-average of the restriction of a product metric on $X\times
\mathbb{P}^r\times\mathbb{P}^1$ (cf. \cite[Lemma 6.14]{Roe}). When
we endow $X$ in the product with the metric $h^X$, the
requirements on restrictions are automatically satisfied since
$h^X$ is $\mu_n$-invariant. To fulfill the requirements on
hermitian normal bundles, we may just choose the Fubini-Study
metric on $\mathbb{P}^1$.
\end{proof}

We summarize some very important results about the application of
the deformation to the normal cone as follows. Their proofs can be
found in \cite[Section 2 and 6.2]{KR1}.

\begin{thm}\label{306}
Let $i: Y\to X$ be an equivariant closed immersion of equivariant
projective manifolds, and let $W=W(i)$ be the deformation to the
normal cone of $i$. Assume that $\overline{\eta}$ is an
equivariant hermitian vector bundle on $Y$. Then

(i). there exists an equivariant hermitian resolution of
$j_*p_Y^*(\overline{\eta})$ on $W$, whose metrics satisfy Bismut
assumption (A) and whose restriction to $\widetilde{X}$ is
equivariantly and orthogonally split;

(ii). the natural morphism from the deformation to the normal cone
$W(i_g)$ to the fixed point submanifold $W(i)_g$ is a closed
immersion, this closed immersion induces the closed immersions
$\mathbb{P}(N_{X_g/{Y_g}}\oplus
\mathcal{O}_{Y_g})\rightarrow\mathbb{P}(N_{X/Y}\oplus\mathcal{O}_Y)_g$
and $\widetilde{X_g}\rightarrow \widetilde{X}_g$;

(iii). the fixed point submanifold of
$\mathbb{P}(N_{X/Y}\oplus\mathcal{O}_Y)$ is
$\mathbb{P}(N_{X_g/{Y_g}}\oplus \mathcal{O}_{Y_g})\coprod_{\zeta
\neq1} \mathbb{P}((N_{X/Y})_\zeta)$;

(iv). the closed immersion $i_{\infty,g}$ factors through
$\mathbb{P}(N_{X_g/{Y_g}}\oplus \mathcal{O}_{Y_g})$ and the other
components $\mathbb{P}((N_{X/Y})_\zeta)$ don't meet $Y$. Hence the
complex $\kappa(\mathcal{O}_Y,N_{X/Y})_g$, obtained by taking the
$0$-degree part of the Koszul resolution, provides a resolution of
$\mathcal{O}_{Y_g}$ on $\mathbb{P}(N_{X/Y}\oplus\mathcal{O}_Y)_g$.
\end{thm}

\subsection{Proof of the vanishing theorem}
We shall first prove the first part of the vanishing theorem for
closed immersions i.e. the existence of an equivariant hermitian
very ample invertible sheaf on $X$ which is relative to the
morphism $h: X\to S$. Generally speaking, such an invertible sheaf
can be constructed rather easily since $X$ admits a
$\mu_n$-projective action and the $\mu_n$-action on $S$ is
supposed to be trivial, but for the whole proof of the vanishing
theorem we would like to construct a special one which is the
pull-back of some equivariant hermitian very ample invertible
sheaf on $W(i)$ under the identification $X\cong X\times\{0\}$.
Our starting point is the following.

\begin{defn}\label{307}
Let $M$ be a $\mu_n$-projective manifold, and let $\mathbb{P}_M^n$
be some relative projective space over $M$. A $\mu_n$-action on
$\mathbb{P}_M^n$ arising from some $\mu_n$-action on the free
sheaf $\mathcal{O}_M^{\oplus n+1}$ via the functorial properties
of the ${\rm Proj}$ symbol will be called a global $\mu_n$-action.
\end{defn}

The advantage of considering global $\mu_n$-action is that on a
projective space which admits a global $\mu_n$-action the twisted
line bundle $\mathcal{O}(1)$ is naturally $\mu_n$-equivariant.

\begin{lem}\label{308}
The morphism $h: X\to S$ factors though some relative projective
space $\mathbb{P}_S^r$ which admits a global $\mu_n$-action.
\end{lem}
\begin{proof}
By assumption, $X$ admits a $\mu_n$-projective action. Then
\cite[Lemma 2.4 and 2.5]{KR1} imply that there exists an
equivariant closed immersion from $X$ to some projective space
$\mathbb{P}^r$ endowed with a global action. By using the
universal property of fibre product, we obtain a morphism from $X$
to $\mathbb{P}_S^r=S\times \mathbb{P}^r$ which is equivariant.
Moreover, this morphism is clearly a closed immersion. Since the
action on $S$ is trivial, the induced action on the fibre product
$S\times \mathbb{P}^r$ is still global. So we are done.
\end{proof}

\begin{lem}\label{309}
Let $l: W(i)\to S$ be the composition $h\circ p_X\circ \pi$. Then
$W(i)$ admits an equivariant very ample invertible sheaf
$\mathcal{L}$ which is relative to $l$.
\end{lem}
\begin{proof}
By Lemma~\ref{308}, $h: X\to S$ factors through some relative
projective space $\mathbb{P}_S^r$ which admits a global
$\mu_n$-action. So $X$ admits an equivariant very ample invertible
sheaf relative to $h$. Since the $\mu_n$-action on $S$ is supposed
to be trivial, $\mathbb{P}_X^1=X\times\mathbb{P}^1\cong
X\times_S\mathbb{P}_S^1$ also admits an equivariant very ample
invertible sheaf relative to the morphism $h\circ p_X$ which is
denoted by $\mathcal{G}$. Moreover, by construction, $W(i)$ admits
a very ample invertible sheaf $\mathcal{O}_W(1)\otimes
\pi^*\mathcal{G}^{\otimes b}$ for some $b\geq0$ which is relative
to the blow-down map $\pi$ (cf. \cite[II. Proposition 7.10]{Ha}).
Assume that $\mathbb{P}_X^1\times_S\mathbb{P}_S^m$ is the relative
projective space associated to $\mathcal{O}_W(1)\otimes
\pi^*\mathcal{G}^{\otimes b}$, and that $\mathbb{P}_S^n$ is the
relative projective space associated to $\mathcal{G}$. Then the
very ample invertible sheave on
$\mathbb{P}_X^1\times_S\mathbb{P}_S^m$ with respect to the
embedding
\begin{displaymath}
\mathbb{P}_X^1\times_S\mathbb{P}_S^m\hookrightarrow
\mathbb{P}_S^n\times_S\mathbb{P}_S^m \end{displaymath} can be
written as $\mathcal{G}\boxtimes \mathcal{O}_{\mathbb{P}_S^m}(1)$
whose restriction to $W(i)$ is equal to $\mathcal{O}_W(1)\otimes
\pi^*\mathcal{G}^{\otimes b+1}$. Therefore,
$\mathcal{O}_W(1)\otimes \pi^*\mathcal{G}^{\otimes b+1}$ is a very
ample invertible sheaf on $W(i)$ relative to $l: W(i)\to S$, this
invertible sheaf is clearly equivariant.
\end{proof}

From now on, we shall fix the equivariant very ample invertible
sheaf $\mathcal{L}$ constructed in Lemma~\ref{309}. We also fix a
$\mu_n$-invariant hermitian metric on $\mathcal{L}$, note that
this metric always exists according to an argument of partition of
unity. When we deal with the tensor product of a coherent sheaf
$\mathcal{F}$ with some power $\mathcal{L}^{\otimes n}$, we just
write it as $\mathcal{F}(n)$ for simplicity. Before we give the
proof of the rest of the vanishing theorem, we shall recall the
concept of equivariant standard complex and some technical
results.

\begin{defn}\label{310}
Let $S$ be a projective manifold and let $\overline{\xi}.$ be a
bounded complex of hermitian vector bundles on $S$. We say
$\overline{\xi}.$ is a standard complex if the homology sheaves of
$\overline{\xi}.$ are all locally free and they are endowed with
some hermitian metrics. We shall write a standard complex as
$(\overline{\xi}.,h^H)$ to emphasize the choice of the metrics on
the homology sheaves.
\end{defn}

\begin{defn}\label{311}
Let $S$ be an equivariant projective manifold. An equivariant
standard complex on $S$ is a bounded complex of equivariant
hermitian vector bundles on $S$ whose restriction to $S_g$ is
standard and the metrics on the homology sheaves are
$g$-invariant. Again we shall write an equivariant standard
complex as $(\overline{\xi}.,h^H)$ to emphasize the choice of the
metrics on the homology sheaves.
\end{defn}

Due to \cite[Theorem 5.9]{T1}, to every equivariant standard
complex $(\overline{\xi}.,h^H)$ on an equivariant projective
manifold $S$, there is a unique axiomatical way to associate an
element $\widetilde{{\rm ch}}_g(\overline{\xi.},h^H)$ in
$\bigoplus_{p\geq0}A^{p,p}(S_g)/{({\rm Im}\partial+{\rm
Im}\overline{\partial})}$ which satisfies the differential
equation \begin{displaymath} {\rm dd}^c\widetilde{{\rm
ch}}_g(\overline{\xi.},h^H)=\sum_{j}(-1)^{j}{\rm
ch}_g(H_j(\overline{\xi.}\mid_{S_g}))-\sum_{j}(-1)^{j}{\rm
ch}_g(\overline{\xi}_j). \end{displaymath} Let $0\to
\overline{\xi.}'\to \overline{\xi.}\to \overline{\xi.}''\to 0$ be
a short exact sequence of equivariant standard complexes on $S$.
Then by restricting to the fixed point submanifold $S_g$, we get a
short exact sequence of standard complexes $0\to
\overline{\xi.}'\mid_{S_g}\to \overline{\xi.}\mid_{S_g}\to
\overline{\xi.}''\mid_{S_g}\to 0$. Hence we obtain a long exact
sequence of homology sheaves of these three standard complexes. We
shall make a stronger assumption. Suppose that for any $j\geq0$,
we have short exact sequence $0\to
H_j(\overline{\xi.}'\mid_{S_g})\to
H_j(\overline{\xi.}\mid_{S_g})\to
H_j(\overline{\xi.}''\mid_{S_g})\to 0$ which is denoted by
$\overline{\chi}_j$. Moreover, for any $j\geq0$, denote by
$\overline{\varepsilon}_j$ the short exact sequence $0\to
\overline{\xi}'_j\to \overline{\xi}_j\to \overline{\xi}''_j\to 0$.

\begin{lem}\label{312}
Let notations and assumptions be as above. The identity
\begin{displaymath}
\widetilde{{\rm ch}}_g(\overline{\xi.}',h^H)-\widetilde{{\rm
ch}}_g(\overline{\xi.},h^H)+\widetilde{{\rm
ch}}_g(\overline{\xi.}'',h^H)=\sum(-1)^j\widetilde{{\rm
ch}}_g(\overline{\chi}_j)-\sum(-1)^j\widetilde{{\rm
ch}}_g(\overline{\varepsilon}_j) \end{displaymath} holds in
$\bigoplus_{p\geq0}A^{p,p}(S_g)/{({\rm Im}\partial+{\rm
Im}\overline{\partial})}$.
\end{lem}
\begin{proof}
On $S_g$, every equivariant standard complex
$(\overline{\xi.},h^H)$ splits into a series of short exact
sequences of equivariant hermitian vector bundles in the following
way \begin{displaymath} 0\rightarrow \overline{{\rm
Im}}\rightarrow \overline{{\rm Ker}}\rightarrow
\overline{H.}\rightarrow 0 \end{displaymath} and
\begin{displaymath}
0\rightarrow \overline{{\rm Ker}}\rightarrow
\overline{\xi.}\mid_{S_g}\rightarrow \overline{{\rm
Im}}\rightarrow 0. \end{displaymath} According to the argument
given after \cite[Remark 5.10]{T1}, $\widetilde{{\rm
ch}}_g(\overline{\xi.},h^H)$ is equal to the alternating sum of
the equivariant Bott-Chern secondary characteristic classes of the
short exact sequences above. Now since we have supposed that $0\to
H_j(\overline{\xi.}'\mid_{S_g})\to
H_j(\overline{\xi.}\mid_{S_g})\to
H_j(\overline{\xi.}''\mid_{S_g})\to 0$ are all exact, by using
Snake lemma, we know that $0\to {\rm
Im}(\overline{\xi.}'\mid_{S_g})\to {\rm
Im}(\overline{\xi.}\mid_{S_g})\to {\rm
Im}(\overline{\xi.}''\mid_{S_g})\to 0$ and $0\to {\rm
Ker}(\overline{\xi.}'\mid_{S_g})\to {\rm
Ker}(\overline{\xi.}\mid_{S_g})\to {\rm
Ker}(\overline{\xi.}''\mid_{S_g})\to 0$ are also all exact
sequences. Then the identity in the statement of this lemma
immediately follows from the construction of $\widetilde{{\rm
ch}}_g(\overline{\xi.},h^H)$ and the additivity property of the
equivariant Bott-Chern secondary characteristic classes.
\end{proof}

\begin{cor}\label{313}
Let $0\to \overline{\xi.}^{(m)}\to \cdots\to
\overline{\xi.}^{(1)}\to \overline{\xi.}^{(0)}\to 0$ be an exact
sequence of equivariant standard complexes on $S$ such that for
every $j\geq0$, $0\to H_j(\overline{\xi.}^{(m)}\mid_{S_g})\to
\cdots\to H_j(\overline{\xi.}^{(1)}\mid_{S_g})\to
H_j(\overline{\xi.}^{(0)}\mid_{S_g})\to 0$ is exact. Then the
identity \begin{displaymath} \sum_{k=0}^m(-1)^k\widetilde{{\rm
ch}}_g(\overline{\xi.}^{(k)},h^H)=\sum(-1)^j\widetilde{{\rm
ch}}_g(\overline{\chi}_j)-\sum(-1)^j\widetilde{{\rm
ch}}_g(\overline{\varepsilon}_j) \end{displaymath} holds in
$\bigoplus_{p\geq0}A^{p,p}(S_g)/{({\rm Im}\partial+{\rm
Im}\overline{\partial})}$.
\end{cor}
\begin{proof}
We claim that for every $1\leq k\leq m$, the kernel of the complex
morphism $\overline{\xi.}^{(k)}\to \overline{\xi.}^{(k-1)}$ is
still an equivariant standard complex on $S$. It is clear that we
only need to prove this for $k=1$. Firstly, the kernel of
$\overline{\xi.}^{(1)}\to \overline{\xi.}^{(0)}$ is a complex of
equivariant hermitian vector bundles, let's denote it by
$\overline{K}$. By restricting to $S_g$ and using an argument of
long exact sequence, we know that the homology sheaves of
$\overline{K}\mid_{S_g}$ are all equivariant hermitian vector
bundles since for any $j\geq0$ the bundle morphism
$H_j(\overline{\xi.}^{(1)}\mid_{S_g})\to
H_j(\overline{\xi.}^{(0)}\mid_{S_g})$ is already surjective.
Therefore, the assumption of exactness on homologies implies that
we can split $0\to \overline{\xi.}^{(m)}\to \cdots\to
\overline{\xi.}^{(1)}\to \overline{\xi.}^{(0)}\to 0$ into a series
of short exact sequences of equivariant standard complexes, so the
identity in the statement of this corollary follows from
Lemma~\ref{312}.
\end{proof}

\begin{rem}\label{314}
A generalized version of Corollary~\ref{313}, in which the exact
sequence of (equivariant) standard complexes is replaced by an
(equivariant) double standard complex was obtained in Xiaonan Ma's
Ph.D thesis (cf. \cite{Ma2}) where more discussions concerning
spectral sequences were involved. Anyway, for arithmetical reason,
we only need these special versions as in Lemma~\ref{312} and
Corollary~\ref{313}.
\end{rem}

Now we turn back to our proof of the vanishing theorem. As before,
let $W=W(i)$ be the deformation to the normal cone associated to
an equivariant closed immersion of projective manifolds $i: Y\to
X$. For simplicity, denote by $P_g^0$ the projective space bundle
$\mathbb{P}(N_{X_g/{Y_g}}\oplus \mathcal{O}_{Y_g})$. Moreover,
given an invariant K\"{a}hler metric on $X$, we fix an invariant
K\"{a}hler metric on $W$ which is constructed in Lemma~\ref{305}.
In this situation, all normal bundles appearing in the
construction of the deformation to the normal cone will be endowed
with the quotient metrics. We recall the following lemma.

\begin{lem}\label{315}
Over $W(i_g)$, there are hermitian metrics on $\mathcal{O}(X_g)$,
$\mathcal{O}(P_g^0)$ and $\mathcal{O}(\widetilde{X_g})$ such that
the isometry
$\overline{\mathcal{O}}(X_g)\cong\overline{\mathcal{O}}(P_g^0)\otimes\overline{\mathcal{O}}(\widetilde{X_g})$
holds and such that the restriction of
$\overline{\mathcal{O}}(X_g)$ to $X_g$ yields the metric of
$N_{W(i_g)/{X_g}}$, the restriction of
$\overline{\mathcal{O}}(\widetilde{X_g})$ to $\widetilde{X_g}$
yields the metric of $N_{W(i_g)/{\widetilde{X_g}}}$ and the
restriction of $\overline{\mathcal{O}}(P_g^0)$ to $P_g^0$ induces
the metric of $N_{W(i_g)/{P_g^0}}$.
\end{lem}
\begin{proof}
This is \cite[Lemma 6.15]{KR1}.
\end{proof}

\begin{defn}\label{316}
Let $\overline{\eta}$ be an equivariant hermitian vector bundle on
$Y$, we say that a resolution
\begin{displaymath}
\overline{\Xi}:\quad 0\to \overline{\widetilde{\xi}}_m\to \cdots
\to \overline{\widetilde{\xi}}_0 \to j_*p_Y^*(\overline{\eta})\to
0 \end{displaymath} satisfies the condition (T) if

(i). the metrics on $\widetilde{\xi.}$ satisfy Bismut assumption
(A);

(ii). the restriction of $\overline{\Xi}$ to $\widetilde{X}$ is an
equivariantly and orthogonally split exact sequence;

(iii). the restrictions of $\overline{\Xi}_\nabla$ to $W(i_g)$,
$X_g$, $P_g^0$, $\widetilde{X_g}$ and $P_g^0\cap \widetilde{X_g}$
are complexes with $l$-acyclic elements and $l$-acyclic
homologies, here $\overline{\Xi}_\nabla$ is the complex of
hermitian vector bundles obtained by omitting the last term
$j_*p_Y^*(\overline{\eta})$ in $\overline{\Xi}$;

(iv). the tensor products
$\overline{\Xi}_\nabla\mid_{W(i_g)}\otimes
\overline{\mathcal{O}}(-X_g)$,
$\overline{\Xi}_\nabla\mid_{W(i_g)}\otimes
\overline{\mathcal{O}}(-P_g^0)$ and
$\overline{\Xi}_\nabla\mid_{W(i_g)}\otimes
\overline{\mathcal{O}}(-\widetilde{X_g})$ are complexes with
$l$-acyclic elements and $l$-acyclic homologies.
\end{defn}

From Theorem~\ref{306} (i), we already know that there always
exists a resolution of $j_*p_Y^*(\overline{\eta})$ which satisfies
the conditions (i) and (ii) in Definition~\ref{316}. Let
$\overline{\Xi}$ be such a resolution, we have the following.

\begin{prop}\label{317}
For $n\gg 0$, $\overline{\Xi}(n)$ satisfies the condition (T).
\end{prop}
\begin{proof}
The reason is that $W(i_g)$, $X_g$, $P_g^0$, $\widetilde{X_g}$ and
$P_g^0\cap \widetilde{X_g}$ are all closed submanifolds of $W$.
\end{proof}

It is well known that both two squares in the following
deformation diagram \begin{displaymath}
\xymatrix{
Y\times\{0\} \ar[rr]^-{s_0} \ar[d]_{i} && Y\times\mathbb{P}^1 \ar[d]^{j} & Y\times\{\infty\} \ar[l]_-{s_\infty} \ar[d]^{i_\infty} \\
 X\times\{0\} \ar[rr] && W &
 \mathbb{P}(N_{X/Y}\oplus N_{\mathbb{P}^1/\infty}) \ar[l]}
\end{displaymath} are ${\rm Tor}$-independent. Moreover, according
to our choices of the K\"{a}hler metrics, we may identify
$Y\times\{0\}$ with $Y$, $X\times\{0\}$ with $X$,
$Y\times\{\infty\}$ with $Y$ and $\mathbb{P}(N_{X/Y}\oplus
N_{\mathbb{P}^1/\infty})$ with $P=\mathbb{P}(N_{X/Y}\oplus
\mathcal{O}_Y)$. So if $\overline{\Xi}$ is a resolution of
$j_*p_Y^*(\overline{\eta})$ on $W$, then the restriction of
$\overline{\Xi}$ to $X$ (resp. $P$) provides a resolution of
$i_*\overline{\eta}$ (resp. ${i_\infty}_*\overline{\eta}$). The
following theorem is the kernel of the whole proof of the
vanishing theorem.

\begin{thm}\label{318}(Deformation theorem)
Let $\overline{\Xi}$ be a resolution of
$j_*p_Y^*(\overline{\eta})$ on $W$ which satisfies the condition
(T), then we have
$\delta(\overline{\Xi}\mid_X)=\delta(\overline{\Xi}\mid_P)$.
\end{thm}
\begin{proof}
Consider the following tensor product of
$\overline{\Xi}_\nabla\mid_{W(i_g)}$ with the Koszul resolution
associated to the immersion $X_g\hookrightarrow W(i_g)$
\begin{displaymath}
0\to
\overline{\Xi}_\nabla\mid_{W(i_g)}\otimes\overline{\mathcal{O}}(-X_g)\to
\overline{\Xi}_\nabla\mid_{W(i_g)}\otimes\overline{\mathcal{O}}_{W(i_g)}\to
\overline{\Xi}_\nabla\mid_{W(i_g)}\otimes{i_{X_g}}_*\overline{\mathcal{O}}_{X_g}\to
0. \end{displaymath} We have to caution the reader that here the
tensor product is not the usual tensor product of two complexes,
precisely our resulting sequence is a double complex and we don't
take its total complex. Since we have assumed that
$\overline{\Xi}$ satisfies the condition (T), this tensor product
induces a short exact sequence of equivariant standard complexes
on $S$ by taking direct images. For $j\geq0$, its $j$-th row is
the following short exact sequence
\begin{displaymath}
\overline{\varepsilon}_j: 0\to
R^0{l_g^0}_*(\overline{\mathcal{O}}(-X_g)\otimes
\overline{\widetilde{\xi}}_j\mid_{W(i_g)})\to
R^0{l_g^0}_*(\overline{\widetilde{\xi}}_j\mid_{W(i_g)})\to
R^0{h_g}_*(\overline{\widetilde{\xi}}_j\mid_{X_g})\to 0
\end{displaymath} where $l_g^0$ is the composition of the inclusion
$W(i_g)\hookrightarrow W$ with the morphism $l$.

Note that the $j$-th homology of
$\overline{\Xi}_\nabla\mid_{W(i_g)}$ is equal to
${j_g}_*(\wedge^j\overline{\widetilde{F}}^\vee\otimes
p_{Y_g}^*\overline{\eta}\mid_{Y_g})\mid_{W(i_g)}$ where
$\widetilde{F}$ is the non-zero degree part of the normal bundle
associated to the immersion $j$. Actually $j_g$ factors through
$j^0_g: Y_g\times \mathbb{P}^1\hookrightarrow W(i_g)$, then the
$j$-th homology of $\overline{\Xi}_\nabla\mid_{W(i_g)}$ can be
rewritten as
${j^0_g}_*(\wedge^j\overline{\widetilde{F}}^\vee\otimes
p_{Y_g}^*\overline{\eta}\mid_{Y_g})$. Write
$Y_{g,0}:=Y_g\times\{0\}$ for simplicity. Using the fact that
${j^0_g}^*\mathcal{O}(-X_g)$ is isomorphic to
$\mathcal{O}(-Y_{g,0})$, we deduce from the short exact sequence
\begin{align*}
0\to
{j^0_g}_*(\overline{\mathcal{O}}(-Y_{g,0})\otimes\wedge^j\overline{\widetilde{F}}^\vee\otimes
p_{Y_g}^*\overline{\eta}\mid_{Y_g})\to
{j^0_g}_*(\overline{\mathcal{O}}_{Y_g\times
\mathbb{P}^1}\otimes\wedge^j\overline{\widetilde{F}}^\vee\otimes
p_{Y_g}^*\overline{\eta}\mid_{Y_g})\\
\to
{j^0_g}_*({i_{Y_g}}_*\overline{\mathcal{O}}_{Y_g}\otimes\wedge^j\overline{\widetilde{F}}^\vee\otimes
p_{Y_g}^*\overline{\eta}\mid_{Y_g})\to 0
\end{align*} that the
$j$-th homologies of the induced short exact sequence of
equivariant standard complexes form a short exact sequence
\begin{align*}
\overline{\chi}_j: 0\to
R^0{u_g}_*(\overline{\mathcal{O}}(-Y_{g,0})\otimes\wedge^j\overline{\widetilde{F}}^\vee\otimes
p_{Y_g}^*\overline{\eta}\mid_{Y_g})\to
R^0{u_g}_*(\wedge^j\overline{\widetilde{F}}^\vee\otimes
p_{Y_g}^*\overline{\eta}\mid_{Y_g})\\
\to R^0{f_g}_*(\wedge^j\overline{F}^\vee\otimes
\overline{\eta}\mid_{Y_g})\to 0
\end{align*}
where $u_g$ is the composition of the inclusion
$Y_g\times\mathbb{P}^1\hookrightarrow W(i_g)$ with the morphism
$l_g^0$.

The main idea of this proof is that the equivariant Bott-Chern
secondary characteristic class of the quotient term of the induced
short exact sequence of equivariant standard complexes is nothing
but $\widetilde{{\rm ch}}_g(\overline{\Xi}_\nabla\mid_X,h^H)$
which appears in the expression of
$\delta(\overline{\Xi}\mid_{X})$ and the equivariant secondary
characteristic classes of $\overline{\chi}_j,
\overline{\varepsilon}_j$ can be computed by Bismut-Ma's immersion
formula.

Precisely, denote by $g_{X_g}$ the Euler-Green current associated
to $X_g$ which was constructed by Bismut, Gillet and Soul\'{e} in
\cite[Section 3. (f)]{BGS4}, it satisfies the differential
equation ${\rm
dd}^cg_{X_g}=\delta_{X_g}-c_1(\overline{\mathcal{O}}(X_g))$. We
write ${\rm Td}(\overline{X_g})$ for ${\rm
Td}(\overline{\mathcal{O}}(X_g))$, \cite[Theorem 3.17]{BGS4}
implies that ${\rm Td}^{-1}(\overline{X_g})g_{X_g}$ is equal to
the singular Bott-Chern current of the Koszul resolution
associated to $X_g\hookrightarrow W(i_g)$ modulo ${\rm
Im}\partial+{\rm Im}\overline{\partial}$. Moreover, write
$\overline{\xi.}$ for the restriction
$\overline{\Xi}_\nabla\mid_X$. Then for any $j\geq0$, we compute
\begin{align*}
\widetilde{{\rm
ch}}_g(\overline{\varepsilon}_j)=&T_g(\omega^{X_g},h^{\xi_j\mid_{X_g}})-T_g(\omega^{W(i_g)},h^{\widetilde{\xi}_j\mid_{W(i_g)}})\\
&+T_g(\omega^{W(i_g)},h^{\mathcal{O}(-X_g)\otimes
\widetilde{\xi}_j\mid_{W(i_g)}})\\
&+\int_{W(i_g)/S}{\rm ch}_g(\overline{\widetilde{\xi}}_j){\rm
Td}(\overline{Tl_g^0}){\rm
Td}^{-1}(\overline{X_g})g_{X_g}\\
&+\int_{X_g/S}{\rm ch_g}(\overline{\xi}_j){\rm
Td}^{-1}(\overline{N}_{W(i_g)/{X_g}})\widetilde{{\rm
Td}}(\overline{Th_g},\overline{Tl_g^0}\mid_{X_g})\\
&+\int_{X_g/S}{\rm ch}_g(\xi_j)R(N_{W(i_g)/{X_g}}){\rm Td}(Th_g).
\end{align*}
Here, one should note that to simplify the last two terms in the
right-hand side of Bismut-Ma's immersion formula, we have used an
Atiyah-Segal-Singer type formula for immersion
\begin{displaymath}
{i_g}_*({\rm Td}_g^{-1}(N){\rm ch}_g(x))={\rm ch}_g(i_*(x)).
\end{displaymath} This formula is the content of \cite[Theorem
6.16]{KR1}. Similarly, for any $j\geq0$, we compute
\begin{align*}
\widetilde{{\rm
ch}}_g(\overline{\chi}_j)=&T_g(\omega^{Y_g},h^{\wedge^jF^\vee\otimes
\eta\mid_{Y_g}})-T_g(\omega^{Y_g\times
\mathbb{P}^1},h^{\wedge^j\widetilde{F}^\vee\otimes
p_{Y_g}^*\eta\mid_{Y_g}})\\
&+T_g(\omega^{Y_g\times
\mathbb{P}^1},h^{\mathcal{O}(-Y_{g,0})\otimes
\wedge^j\widetilde{F}^\vee\otimes
p_{Y_g}^*\eta\mid_{Y_g}})\\
&+\int_{Y_g\times\mathbb{P}^1/S}{\rm
ch}_g(\wedge^j\overline{\widetilde{F}}^\vee\otimes
p_{Y_g}^*\overline{\eta}\mid_{Y_g}){\rm Td}(\overline{Tu_g}){\rm
Td}^{-1}(\overline{Y_{g,0}})g_{Y_{g,0}}\\
&+\int_{Y_g/S}{\rm ch_g}(\wedge^j\overline{F}^\vee\otimes
\overline{\eta}\mid_{Y_g}){\rm
Td}^{-1}(\overline{N}_{Y_g\times\mathbb{P}^1/{Y_{g,0}}})\widetilde{{\rm
Td}}(\overline{Tf_g},\overline{Tu_g}\mid_{Y_{g,0}})\\
&+\int_{Y_g/S}{\rm ch}_g(\wedge^jF^\vee\otimes
\eta\mid_{Y_g})R(N_{Y_g\times\mathbb{P}^1/{Y_{g,0}}}){\rm
Td}(Tf_g).
\end{align*}

Denote by $\overline{\Omega}(W(i_g))$ (resp.
$\overline{\Omega}(-X_g)$) the middle (resp. sub) term of the
induced short exact sequence of equivariant standard complexes.
According to Lemma~\ref{312}, we have
\begin{align}
&\widetilde{{\rm
ch}}_g(\overline{\Xi}_\nabla\mid_X,h^H)-\widetilde{{\rm
ch}}_g(\overline{\Omega}(W(i_g)),h^H)+\widetilde{{\rm
ch}}_g(\overline{\Omega}(-X_g),h^H)\notag\\
=&\sum(-1)^j\widetilde{{\rm
ch}}_g(\overline{\chi}_j)-\sum(-1)^j\widetilde{{\rm
ch}}_g(\overline{\varepsilon}_j)\notag\\
=&\sum(-1)^jT_g(\omega^{Y_g},h^{\wedge^jF^\vee\otimes
\eta\mid_{Y_g}})-\sum(-1)^jT_g(\omega^{Y_g\times
\mathbb{P}^1},h^{\wedge^j\widetilde{F}^\vee\otimes
p_{Y_g}^*\eta\mid_{Y_g}})\notag\\
&+\sum(-1)^jT_g(\omega^{Y_g\times
\mathbb{P}^1},h^{\mathcal{O}(-Y_{g,0})\otimes
\wedge^j\widetilde{F}^\vee\otimes
p_{Y_g}^*\eta\mid_{Y_g}})\notag\\
&+\int_{Y_g\times\mathbb{P}^1/S}\sum(-1)^j{\rm
ch}_g(\wedge^j\overline{\widetilde{F}}^\vee\otimes
p_{Y_g}^*\overline{\eta}\mid_{Y_g}){\rm Td}(\overline{Tu_g}){\rm
Td}^{-1}(\overline{Y_{g,0}})g_{Y_{g,0}}\notag\\
&+\int_{Y_g/S}\sum(-1)^j{\rm
ch_g}(\wedge^j\overline{F}^\vee\otimes
\overline{\eta}\mid_{Y_g}){\rm
Td}^{-1}(\overline{N}_{Y_g\times\mathbb{P}^1/{Y_{g,0}}})\widetilde{{\rm
Td}}(\overline{Tf_g},\overline{Tu_g}\mid_{Y_{g,0}})\notag\\
&+\int_{Y_g/S}\sum(-1)^j{\rm ch}_g(\wedge^jF^\vee\otimes
\eta\mid_{Y_g})R(N_{Y_g\times\mathbb{P}^1/{Y_{g,0}}}){\rm
Td}(Tf_g)\notag\\
&-\sum(-1)^jT_g(\omega^{X_g},h^{\xi_j\mid_{X_g}})+\sum(-1)^jT_g(\omega^{W(i_g)},h^{\widetilde{\xi}_j\mid_{W(i_g)}})\notag\\
&-\sum(-1)^jT_g(\omega^{W(i_g)},h^{\mathcal{O}(-X_g)\otimes
\widetilde{\xi}_j\mid_{W(i_g)}})\notag\\
&-\int_{W(i_g)/S}\sum(-1)^j{\rm
ch}_g(\overline{\widetilde{\xi}}_j){\rm Td}(\overline{Tl_g^0}){\rm
Td}^{-1}(\overline{X_g})g_{X_g}\notag\\
&-\int_{X_g/S}\sum(-1)^j{\rm ch_g}(\overline{\xi}_j){\rm
Td}^{-1}(\overline{N}_{W(i_g)/{X_g}})\widetilde{{\rm
Td}}(\overline{Th_g},\overline{Tl_g^0}\mid_{X_g})\notag\\
&-\int_{X_g/S}\sum(-1)^j{\rm ch}_g(\xi_j)R(N_{W(i_g)/{X_g}}){\rm
Td}(Th_g).\label{Eq(1)}
\end{align}

Similarly, we consider the tensor products of
$\overline{\Xi}_\nabla\mid_{W(i_g)}$ with the following three
Koszul resolutions \begin{displaymath} 0\to
\overline{\mathcal{O}}(-P^0_g)\to
\overline{\mathcal{O}}_{W(i_g)}\to
{i_{P^0_g}}_*\overline{\mathcal{O}}_{P^0_g}\to 0,
\end{displaymath}
\begin{displaymath}
0\to \overline{\mathcal{O}}(-\widetilde{X_g})\to
\overline{\mathcal{O}}_{W(i_g)}\to
{i_{\widetilde{X_g}}}_*\overline{\mathcal{O}}_{\widetilde{X_g}}\to
0, \end{displaymath} and
\begin{align*}
0\to
\overline{\mathcal{O}}(-\widetilde{X_g})\otimes\overline{\mathcal{O}}(-P^0_g)\to
\overline{\mathcal{O}}(-\widetilde{X_g})\oplus&\overline{\mathcal{O}}(-P^0_g)\to
\overline{\mathcal{O}}_{W(i_g)}\\
&\to {i_{\widetilde{X_g}\cap
P^0_g}}_*\overline{\mathcal{O}}_{\widetilde{X_g}\cap P^0_g}\to
0.
\end{align*}
We shall still denote by $\overline{\chi.}$ (resp.
$\overline{\varepsilon.}$) the exact sequences consisting of
homologies (resp. elements) in the induced exact sequences of
equivariant standard complexes.

For the first one, denote by $g_{P^0_g}$ the Euler-Green current
associated to $P^0_g$ and write $\overline{\xi.}^\infty$ for the
restriction $\overline{\Xi}_\nabla\mid_P$. Moreover, denote by
$\overline{\Omega}(-P^0_g)$ the sub term of the induced short
exact sequence of equivariant standard complexes and denote by
$b_g$ the composition of the inclusion $P_g^0\hookrightarrow
W(i_g)$ with the morphism $l_g^0$. According to Lemma~\ref{312},
we have
\begin{align}
&\widetilde{{\rm
ch}}_g(\overline{\Xi}_\nabla\mid_{P^0_g},h^H)-\widetilde{{\rm
ch}}_g(\overline{\Omega}(W(i_g)),h^H)+\widetilde{{\rm
ch}}_g(\overline{\Omega}(-P^0_g),h^H)\notag\\
=&\sum(-1)^j\widetilde{{\rm
ch}}_g(\overline{\chi}_j)-\sum(-1)^j\widetilde{{\rm
ch}}_g(\overline{\varepsilon}_j)\notag\\
=&\sum(-1)^jT_g(\omega^{Y_g},h^{\wedge^jF_\infty^\vee\otimes
\eta\mid_{Y_g}})-\sum(-1)^jT_g(\omega^{Y_g\times
\mathbb{P}^1},h^{\wedge^j\widetilde{F}^\vee\otimes
p_{Y_g}^*\eta\mid_{Y_g}})\notag\\
&+\sum(-1)^jT_g(\omega^{Y_g\times
\mathbb{P}^1},h^{\mathcal{O}(-Y_{g,\infty})\otimes
\wedge^j\widetilde{F}^\vee\otimes
p_{Y_g}^*\eta\mid_{Y_g}})\notag\\
&+\int_{Y_g\times\mathbb{P}^1/S}\sum(-1)^j{\rm
ch}_g(\wedge^j\overline{\widetilde{F}}^\vee\otimes
p_{Y_g}^*\overline{\eta}\mid_{Y_g}){\rm Td}(\overline{Tu_g}){\rm
Td}^{-1}(\overline{Y_{g,\infty}})g_{Y_{g,\infty}}\notag\\
&+\int_{Y_g/S}\{\sum(-1)^j{\rm
ch_g}(\wedge^j\overline{F}_\infty^\vee\otimes
\overline{\eta}\mid_{Y_g}){\rm
Td}^{-1}(\overline{N}_{Y_g\times\mathbb{P}^1/{Y_{g,\infty}}})\notag\\
&\qquad\qquad\qquad\qquad\qquad\qquad\qquad\qquad\qquad\cdot\widetilde{{\rm
Td}}(\overline{Tf_g},\overline{Tu_g}\mid_{Y_{g,\infty}})\}\notag\\
&+\int_{Y_g/S}\sum(-1)^j{\rm ch}_g(\wedge^jF_\infty^\vee\otimes
\eta\mid_{Y_g})R(N_{Y_g\times\mathbb{P}^1/{Y_{g,\infty}}}){\rm
Td}(Tf_g)\notag\\
&-\sum(-1)^jT_g(\omega^{P^0_g},h^{\xi_j^\infty\mid_{P^0_g}})+\sum(-1)^jT_g(\omega^{W(i_g)},h^{\widetilde{\xi}_j\mid_{W(i_g)}})\notag\\
&-\sum(-1)^jT_g(\omega^{W(i_g)},h^{\mathcal{O}(-P^0_g)\otimes
\widetilde{\xi}_j\mid_{W(i_g)}})\notag\\
&-\int_{W(i_g)/S}\sum(-1)^j{\rm
ch}_g(\overline{\widetilde{\xi}}_j){\rm Td}(\overline{Tl_g^0}){\rm
Td}^{-1}(\overline{P^0_g})g_{P^0_g}\notag\\
&-\int_{P^0_g/S}\sum(-1)^j{\rm ch_g}(\overline{\xi}^\infty_j){\rm
Td}^{-1}(\overline{N}_{W(i_g)/{P^0_g}})\widetilde{{\rm
Td}}(\overline{Tb_g},\overline{Tl_g^0}\mid_{P^0_g})\notag\\
&-\int_{P^0_g/S}\sum(-1)^j{\rm
ch}_g(\xi^\infty_j)R(N_{W(i_g)/{P^0_g}}){\rm
Td}(Tb_g)\label{Eq(2)}
\end{align} where $\overline{F}_\infty$ is the non-zero degree
part of the hermitian normal bundle $\overline{N}_\infty$
associated to $i_\infty$.

For the second one, denote by $g_{\widetilde{X_g}}$ the
Euler-Green current associated to $\widetilde{X_g}$ and denote by
$\overline{\Omega}(-\widetilde{X_g})$ the sub term of the induced
short exact sequence of equivariant standard complexes. Since the
restriction of $\overline{\Xi}$ to the component $\widetilde{X}$
is equivariantly and orthogonally split, we know that
$\widetilde{{\rm ch}}_g(\overline{\Xi}\mid_{\widetilde{X_g}},h^H)$
is equal to $0$ and the summation $\sum(-1)^j{\rm
ch}_g(\overline{\widetilde{\xi}}_j)$ vanishes on
$\widetilde{X_g}$. Using again Lemma~\ref{312}, we obtain
\begin{align}
&-\widetilde{{\rm
ch}}_g(\overline{\Omega}(W(i_g)),h^H)+\widetilde{{\rm
ch}}_g(\overline{\Omega}(-\widetilde{X_g}),h^H)\notag\\
=&\sum(-1)^j\widetilde{{\rm
ch}}_g(\overline{\chi}_j)-\sum(-1)^j\widetilde{{\rm
ch}}_g(\overline{\varepsilon}_j)\notag\\
=&-\sum(-1)^jT_g(\omega^{Y_g\times
\mathbb{P}^1},h^{\wedge^j\widetilde{F}^\vee\otimes
p_{Y_g}^*\eta\mid_{Y_g}})\notag\\
&+\sum(-1)^jT_g(\omega^{Y_g\times
\mathbb{P}^1},h^{{j^0_g}^*\mathcal{O}(-\widetilde{X_g})\otimes
\wedge^j\widetilde{F}^\vee\otimes
p_{Y_g}^*\eta\mid_{Y_g}})\notag\\
&-\int_{Y_g\times\mathbb{P}^1/S}\{\sum(-1)^j{\rm
ch}_g(\wedge^j\overline{\widetilde{F}}^\vee\otimes
p_{Y_g}^*\overline{\eta}\mid_{Y_g}){\rm
Td}(\overline{Tu_g})\notag\\
&\qquad\qquad\qquad\qquad\qquad\qquad\qquad\cdot\widetilde{{\rm
ch}}({j^0_g}^*\overline{\mathcal{O}}(-\widetilde{X_g}),\overline{O}_{Y_g\times\mathbb{P}^1})\}\notag\\
&+\sum(-1)^jT_g(\omega^{W(i_g)},h^{\widetilde{\xi}_j\mid_{W(i_g)}})\notag\\
&-\sum(-1)^jT_g(\omega^{W(i_g)},h^{\mathcal{O}(-\widetilde{X_g})\otimes
\widetilde{\xi}_j\mid_{W(i_g)}})\notag\\
&-\int_{W(i_g)/S}\sum(-1)^j{\rm
ch}_g(\overline{\widetilde{\xi}}_j){\rm Td}(\overline{Tl_g^0}){\rm
Td}^{-1}(\overline{\widetilde{X_g}})g_{\widetilde{X_g}}.\label{Eq(3)}
\end{align}
Here the element $\widetilde{{\rm
ch}}({j^0_g}^*\overline{\mathcal{O}}(-\widetilde{X_g}),\overline{O}_{Y_g\times\mathbb{P}^1})$
is the equivariant secondary characteristic class of the following
short exact sequence \begin{displaymath} 0\to 0\to
{j^0_g}^*\overline{\mathcal{O}}(-\widetilde{X_g})\to
\overline{O}_{Y_g\times\mathbb{P}^1}\to 0. \end{displaymath}

We now consider the last one. This is also a Koszul resolution
because $\widetilde{X_g}$ and $P_g^0$ intersect transversally. By
\cite[Theorem 3.20]{BGS4}, the Euler-Green current associated to
$\widetilde{X_g}\cap P^0_g$ is the current
$c_1(\overline{\mathcal{O}}(P^0_g))g_{\widetilde{X_g}}+\delta_{\widetilde{X_g}}g_{P^0_g}$.
Then, by using the isometry
$\overline{\mathcal{O}}(X_g)\cong\overline{\mathcal{O}}(P_g^0)\otimes\overline{\mathcal{O}}(\widetilde{X_g})$
and Corollary~\ref{313}, we get
\begin{align}
&-\widetilde{{\rm
ch}}_g(\overline{\Omega}(W(i_g)),h^H)+\widetilde{{\rm
ch}}_g(\overline{\Omega}(-\widetilde{X_g}),h^H)\notag\\
&\qquad\qquad\qquad\qquad+\widetilde{{\rm
ch}}_g(\overline{\Omega}(-P^0_g),h^H)-\widetilde{{\rm
ch}}_g(\overline{\Omega}(-X_g),h^H)\notag\\
=&\sum(-1)^j\widetilde{{\rm
ch}}_g(\overline{\chi}_j)-\sum(-1)^j\widetilde{{\rm
ch}}_g(\overline{\varepsilon}_j)\notag\\
=&-\sum(-1)^jT_g(\omega^{Y_g\times
\mathbb{P}^1},h^{\wedge^j\widetilde{F}^\vee\otimes
p_{Y_g}^*\eta\mid_{Y_g}})\notag\\
&+\sum(-1)^jT_g(\omega^{Y_g\times
\mathbb{P}^1},h^{{j^0_g}^*\mathcal{O}(-\widetilde{X_g})\otimes
\wedge^j\widetilde{F}^\vee\otimes
p_{Y_g}^*\eta\mid_{Y_g}})\notag\\
&+\sum(-1)^jT_g(\omega^{Y_g\times
\mathbb{P}^1},h^{\mathcal{O}(-Y_{g,\infty})\otimes
\wedge^j\widetilde{F}^\vee\otimes
p_{Y_g}^*\eta\mid_{Y_g}})\notag\\
&-\sum(-1)^jT_g(\omega^{Y_g\times
\mathbb{P}^1},h^{\mathcal{O}(-Y_{g,0})\otimes
\wedge^j\widetilde{F}^\vee\otimes
p_{Y_g}^*\eta\mid_{Y_g}})\notag\\
&-\int_{Y_g\times\mathbb{P}^1/S}\sum(-1)^j{\rm
ch}_g(\wedge^j\overline{\widetilde{F}}^\vee\otimes
p_{Y_g}^*\overline{\eta}\mid_{Y_g}){\rm
Td}(\overline{Tu_g})\widetilde{{\rm
ch}}(\overline{\Theta})\notag\\
&+\sum(-1)^jT_g(\omega^{W(i_g)},h^{\widetilde{\xi}_j\mid_{W(i_g)}})\notag\\
&-\sum(-1)^jT_g(\omega^{W(i_g)},h^{\mathcal{O}(-\widetilde{X_g})\otimes
\widetilde{\xi}_j\mid_{W(i_g)}})\notag\\
&-\sum(-1)^jT_g(\omega^{W(i_g)},h^{\mathcal{O}(-P^0_g)\otimes
\widetilde{\xi}_j\mid_{W(i_g)}})\notag\\
&+\sum(-1)^jT_g(\omega^{W(i_g)},h^{\mathcal{O}(-X_g)\otimes
\widetilde{\xi}_j\mid_{W(i_g)}})\notag\\
&-\int_{W(i_g)/S}\{\sum(-1)^j{\rm
ch}_g(\overline{\widetilde{\xi}}_j){\rm Td}(\overline{Tl_g^0}){\rm
Td}^{-1}(\overline{\widetilde{X_g}}){\rm
Td}^{-1}(\overline{P^0_g})\notag\\
&\qquad\qquad\qquad\qquad\qquad\qquad\quad\cdot[c_1(\overline{\mathcal{O}}(P^0_g))g_{\widetilde{X_g}}+\delta_{\widetilde{X_g}}g_{P^0_g}]\}.\notag\\\label{Eq(4)}
\end{align}
Here the element $\widetilde{{\rm ch}}(\overline{\Theta})$ is the
equivariant secondary characteristic class of the following short
exact sequence \begin{displaymath} \overline{\Theta}:\quad 0\to
\overline{\mathcal{O}}(-Y_{g,0})\to
{j^0_g}^*\overline{\mathcal{O}}(-\widetilde{X_g})\oplus\overline{\mathcal{O}}(-Y_{g,\infty})\to
\overline{O}_{Y_g\times\mathbb{P}^1}\to 0. \end{displaymath}

Since $s_0: Y\times\{0\} \to Y\times\mathbb{P}^1$ and $s_\infty:
Y\times\{\infty\}\to Y\times\mathbb{P}^1$ are sections of smooth
morphism, the normal sequences
\begin{displaymath}
0\to \overline{Tf_g}\to \overline{Tu_g}\mid_{Y_{g,0}}\to
\overline{N}_{Y_g\times\mathbb{P}^1/{Y_{g,0}}}\to 0
\end{displaymath} and \begin{displaymath}
0\to \overline{Tf_g}\to \overline{Tu_g}\mid_{Y_{g,\infty}}\to
\overline{N}_{Y_g\times\mathbb{P}^1/{Y_{g,\infty}}}\to 0
\end{displaymath} are orthogonally split so that $\widetilde{{\rm
Td}}(\overline{Tf_g},\overline{Tu_g}\mid_{Y_{g,0}})$ and
$\widetilde{{\rm
Td}}(\overline{Tf_g},\overline{Tu_g}\mid_{Y_{g,\infty}})$ are both
equal to $0$. Moreover, the normal bundles
$N_{Y_g\times\mathbb{P}^1/{Y_{g,0}}}$ and
$N_{Y_g\times\mathbb{P}^1/{Y_{g,\infty}}}$ are isomorphic to
trivial bundles so that $R(N_{Y_g\times\mathbb{P}^1/{Y_{g,0}}})$
and $R(N_{Y_g\times\mathbb{P}^1/{Y_{g,\infty}}})$ are both equal
to $0$. Furthermore, we may drop all the terms where an integral
is taken over $\widetilde{X_g}$ because $\sum(-1)^j{\rm
ch}_g(\overline{\widetilde{\xi}}_j)$ vanishes on
$\widetilde{X_g}$.

Now, we compute
(\ref{Eq(1)})$-$(\ref{Eq(2)})$-$(\ref{Eq(3)})$+$(\ref{Eq(4)})
which is
\begin{align*}
&\widetilde{{\rm
ch}}_g(\overline{\Xi}_\nabla\mid_X,h^H)-\widetilde{{\rm
ch}}_g(\overline{\Xi}_\nabla\mid_{P_g^0},h^H)+\sum(-1)^jT_g(\omega^{X_g},h^{\xi_j\mid_{X_g}})\\
&-\sum(-1)^jT_g(\omega^{P^0_g},h^{\xi_j^\infty\mid_{P^0_g}})-\sum(-1)^jT_g(\omega^{Y_g},h^{\wedge^jF^\vee\otimes
\eta\mid_{Y_g}})\\
&\qquad\qquad\qquad\qquad\qquad\qquad+\sum(-1)^jT_g(\omega^{Y_g},h^{\wedge^jF_\infty^\vee\otimes
\eta\mid_{Y_g}})\\
=&\int_{Y_g\times\mathbb{P}^1/S}\{\sum(-1)^j{\rm
ch}_g(\wedge^j\overline{\widetilde{F}}^\vee\otimes
p_{Y_g}^*\overline{\eta}\mid_{Y_g}){\rm
Td}(\overline{Tu_g})\cdot[{\rm
Td}^{-1}(\overline{Y_{g,0}})g_{Y_{g,0}}\\
&\qquad\qquad-{\rm
Td}^{-1}(\overline{Y_{g,\infty}})g_{Y_{g,\infty}}+\widetilde{{\rm
ch}}({j^0_g}^*\overline{\mathcal{O}}(-\widetilde{X_g}),\overline{O}_{Y_g\times\mathbb{P}^1})-\widetilde{{\rm
ch}}_g(\overline{\Theta})]\}\\
&-\int_{W(i_g)/S}\sum(-1)^j{\rm
ch}_g(\overline{\widetilde{\xi}}_j){\rm Td}(\overline{Tl_g^0}){\rm
Td}^{-1}(\overline{X_g})g_{X_g}\\
&-\int_{X_g/S}\sum(-1)^j{\rm ch_g}(\overline{\xi}_j){\rm
Td}^{-1}(\overline{N}_{W(i_g)/{X_g}})\widetilde{{\rm
Td}}(\overline{Th_g},\overline{Tl_g^0}\mid_{X_g})\\
&-\int_{X_g/S}\sum(-1)^j{\rm ch}_g(\xi_j)R(N_{W(i_g)/{X_g}}){\rm
Td}(Th_g)\\
&+\int_{W(i_g)/S}\sum(-1)^j{\rm
ch}_g(\overline{\widetilde{\xi}}_j){\rm Td}(\overline{Tl_g^0}){\rm
Td}^{-1}(\overline{P^0_g})g_{P^0_g}\\
&+\int_{P^0_g/S}\sum(-1)^j{\rm ch_g}(\overline{\xi}^\infty_j){\rm
Td}^{-1}(\overline{N}_{W(i_g)/{P^0_g}})\widetilde{{\rm
Td}}(\overline{Tb_g},\overline{Tl_g^0}\mid_{P^0_g})\\
&+\int_{P^0_g/S}\sum(-1)^j{\rm
ch}_g(\xi^\infty_j)R(N_{W(i_g)/{P^0_g}}){\rm
Td}(Tb_g)\\
&+\int_{W(i_g)/S}\sum(-1)^j{\rm
ch}_g(\overline{\widetilde{\xi}}_j){\rm Td}(\overline{Tl_g^0}){\rm
Td}^{-1}(\overline{\widetilde{X_g}})g_{\widetilde{X_g}}\\
&-\int_{W(i_g)/S}\sum(-1)^j{\rm
ch}_g(\overline{\widetilde{\xi}}_j){\rm Td}(\overline{Tl_g^0}){\rm
Td}^{-1}(\overline{\widetilde{X_g}}){\rm
Td}^{-1}(\overline{P^0_g})c_1(\overline{\mathcal{O}}(P^0_g))g_{\widetilde{X_g}}.
\end{align*}

Denote by $i_X$ (resp. $i_P$) the inclusion from $X$ to $W(i)$
(resp. $P$ to $W(i)$). We may use the Atiyah-Segal-Singer type
formula for immersions and the projection formula in cohomology to
compute
\begin{align*}
&{{i_X}_g}_*\big(\sum(-1)^j{\rm
ch}_g(\xi_j)R(N_{W(i_g)/{X_g}}){\rm
Td}(Th_g)\big)\\
=&{{i_X}_g}_*\big(R(N_{W(i_g)/{X_g}}){\rm Td}(Th_g){i_g}_*({\rm
Td}_g^{-1}(N_{X/Y}){\rm ch}_g(\eta))\big)\\
=&({i_X}_g\circ i_g)_*\big(R(N_{W(i_g)/{X_g}}){\rm Td}(Th_g){\rm
Td}_g^{-1}(N_{X/Y}){\rm ch}_g(\eta)\big).
\end{align*}
Note that the restriction of $N_{W(i_g)/{X_g}}$ to $Y_g$ is
trivial so that the last expression vanishes. An entirely
analogous reasoning implies that
\begin{displaymath}
{{i_P}_g}_*\big(\sum(-1)^j{\rm
ch}_g(\xi_j^\infty)R(N_{W(i_g)/{P_g^0}}){\rm Td}(Tb_g)\big)=0.
\end{displaymath}

Thus, we are left with the equality
\begin{align*}
&\widetilde{{\rm
ch}}_g(\overline{\Xi}_\nabla\mid_X,h^H)-\widetilde{{\rm
ch}}_g(\overline{\Xi}_\nabla\mid_{P_g^0},h^H)+\sum(-1)^jT_g(\omega^{X_g},h^{\xi_j\mid_{X_g}})\\
&-\sum(-1)^jT_g(\omega^{P^0_g},h^{\xi_j^\infty\mid_{P^0_g}})-\sum(-1)^jT_g(\omega^{Y_g},h^{\wedge^jF^\vee\otimes
\eta\mid_{Y_g}})\\
&\qquad\qquad\qquad\qquad\qquad\qquad+\sum(-1)^jT_g(\omega^{Y_g},h^{\wedge^jF_\infty^\vee\otimes
\eta\mid_{Y_g}})\\
=&\int_{Y_g\times\mathbb{P}^1/S}\{\sum(-1)^j{\rm
ch}_g(\wedge^j\overline{\widetilde{F}}^\vee\otimes
p_{Y_g}^*\overline{\eta}\mid_{Y_g}){\rm
Td}(\overline{Tu_g})\cdot[{\rm
Td}^{-1}(\overline{Y_{g,0}})g_{Y_{g,0}}\\
&\qquad\qquad\qquad-{\rm
Td}^{-1}(\overline{Y_{g,\infty}})g_{Y_{g,\infty}}+\widetilde{{\rm
ch}}({j^0_g}^*\overline{\mathcal{O}}(-\widetilde{X_g}),\overline{O_{Y_g\times\mathbb{P}^1}})-\widetilde{{\rm
ch}}_g(\overline{\Theta})]\}\\
&-\int_{W(i_g)/S}\{\sum(-1)^j{\rm
ch}_g(\overline{\widetilde{\xi}}_j){\rm
Td}(\overline{Tl_g^0})\cdot[{\rm
Td}^{-1}(\overline{X_g})g_{X_g}-{\rm
Td}^{-1}(\overline{P^0_g})g_{P^0_g}\\
&\qquad\qquad\qquad\quad-{\rm
Td}^{-1}(\overline{\widetilde{X_g}})g_{\widetilde{X_g}}+{\rm
Td}^{-1}(\overline{\widetilde{X_g}}){\rm
Td}^{-1}(\overline{P_g^0})c_1(\overline{\mathcal{O}}(P^0_g))g_{\widetilde{X_g}}]\}\\
&-\int_{X_g/S}\sum(-1)^j{\rm ch_g}(\overline{\xi}_j){\rm
Td}^{-1}(\overline{N}_{W(i_g)/{X_g}})\widetilde{{\rm
Td}}(\overline{Th_g},\overline{Tl_g^0}\mid_{X_g})\\
&+\int_{P^0_g/S}\sum(-1)^j{\rm ch_g}(\overline{\xi}^\infty_j){\rm
Td}^{-1}(\overline{N}_{W(i_g)/{P^0_g}})\widetilde{{\rm
Td}}(\overline{Tb_g},\overline{Tl_g^0}\mid_{P^0_g}).
\end{align*}

Using the differential equation which
$T_g(\overline{\widetilde{\xi.}})$ satisfies, we compute
\begin{align}
&-\int_{W(i_g)/S}\{\sum(-1)^j{\rm
ch}_g(\overline{\widetilde{\xi}}_j){\rm
Td}(\overline{Tl_g^0})\cdot[{\rm
Td}^{-1}(\overline{X_g})g_{X_g}-{\rm
Td}^{-1}(\overline{P^0_g})g_{P^0_g}\notag\\
&\qquad\qquad\qquad-{\rm
Td}^{-1}(\overline{\widetilde{X_g}})g_{\widetilde{X_g}}+{\rm
Td}^{-1}(\overline{\widetilde{X_g}}){\rm
Td}^{-1}(\overline{P_g^0})c_1(\overline{\mathcal{O}}(P^0_g))g_{\widetilde{X_g}}]\}\notag\\
=&\int_{W(i_g)/S}\{{\rm
Td}(\overline{Tl_g^0})T_g(\overline{\widetilde{\xi.}})\cdot[{\rm
Td}^{-1}(\overline{X_g})\delta_{X_g}-{\rm
Td}^{-1}(\overline{P^0_g})\delta_{P^0_g}\notag\\
&\qquad\qquad\qquad-{\rm
Td}^{-1}(\overline{\widetilde{X_g}})\delta_{\widetilde{X_g}}+{\rm
Td}^{-1}(\overline{\widetilde{X_g}}){\rm
Td}^{-1}(\overline{P_g^0})c_1(\overline{\mathcal{O}}(P^0_g))\delta_{\widetilde{X_g}}]\}\notag\\
&-\int_{W(i_g)/S}\{{\rm Td}(\overline{Tl_g^0}){\rm
ch}_g(p_Y^*\overline{\eta}){\rm
Td}_g^{-1}(\overline{N}_{W/{Y\times\mathbb{P}^1}})\delta_{Y_g\times\mathbb{P}^1}\cdot[{\rm Td}^{-1}(\overline{X_g})g_{X_g}\notag\\
&-{\rm
Td}^{-1}(\overline{P^0_g})g_{P^0_g}-{\rm
Td}^{-1}(\overline{\widetilde{X_g}})g_{\widetilde{X_g}}+{\rm
Td}^{-1}(\overline{\widetilde{X_g}}){\rm
Td}^{-1}(\overline{P_g^0})c_1(\overline{\mathcal{O}}(P^0_g))g_{\widetilde{X_g}}]\}.\notag\\
\label{eq5}
\end{align}
Here we have used the equation
\begin{align} {\rm
Td}^{-1}(\overline{X_g})c_1(\overline{\mathcal{O}}(X_g))-&{\rm
Td}^{-1}(\overline{P^0_g})c_1(\overline{\mathcal{O}}(P^0_g))-{\rm
Td}^{-1}(\overline{\widetilde{X_g}})c_1(\overline{\mathcal{O}}(\widetilde{X_g}))\notag\\
+&{\rm Td}^{-1}(\overline{\widetilde{X_g}}){\rm
Td}^{-1}(\overline{P_g^0})c_1(\overline{\mathcal{O}}(P^0_g))c_1(\overline{\mathcal{O}}(\widetilde{X_g}))=0\label{eq505}
\end{align}
which is \cite[(23)]{KR1}.

Again using the fact that $\overline{\widetilde{\xi.}}$ is
equivariantly and orthogonally split on $\widetilde{X}$, the first
integral in the right-hand side of (\ref{eq5}) is equal to
\begin{align*}
&\int_{X_g/S}{\rm
Td}(\overline{Tl_g^0})T_g(\overline{\xi.}){\rm
Td}^{-1}(\overline{N}_{W(i_g)/{X_g}})\\
&\qquad\qquad\qquad\qquad\qquad\qquad-\int_{P^0_g/S}{\rm
Td}(\overline{Tl_g^0})T_g(\overline{\xi.}^\infty){\rm
Td}^{-1}(\overline{N}_{W(i_g)/{P^0_g}}).
\end{align*}
According to the normal sequence $0\to \overline{Th_g}\to
\overline{Tl_g^0}\mid_{X_g}\to \overline{N}_{W(i_g)/{X_g}}\to 0$,
we may write \begin{displaymath} {\rm Td}(\overline{Tl_g^0})={\rm
Td}(\overline{Th_g}){\rm Td}(\overline{N}_{W(i_g)/{X_g}})-{\rm
dd}^c\widetilde{{\rm
Td}}(\overline{Th_g},\overline{Tl_g^0}\mid_{X_g}).
\end{displaymath} So we get
\begin{align*}
&\int_{X_g/S}{\rm Td}(\overline{Tl_g^0})T_g(\overline{\xi.}){\rm
Td}^{-1}(\overline{N}_{W(i_g)/{X_g}})\\
=&\int_{X_g/S}{\rm
Td}(\overline{Th_g})T_g(\overline{\xi.})\\
&-\int_{X_g/S}\widetilde{{\rm
Td}}(\overline{Th_g},\overline{Tl_g^0}\mid_{X_g})\delta_{Y_g}{\rm
ch}_g(\overline{\eta}){\rm Td}_g^{-1}(\overline{N}){\rm
Td}^{-1}(\overline{N}_{W(i_g)/{X_g}})\\
&+\int_{X_g/S}\sum(-1)^j{\rm ch_g}(\overline{\xi}_j){\rm
Td}^{-1}(\overline{N}_{W(i_g)/{X_g}})\widetilde{{\rm
Td}}(\overline{Th_g},\overline{Tl_g^0}\mid_{X_g}).
\end{align*}
Similarly we have
\begin{align*}
&\int_{P^0_g/S}{\rm
Td}(\overline{Tl_g^0})T_g(\overline{\xi.}^\infty){\rm
Td}^{-1}(\overline{N}_{W(i_g)/{P^0_g}})\\
=&\int_{P^0_g/S}{\rm
Td}(\overline{Tb_g})T_g(\overline{\xi.}^\infty)\\
&-\int_{P^0_g/S}\widetilde{{\rm
Td}}(\overline{Tb_g},\overline{Tl_g^0}\mid_{P^0_g})\delta_{Y_g}{\rm
ch}_g(\overline{\eta}){\rm Td}_g^{-1}(\overline{N}_\infty){\rm
Td}^{-1}(\overline{N}_{W(i_g)/{P^0_g}})\\
&+\int_{P^0_g/S}\sum(-1)^j{\rm ch_g}(\overline{\xi}^\infty_j){\rm
Td}^{-1}(\overline{N}_{W(i_g)/{P^0_g}})\widetilde{{\rm
Td}}(\overline{Tb_g},\overline{Tl_g^0}\mid_{P^0_g}).
\end{align*}
Note that the normal sequence of $\overline{Th_g}$
in $\overline{Tl_g^0}$ (resp. $\overline{Tb_g}$ in
$\overline{Tl_g^0}$) is orthogonally split on $Y_g\times\{0\}$
(resp. $Y_g\times\{\infty\}$), then $\widetilde{{\rm
Td}}(\overline{Th_g},\overline{Tl_g^0}\mid_{X_g})\delta_{Y_g}$ and
$\widetilde{{\rm
Td}}(\overline{Tb_g},\overline{Tl_g^0}\mid_{P^0_g})\delta_{Y_g}$
are both equal to $0$. Combining these computations above we get
\begin{align}
&\int_{X_g/S}{\rm Td}(\overline{Tl_g^0})T_g(\overline{\xi.}){\rm
Td}^{-1}(\overline{N}_{W(i_g)/{X_g}})\notag\\
&\qquad\qquad\qquad\qquad-\int_{P^0_g/S}{\rm
Td}(\overline{Tl_g^0})T_g(\overline{\xi.}^\infty){\rm
Td}^{-1}(\overline{N}_{W(i_g)/{P^0_g}})\notag\\
=&\int_{X_g/S}{\rm
Td}(\overline{Th_g})T_g(\overline{\xi.})\notag\\
&+\int_{X_g/S}\sum(-1)^j{\rm
ch_g}(\overline{\xi}_j){\rm
Td}^{-1}(\overline{N}_{W(i_g)/{X_g}})\widetilde{{\rm
Td}}(\overline{Th_g},\overline{Tl_g^0}\mid_{X_g})\notag\\
&-\int_{P^0_g/S}{\rm
Td}(\overline{Tb_g})T_g(\overline{\xi.}^\infty)\notag\\
&-\int_{P^0_g/S}\sum(-1)^j{\rm
ch_g}(\overline{\xi}^\infty_j){\rm
Td}^{-1}(\overline{N}_{W(i_g)/{P^0_g}})\widetilde{{\rm
Td}}(\overline{Tb_g},\overline{Tl_g^0}\mid_{P^0_g}).\notag\\
\label{eq6}
\end{align}

We now compute the second integral in the right-hand side of
(\ref{eq5}). According to the normal sequence
\begin{displaymath}
0\to \overline{Tu_g}\to
\overline{Tl_g^0}\mid_{Y_g\times\mathbb{P}^1}\to
\overline{N}_{W(i_g)/{Y_g\times\mathbb{P}^1}}\to 0,
\end{displaymath} we may write
\begin{displaymath}
{\rm Td}(\overline{Tl_g^0})={\rm Td}(\overline{Tu_g}){\rm
Td}(\overline{N}_{W(i_g)/{Y_g\times\mathbb{P}^1}})-{\rm
dd}^c\widetilde{{\rm
Td}}(\overline{Tu_g},\overline{Tl_g^0}\mid_{Y_g\times\mathbb{P}^1}).
\end{displaymath} Hence
\begin{align*}
&-\int_{W(i_g)/S}\{{\rm Td}(\overline{Tl_g^0}){\rm
ch}_g(p_Y^*\overline{\eta}){\rm
Td}_g^{-1}(\overline{N}_{W/{Y\times\mathbb{P}^1}})\delta_{Y_g\times\mathbb{P}^1}\cdot[{\rm Td}^{-1}(\overline{X_g})g_{X_g}\\
&-{\rm
Td}^{-1}(\overline{P^0_g})g_{P^0_g}-{\rm
Td}^{-1}(\overline{\widetilde{X_g}})g_{\widetilde{X_g}}+{\rm
Td}^{-1}(\overline{\widetilde{X_g}}){\rm
Td}^{-1}(\overline{P_g^0})c_1(\overline{\mathcal{O}}(P^0_g))g_{\widetilde{X_g}}]\}\\
=&-\int_{Y_g\times\mathbb{P}^1/S}\{{\rm Td}(\overline{Tu_g}){\rm
ch}_g(p_Y^*\overline{\eta}){\rm
Td}_g^{-1}(\overline{\widetilde{F}})\cdot {j^0_g}^*[{\rm Td}^{-1}(\overline{X_g})g_{X_g}\\
&-{\rm
Td}^{-1}(\overline{P^0_g})g_{P^0_g}-{\rm
Td}^{-1}(\overline{\widetilde{X_g}})g_{\widetilde{X_g}}+{\rm
Td}^{-1}(\overline{\widetilde{X_g}}){\rm
Td}^{-1}(\overline{P_g^0})c_1(\overline{\mathcal{O}}(P^0_g))g_{\widetilde{X_g}}]\}\\
&+\int_{Y_g\times\mathbb{P}^1/S}\{\widetilde{{\rm
Td}}(\overline{Tu_g},\overline{Tl_g^0}\mid_{Y_g\times\mathbb{P}^1}){\rm
ch}_g(p_Y^*\overline{\eta}){\rm
Td}_g^{-1}(\overline{N}_{W/{Y\times\mathbb{P}^1}})\\
&\cdot[{\rm
Td}^{-1}(\overline{X_g})(\delta_{X_g}-c_1(\overline{\mathcal{O}}(X_g)))-{\rm
Td}^{-1}(\overline{P^0_g})(\delta_{P^0_g}-c_1(\overline{\mathcal{O}}(P^0_g)))\\
&-{\rm
Td}^{-1}(\overline{\widetilde{X_g}})(\delta_{\widetilde{X_g}}-c_1(\overline{\mathcal{O}}(\widetilde{X_g})))\\
&+{\rm Td}^{-1}(\overline{\widetilde{X_g}}){\rm
Td}^{-1}(\overline{P_g^0})c_1(\overline{\mathcal{O}}(P^0_g))(\delta_{\widetilde{X_g}}-c_1(\overline{\mathcal{O}}(\widetilde{X_g})))]\}.
\end{align*}
By our choices of the metrics, we have ${\rm
Td}_g^{-1}(\overline{N}_{W/{Y\times\mathbb{P}^1}})\mid_{Y_{g,0}}={\rm
Td}_g^{-1}(\overline{N})$, ${\rm
Td}(\overline{X_g})\mid_{Y_{g,0}}=1$ and ${\rm
Td}_g^{-1}(\overline{N}_{W/{Y\times\mathbb{P}^1}})\mid_{Y_{g,\infty}}={\rm
Td}_g^{-1}(\overline{N}_\infty)$, ${\rm
Td}(\overline{P^0_g})\mid_{Y_{g,\infty}}=1$. Furthermore, by
replacing all tangent bundles by relative tangent bundles, on can
carry through the proof given in \cite[P. 378-379]{KR1} to show
that \begin{displaymath} \widetilde{{\rm
Td}}(\overline{Tu_g},\overline{Tl_g^0}\mid_{Y_g\times\mathbb{P}^1})\mid_{Y_{g,0}}=\widetilde{{\rm
Td}}(\overline{Tf_g},\overline{Th_g}\mid_{Y_g}) \end{displaymath}
and \begin{displaymath} \widetilde{{\rm
Td}}(\overline{Tu_g},\overline{Tl_g^0}\mid_{Y_g\times\mathbb{P}^1})\mid_{Y_{g,\infty}}=\widetilde{{\rm
Td}}(\overline{Tf_g},\overline{Tb_g}\mid_{Y_g}). \end{displaymath}
So combining with the equation (\ref{eq505}), we get
\begin{align}
&-\int_{W(i_g)/S}\{{\rm Td}(\overline{Tl_g^0}){\rm
ch}_g(p_Y^*\overline{\eta}){\rm
Td}_g^{-1}(\overline{N}_{W/{Y\times\mathbb{P}^1}})\delta_{Y_g\times\mathbb{P}^1}\cdot[{\rm Td}^{-1}(\overline{X_g})g_{X_g}\notag\\
&-{\rm
Td}^{-1}(\overline{P^0_g})g_{P^0_g}-{\rm
Td}^{-1}(\overline{\widetilde{X_g}})g_{\widetilde{X_g}}+{\rm
Td}^{-1}(\overline{\widetilde{X_g}}){\rm
Td}^{-1}(\overline{P_g^0})c_1(\overline{\mathcal{O}}(P^0_g))g_{\widetilde{X_g}}]\}\notag\\
=&-\int_{Y_g\times\mathbb{P}^1/S}\{{\rm Td}(\overline{Tu_g}){\rm
ch}_g(p_Y^*\overline{\eta}){\rm
Td}_g^{-1}(\overline{\widetilde{F}})\cdot{j^0_g}^*[{\rm Td}^{-1}(\overline{X_g})g_{X_g}\notag\\
&-{\rm
Td}^{-1}(\overline{P^0_g})g_{P^0_g}-{\rm
Td}^{-1}(\overline{\widetilde{X_g}})g_{\widetilde{X_g}}+{\rm
Td}^{-1}(\overline{\widetilde{X_g}}){\rm
Td}^{-1}(\overline{P_g^0})c_1(\overline{\mathcal{O}}(P^0_g))g_{\widetilde{X_g}}]\}\notag\\
&+\int_{Y_g/S}{\rm ch}_g(\overline{\eta}){\rm
Td}_g^{-1}(\overline{N})\widetilde{{\rm
Td}}(\overline{Tf_g},\overline{Th_g}\mid_{Y_g})\notag\\
&-\int_{Y_g/S}{\rm ch}_g(\overline{\eta}){\rm
Td}_g^{-1}(\overline{N}_\infty)\widetilde{{\rm
Td}}(\overline{Tf_g},\overline{Tb_g}\mid_{Y_g}).\notag\\
\label{eq7}
\end{align}

At last, using the fact that the intersections in the deformation
diagram are transversal and the fact that
$j_g^0(Y_g\times\mathbb{P}^1)$ has no intersection with
$\widetilde{X_g}$, we can compute
\begin{align}
&\int_{Y_g\times\mathbb{P}^1/S}\{\sum(-1)^j{\rm
ch}_g(\wedge^j\overline{\widetilde{F}}^\vee\otimes
p_{Y_g}^*\overline{\eta}\mid_{Y_g}){\rm
Td}(\overline{Tu_g})\cdot[{\rm Td}^{-1}(\overline{Y_{g,0}})g_{Y_{g,0}}\notag\\
&\qquad\qquad\qquad-{\rm
Td}^{-1}(\overline{Y_{g,\infty}})g_{Y_{g,\infty}}+\widetilde{{\rm
ch}}({j^0_g}^*\overline{\mathcal{O}}(-\widetilde{X_g}),\overline{O}_{Y_g\times\mathbb{P}^1})-\widetilde{{\rm
ch}}_g(\overline{\Theta})]\}\notag\\
=&\int_{Y_g\times\mathbb{P}^1/S}\{{\rm Td}(\overline{Tu_g}){\rm
ch}_g(p_Y^*\overline{\eta}){\rm
Td}_g^{-1}(\overline{\widetilde{F}})\cdot {j^0_g}^*[{\rm
Td}^{-1}(\overline{X_g})g_{X_g}-{\rm
Td}^{-1}(\overline{P^0_g})g_{P^0_g}\notag\\
&\qquad\qquad\qquad-{\rm
Td}^{-1}(\overline{\widetilde{X_g}})g_{\widetilde{X_g}}+{\rm
Td}^{-1}(\overline{\widetilde{X_g}}){\rm
Td}^{-1}(\overline{P_g^0})c_1(\overline{\mathcal{O}}(P^0_g))g_{\widetilde{X_g}}]\}.\notag\\
\label{eq8}
\end{align}

Gathering (\ref{eq5}), (\ref{eq6}), (\ref{eq7}) and (\ref{eq8}) we
finally get
\begin{align}
&\widetilde{{\rm
ch}}_g(\overline{\Xi}_\nabla\mid_X,h^H)-\widetilde{{\rm
ch}}_g(\overline{\Xi}_\nabla\mid_{P_g^0},h^H)+\sum(-1)^jT_g(\omega^{X_g},h^{\xi_j\mid_{X_g}})\notag\\
&-\sum(-1)^jT_g(\omega^{P^0_g},h^{\xi_j^\infty\mid_{P^0_g}})-\sum(-1)^jT_g(\omega^{Y_g},h^{\wedge^jF^\vee\otimes
\eta\mid_{Y_g}})\notag\\
&\qquad\qquad\qquad\qquad\qquad\qquad+\sum(-1)^jT_g(\omega^{Y_g},h^{\wedge^jF_\infty^\vee\otimes
\eta\mid_{Y_g}})\notag\\
=&\int_{X_g/S}{\rm
Td}(\overline{Th_g})T_g(\overline{\xi.})-\int_{P^0_g/S}{\rm
Td}(\overline{Tb_g})T_g(\overline{\xi.}^\infty)\notag\\
&+\int_{Y_g/S}{\rm ch}_g(\overline{\eta}){\rm
Td}_g^{-1}(\overline{N})\widetilde{{\rm
Td}}(\overline{Tf_g},\overline{Th_g}\mid_{Y_g})\notag\\
&-\int_{Y_g/S}{\rm ch}_g(\overline{\eta}){\rm
Td}_g^{-1}(\overline{N}_\infty)\widetilde{{\rm
Td}}(\overline{Tf_g},\overline{Tb_g}\mid_{Y_g}).\label{Eq(A)}
\end{align}

On the other hand, by definition, we have
\begin{align*}
\delta(\overline{\Xi}\mid_P):=&\widetilde{{\rm
ch}}_g(\overline{\xi.}^\infty,\overline{\eta})-\sum_k(-1)^kT_g(\omega^{Y_g},h^{\wedge^kF_\infty^\vee\otimes\eta\mid_{Y_g}})\\
&+\sum_k(-1)^kT_g(\omega^{P_g},h^{\xi^\infty_k\mid_{P_g}})\\
&-\int_{Y_g/S}{\rm Td}(Tf_g){\rm Td}_g^{-1}(F){\rm
ch}_g(\eta)R(N_g)\\
&-\int_{P_g/S}T_g(\overline{\xi.}^\infty){\rm
Td}(\overline{Tb'_g})\\
&-\int_{Y_g/S}{\rm ch}_g(\overline{\eta}){\rm
Td}_g^{-1}(\overline{N}_\infty)\widetilde{{\rm
Td}}(\overline{Tf_g},\overline{Tb'_g}\mid_{Y_g})
\end{align*}
where $b': P\to S$ is the composition of the inclusion
$P\hookrightarrow W(i)$ and the morphism $l$. Note that $P_g^0$ is
an open and closed submanifold of $P_g$ and
$\overline{\xi.}^\infty$ is orthogonally split on the other
components since they all belong to $\widetilde{X}_g$, then we can
rewrite $\delta(\overline{\Xi}\mid_P)$ as
\begin{align*}
\delta(\overline{\Xi}\mid_P)=&\widetilde{{\rm
ch}}_g(\overline{\Xi}_\nabla\mid_{P_g^0},h^H)-\sum_k(-1)^kT_g(\omega^{Y_g},h^{\wedge^kF_\infty^\vee\otimes\eta\mid_{Y_g}})\\
&+\sum_k(-1)^kT_g(\omega^{P^0_g},h^{\xi^\infty_k\mid_{P^0_g}})\\
&-\int_{Y_g/S}{\rm Td}(Tf_g){\rm Td}_g^{-1}(F){\rm
ch}_g(\eta)R(N_g)\\
&-\int_{P^0_g/S}T_g(\overline{\xi.}^\infty){\rm
Td}(\overline{Tb_g})\\
&-\int_{Y_g/S}{\rm ch}_g(\overline{\eta}){\rm
Td}_g^{-1}(\overline{N}_\infty)\widetilde{{\rm
Td}}(\overline{Tf_g},\overline{Tb_g}\mid_{Y_g}).
\end{align*}

Comparing with the definition of $\delta(\overline{\Xi}\mid_X)$,
the equality (\ref{Eq(A)}) implies that
\begin{displaymath}
\delta(\overline{\Xi}\mid_X)-\delta(\overline{\Xi}\mid_P)=0
\end{displaymath}
which completes the proof of this deformation theorem.
\end{proof}

Now we consider the zero section imbedding $i_\infty: Y\to
P=\mathbb{P}(N_\infty\oplus \mathcal{O}_Y)$. Here we use the fact
that $N_\infty$ is isomorphic to $N_{X/Y}$, we caution the reader
that this is not necessarily an isometry since
$\overline{N}_\infty$ carries the quotient metric induced by the
K\"{a}hler metric on $P$ but $N_{X/Y}$ carries the quotient metric
induced by the K\"{a}hler metric on $X$. We recall that on $P$ we
have a tautological exact sequence
\begin{displaymath}
0\to \mathcal{O}(-1)\to \pi_P^*(N_\infty\oplus\mathcal{O}_Y)\to
Q\to 0. \end{displaymath} The equivariant section $\sigma:
\mathcal{O}_P\to \pi_P^*(N_\infty\oplus\mathcal{O}_Y)\to Q$
induces the following Koszul resolution
\begin{displaymath}
0\rightarrow \wedge^{{\rm rk}Q}Q^\vee\rightarrow\cdots\rightarrow
Q^\vee\rightarrow\mathcal{O}_P\rightarrow
{i_\infty}_*\mathcal{O}_Y\rightarrow 0. \end{displaymath} Since
$\sigma$ is equivariant, the image of $\mathcal{O}_{P_g}$ under
$\sigma\mid_{P_g}$ is contained in $Q_g$. This means that
$\sigma\mid_{P_g}$ induces a Koszul resolution on $P_g$ of the
following form \begin{displaymath} 0\rightarrow \wedge^{{\rm
rk}Q_g}Q_g^\vee\rightarrow\cdots\rightarrow
Q_g^\vee\rightarrow\mathcal{O}_{P_g}\rightarrow
{i_{\infty,g}}_*\mathcal{O}_{Y_g}\rightarrow 0. \end{displaymath}

\begin{prop}\label{319}
Let
$\overline{\kappa}:=\overline{\kappa}(\overline{\eta},\overline{N}_\infty)$
be a hermitian Koszul resolution on $P$ defined in
Definition~\ref{303}. Then for $n\gg0$, we have
$\delta(\overline{\kappa}(n))=0$.
\end{prop}
\begin{proof}
Denote the non-zero degree part of $Q\mid_{P_g}$ by $Q_\bot$, then
we have the following isometry
\begin{displaymath}
\wedge^i\overline{Q}^\vee\mid_{P_g}=\wedge^i(\overline{Q}_g^\vee\oplus
\overline{Q}_\bot^\vee)\cong
\bigoplus_{t+s=i}(\wedge^t\overline{Q}_g^\vee\otimes\wedge^s\overline{Q}_\bot^\vee).
\end{displaymath} Consider the following complex of equivariant hermitian
vector bundles on $P_g$
\begin{align*}
0\to \wedge^{{\rm
rk}Q_g}\overline{Q}_g^\vee\otimes(\wedge^k\overline{Q}_\bot^\vee\otimes
\pi_{P_g}^*\overline{\eta}\mid_{Y_g})\to \cdots\to
\overline{Q}_g^\vee&\otimes(\wedge^k\overline{Q}_\bot^\vee\otimes
\pi_{P_g}^*\overline{\eta}\mid_{Y_g})\\
&\to
\wedge^k\overline{Q}_\bot^\vee\otimes
\pi_{P_g}^*\overline{\eta}\mid_{Y_g}\to 0
\end{align*}
which provides a resolution of
${i_{\infty,g}}_*(\wedge^k\overline{F}_\infty^\vee\otimes\overline{\eta}\mid_{Y_g})$
where $F_\infty$, as before, is the non-zero degree part of the
normal bundle $N_\infty$ associated to $i_\infty$. We denote this
resolution by $\overline{\kappa}^{(k)}$, then according to the
arguments given before this proposition we have a decomposition of
complexes
$\overline{\kappa}_\nabla\mid_{P_g}\cong\bigoplus_{k\geq0}\overline{\kappa}_\nabla^{(k)}[-k]$
where $\overline{\kappa}_\nabla^{(k)}[-k]$ is obtained from
$\overline{\kappa}_\nabla^{(k)}$ by shifting degree. Replacing
$\overline{\kappa}$ by $\overline{\kappa}(n)$ for big enough $n$,
we may assume that all elements in $\overline{\kappa}$ and
$\overline{\kappa}^{(k)}$ are acyclic. Therefore, by Bisumt-Ma's
immersion formula we have the following equality
\begin{align*}
\widetilde{{\rm
ch}}_g({b'_g}_*\overline{\kappa}^{(k)})=&T_g(\omega^{Y_g},h^{\wedge^kF_\infty^\vee\otimes\eta\mid_{Y_g}})-\sum_{i=0}^{{\rm
rk}Q_g}(-1)^iT_g(\omega^{P_g},h^{\wedge^iQ_g^\vee\otimes\wedge^kQ_\bot^\vee\otimes\pi_{P_g}^*\eta\mid_{Y_g}})\\
&+\int_{Y_g/S}{\rm
ch}_g(\wedge^kF_\infty^\vee\otimes\eta\mid_{Y_g})R(N_{\infty,g}){\rm
Td}(Tf_g)\\
&+\int_{P_g/S}{\rm
Td}(\overline{Tb'_g})T_g(\overline{\kappa}^{(k)})\\
&+\int_{Y_g/S}{\rm
ch}_g(\wedge^k\overline{F}_\infty^\vee\otimes\overline{\eta}\mid_{Y_g}){\rm
Td}^{-1}(\overline{N}_{\infty,g})\widetilde{{\rm
Td}}(\overline{Tf_g},\overline{Tb'_g}\mid_{Y_g}).
\end{align*}
It is easily seen from the decomposition
$\overline{\kappa}_\nabla\mid_{P_g}=\bigoplus_{k\geq0}\overline{\kappa}_\nabla^{(k)}[-k]$
that the secondary characteristic class $\widetilde{{\rm
ch}}_g(\overline{\kappa})$ appearing in the definition of
$\delta(\overline{\kappa})$ is exactly $\sum(-1)^k\widetilde{{\rm
ch}}_g({b'_g}_*\overline{\kappa}^{(k)})$. So taking the
alternating sum of both two sides of the equality above and using
the fact that equivariant analytic torsion form is additive for
direct sum of acyclic bundles, we know that to prove
$\delta(\overline{\kappa})=0$, we are left to show that
$\sum(-1)^kT_g(\overline{\kappa}^{(k)})$ is equal to
$T_g(\overline{\kappa})$. In fact, by using \cite[Lemma
3.15]{KR1}, we have modulo ${\rm Im}\partial+{\rm
Im}\overline{\partial}$
\begin{align*}
\sum(-1)^kT_g(\overline{\kappa}^{(k)})&=\sum(-1)^k{\rm
ch}_g(\wedge^k\overline{Q}_\bot^\vee){\rm
ch}_g(\pi_{P_g}^*\overline{\eta}\mid_{Y_g})T_g(\overline{\wedge^\cdot
Q_g^\vee})\\
&={\rm Td}_g^{-1}(\overline{Q}_\bot){\rm
ch}_g(\pi_{P_g}^*\overline{\eta}\mid_{Y_g})T_g(\overline{\wedge^\cdot
Q_g^\vee})\\
&={\rm Td}_g^{-1}(\overline{Q}){\rm
ch}_g(\pi_{P_g}^*\overline{\eta}\mid_{Y_g})T_g(\overline{\wedge^\cdot
Q_g^\vee}){\rm Td}(\overline{Q}_g)\\
&={\rm
ch}_g(\pi_{P_g}^*\overline{\eta}\mid_{Y_g})T_g(\overline{\wedge^\cdot
Q^\vee})\\
&=T_g(\overline{\kappa}).
\end{align*}
So we are done.
\end{proof}

It's now ready to finish the proof of the vanishing theorem for
equivariant closed immersions. Let $\overline{\eta}$ be an
equivariant hermitian vector bundle on $Y$, assume that
\begin{displaymath}
\overline{\Psi}:\quad 0\to \overline{\xi}_m\to \cdots\to
\overline{\xi}_1\to \overline{\xi}_0\to i_*\overline{\eta}\to 0
\end{displaymath} is a resolution of $i_*\overline{\eta}$ by
equivariant hermitian vector bundles on $X$ which satisfies Bismut
assumption (A). We need to prove that for $n\gg0$,
$\delta(\overline{\Psi}(n))=0$.

\begin{proof}(of Theorem~\ref{301})
We first construct a resolution of $p_Y^*\overline{\eta}$ on
$W(i)$ as \begin{displaymath} \overline{\Xi}:\quad 0\to
\overline{\widetilde{\xi}}_m\to \cdots \to
\overline{\widetilde{\xi}}_0\to \overline{\widetilde{\xi}}_0 \to
j_*p_Y^*(\overline{\eta})\to 0 \end{displaymath} which satisfies
the condition (i) and (ii) in Definition~\ref{316}. Then the
restriction of $\overline{\Xi}$ to $X$ (resp. $P$) provides a
resolution of $i_*\overline{\eta}$ (resp.
${i_\infty}_*\overline{\eta}$). Over $X$, we can find a third
resolution $\overline{\Phi}$ of $i_*\overline{\eta}$ which
dominates $\overline{\Psi}$ and $\overline{\Xi}\mid_X$. Namely we
get short exact sequences of exact sequences \begin{displaymath} 0
\to \overline{{\rm Ker}}(n)\to \overline{\Phi}(n)\to
\overline{\Psi}(n)\to 0 \end{displaymath} and
\begin{displaymath}
0 \to\overline{{\rm Ker}}'(n)\to \overline{\Phi}(n)\to
\overline{\Xi}(n)\mid_X\to 0. \end{displaymath} Then after
omitting $i_*\overline{\eta}$ their restrictions to $X_g$ become
two exact sequences of complexes. Since $n\gg0$ we may assume that
all elements and homologies in the induced double complexes are
acyclic, so that by taking direct images we get two exact
sequences of equivariant standard complexes on $S$. These two
short exact sequences of equivariant standard complexes clearly
satisfy the assumptions in Lemma~\ref{312}. Therefore, using
Lemma~\ref{312}, Bismut-Ma's immersion formula and the double
complex formula of equivariant Bott-Chern singular currents (cf.
Theorem~\ref{214}), we obtain that
\begin{align*}
&\widetilde{{\rm ch}}_g(\overline{\Psi}(n))-\widetilde{{\rm
ch}}_g(\overline{\Phi}(n))+\widetilde{{\rm
ch}}_g(\overline{{\rm Ker}}(n))\\
&\qquad\qquad+T_g(\omega^{X_g},h^{\Psi(n)_\nabla})-T_g(\omega^{X_g},h^{\Phi(n)_\nabla})+T_g(\omega^{X_g},h^{{\rm Ker}(n)_\nabla})\\
=&\int_{X_g/S}[T_g(\overline{\Psi}(n)_\nabla)-T_g(\overline{\Phi}(n)_\nabla)+T_g(\overline{{\rm
Ker}}(n)_\nabla)]\cdot{\rm Td}(\overline{Th_g})
\end{align*}
which implies that
\begin{displaymath}
\delta(\overline{\Phi}(n))=\delta(\overline{\Psi}(n))+\delta(\overline{{\rm
Ker}}(n)).
\end{displaymath} By applying Bismut-Ma's immersion formula to the
case where the immersion is the identity map and $\overline{\eta}$
is equal to the zero bundle, we get $\delta(\overline{{\rm
Ker}}(n))=0$ so that
$\delta(\overline{\Phi}(n))=\delta(\overline{\Psi}(n))$.
Similarly, we have
$\delta(\overline{\Phi}(n))=\delta(\overline{\Xi}(n)\mid_X)$ and
hence
$\delta(\overline{\Psi}(n))=\delta(\overline{\Xi}(n)\mid_X)$. An
entirely analogous reasoning implies that
$\delta(\overline{\kappa}(n))=\delta(\overline{\Xi}(n)\mid_P)$.
Then the vanishing of $\delta(\overline{\Psi}(n))$ follows from
Theorem~\ref{318} and Proposition~\ref{319}.
\end{proof}

\section{Equivariant arithmetic Grothendieck groups with fixed wave front
sets} By an arithmetic ring $D$ we understand a regular,
excellent, Noetherian integral ring, together with a finite set
$\mathcal{S}$ of embeddings $D\hookrightarrow \C$, which is
invariant under a conjugate-linear involution $F_\infty$ (cf.
\cite[Def. 3.1.1]{GS}). Denote by $\mu_n$ the diagonalisable group
scheme over $D$ associated to $\Z/{n\Z}$. A $\mu_n$-equivariant
arithmetic scheme over $D$ is a Noetherian scheme of finite type,
endowed with a $\mu_n$-projective action over $D$ (cf.
\cite[Section 2]{KR1}). Let $X$ be a $\mu_n$-equivariant
arithmetic scheme whose generic fibre is smooth, then $X(\C)$, the
set of complex points of the variety $\coprod_{\sigma\in
\mathcal{S}}X\times_D\C$, is a disjoint union of projective
manifolds. This manifold admits an action of the group of complex
$n$-th roots of unity and an anti-holomorphic involution induced
by $F_\infty$ which is still denoted by $F_\infty$. It was shown
in \cite[Prop. 3.1]{Th} that if $X$ is regular, then the fixed
point subscheme $X_{\mu_n}$ is also regular. Fix a primitive
$n$-th root of unity $\zeta_n$ and denote its corresponding
holomorphic automorphism on $X(\C)$ by $g$, by GAGA principle we
have a natural isomorphism $X_{\mu_n}(\C)\cong X(\C)_g$.

\begin{defn}\label{401}
An equivariant hermitian sheaf (resp. vector bundle)
$\overline{E}$ on $X$ is a coherent sheaf (resp. vector bundle)
$E$ on $X$, assumed locally free on $X(\C)$, endowed with a
$\mu_n$-action which lifts the action of $\mu_n$ on $X$ and a
hermitian metric $h$ on the associated bundle $E_\C$, which is
invariant under $F_\infty$ and $g$.
\end{defn}

\begin{rem}\label{402}
Let $\overline{E}$ be an equivariant hermitian sheaf (resp. vector
bundle) on $X$, the restriction of $\overline{E}$ to the fixed
point subscheme $X_{\mu_n}$ has a natural $\Z/{n\Z}$-grading
structure $\overline{E}\mid_{X_{\mu_n}}\cong \oplus_{k\in
\Z/{n\Z}}\overline{E}_k$. We shall often write
$\overline{E}_{\mu_n}$ for $\overline{E}_0$. It is clear that the
associated bundle of $\overline{E}_{\mu_n}$ over $X(\C)$ is
exactly equal to $\overline{E}_g$.
\end{rem}

Over a complex manifold $M$, we may consider the current space
which is the continuous dual of the space of smooth complex valued
differential forms (cf. \cite[Chapter IX]{deRh}). The wave front
set ${\rm WF}(\omega)$ of a current $\omega$ over $M$ is a closed
conical subset of the cotangent bundle
$T_\R^*M_0:=T_\R^*M\backslash\{0\}$. This conical subset measures
the singularities of $\omega$, actually the projection of ${\rm
WF}(\omega)$ in $M$ is equal to the singular locus of the support
of $\omega$. It also allows us to define certain products and
pull-backs of currents. We refer to \cite[Chapter VIII]{Hoe} for
the definition and various properties of wave front set.

Now let $X$ be a $\mu_n$-equivariant arithmetic scheme with smooth
generic fibre and let $S$ be a conical subset of
$T_\R^*X(\C)_{g,0}$, denote by $D^{p,p}(X(\C)_g,S)$ the set of
currents $\omega$ of type $(p,p)$ on $X(\C)_g$ which satisfy
$F_\infty^*\omega=(-1)^p\omega$ and whose wave front sets are
contained in $S$, we shall write
$\widetilde{\mathcal{U}}(X_{\mu_n},S)$ for the current class
\begin{displaymath}
\widetilde{\mathcal{U}}(X(\C)_g,S):=\bigoplus_{p\geq
0}(D^{p,p}(X(\C)_g,S)/({\rm Im}\partial+{\rm
Im}\overline{\partial})). \end{displaymath}

Let $\overline{E}$ be an equivariant hermitian sheaf or vector
bundle on $X$. Following the same notations and definitions as in
\cite[Section 3]{KR1}, we write ${\rm ch}_g(\overline{E})$ for the
equivariant Chern character form ${\rm ch}_g((E_\C,h))$ associated
to the hermitian holomorphic vector bundle $(E_\C,h)$ on $X(\C)$.
Similarly, we have the equivariant Todd form ${\rm
Td}_g(\overline{E})$. Furthermore, let $\overline{\varepsilon}:
0\rightarrow \overline{E}'\rightarrow \overline{E}\rightarrow
\overline{E}''\rightarrow 0$ be an exact sequence of equivariant
hermitian sheaves or vector bundles on $X$, we can associate to it
an equivariant Bott-Chern secondary characteristic class
$\widetilde{{\rm ch}}_g(\overline{\varepsilon})\in
\widetilde{\mathcal{U}}(X_{\mu_n},\emptyset)$ which satisfies the
differential equation
\begin{displaymath}
{\rm dd}^c\widetilde{{\rm ch}}_g(\overline{\varepsilon})={\rm
ch}_g(\overline{E}')-{\rm ch}_g(\overline{E})+{\rm
ch}_g(\overline{E}''). \end{displaymath}

\begin{defn}\label{403}
Let $X$ be a $\mu_n$-equivariant arithmetic scheme with smooth
generic fibre and let $S$ be a conical subset of
$T_\R^*X(\C)_{g,0}$, we define the equivariant arithmetic
Grothendieck group $\widehat{G_0}(X,\mu_n,S)$ (resp.
$\widehat{K_0}(X,\mu_n,S)$) with respect to $X$ and $S$ as the
free abelian group generated by the elements of
$\widetilde{\mathcal{U}}(X_{\mu_n},S)$ and by the equivariant
isometry classes of equivariant hermitian sheaves (resp. vector
bundles) on $X$, together with the relations

(i). for every exact sequence $\overline{\varepsilon}$ as above,
$\widetilde{{\rm
ch}}_g(\overline{\varepsilon})=\overline{E}'-\overline{E}+\overline{E}''$;

(ii). if $\alpha\in \widetilde{\mathcal{U}}(X_{\mu_n},S)$ is the
sum of two elements $\alpha'$ and $\alpha''$ in
$\widetilde{\mathcal{U}}(X_{\mu_n},S)$, then the equality
$\alpha=\alpha'+\alpha''$ holds in $\widehat{G_0}(X,\mu_n,S)$
(resp. $\widehat{K_0}(X,\mu_n,S)$).
\end{defn}

\begin{rem}\label{404}
(i). When $S'\subset S$, \cite[Theorem 3.9 (ii)]{T2} implies that
the natural map from $\widetilde{\mathcal{U}}(X_{\mu_n},S')$ to
$\widetilde{\mathcal{U}}(X_{\mu_n},S)$ is injective. So the first
generating relation in Definition~\ref{403} does make sense.

(ii). When $X$ is regular, one can carry out the proof of
\cite[Proposition 4.2]{KR1} to show that the natural morphism from
$\widehat{K_0}(X,\mu_n,S)$ to $\widehat{G_0}(X,\mu_n,S)$ is an
isomorphism.

(iii). The definition of the equivariant arithmetic Grothendieck
group implies that there are exact sequences
\begin{displaymath}
\xymatrix{
 \widetilde{\mathcal{U}}(X_{\mu_n},S)
\ar[r]^-{a} & \widehat{G_0}(X,\mu_n,S) \ar[r]^-{\pi} &
G_0(X,\mu_n) \ar[r] & 0} \end{displaymath} and
\begin{displaymath}
\xymatrix{
 \widetilde{\mathcal{U}}(X_{\mu_n},S)
\ar[r]^-{a} & \widehat{K_0}(X,\mu_n,S) \ar[r]^-{\pi} &
K_0(X,\mu_n) \ar[r] & 0} \end{displaymath} where $a$ is the
natural map which sends $\alpha\in
\widetilde{\mathcal{U}}(X_{\mu_n},S)$ to the class of $\alpha$ in
$\widehat{G_0}(X,\mu_n,S)$ (resp. $\widehat{K_0}(X,\mu_n,S)$) and
$\pi$ is the forgetful map. Here the group $G_0(X,\mu_n)$ is the
Grothendieck group of $\mu_n$-equivariant coherent sheaves which
are locally free on $X(\C)$, by a theorem of Quillen (cf.
\cite[Thm. 3 Cor. 1]{Qu1}) we know that it is isomorphic to the
ordinary Grothendieck group of $\mu_n$-equivariant coherent
sheaves.
\end{rem}

In \cite[Section 3]{T1}, we have introduced the ring structure of
$\widehat{K_0}(X,\mu_n,\emptyset)$. Since we may have a
well-defined product of two currents if their wave front sets have
no intersection, and the wave front set is invariant under the
operation of multiplying a smooth current, we know that the
Grothendieck group $\widehat{K_0}(X,\mu_n,S)$ has a
$\widehat{K_0}(X,\mu_n,\emptyset)$-module structure. The same
thing goes to $\widehat{G_0}(X,\mu_n,S)$. Furthermore, let
$\overline{I}$ be the $\mu_n$-equivariant hermitian $D$-module
whose term of degree $1$ is $D$ endowed with the trivial metric
and whose other terms are $0$. Then we may make
$\widehat{K_0}(D,\mu_n,\emptyset)$ an $R(\mu_n)$-algebra under the
ring morphism which sends $T$ to $\overline{I}$. By doing
pull-backs, we may endow every arithmetic Grothendieck group we
defined before with an $R(\mu_n)$-module structure. Notice finally
that there is a well-defined map from
$\widehat{G_0}(X,\mu_n,\emptyset)$ (resp.
$\widehat{K_0}(X,\mu_n,\emptyset)$) to the space of complex closed
differential forms, which is defined by the formula ${\rm
ch}_g(\overline{E}+\alpha):={\rm ch}_g(\overline{E})+{\rm
dd}^c\alpha$ where $\overline{E}$ is an equivariant hermitian
sheaf (resp. vector bundle) and $\alpha\in
\widetilde{\mathcal{U}}(X_{\mu_n},\emptyset)$.

Now we investigate the wave front set of a current after doing
push-forward. Let $f$ be a holomorphic map of compact complex
manifolds, we may define a push-forward $f_*$ on current space
which is the dual map of the pull-back of smooth forms. When $f$
is smooth, the push-forward $f_*$ extends the integration of
smooth forms over the fibre. Assume that we are given a smooth
morphism $f: U\to V$ of compact complex manifolds, then $f_*$
induces a current $K$ over the product space $V\times U$ defined
as \begin{displaymath} K(\alpha\otimes \beta)=(f_*\beta)(\alpha)
\end{displaymath} where $\alpha$ and $\beta$ are smooth forms over
$V$ and $U$ respectively. Define \begin{displaymath} M=\{(v,u)\in
V\times U \mid f(u)=v\} \end{displaymath} which is a submanifold
in $V\times U$. From the fact that $f_*\beta$ is just the
integration of smooth forms over the fibre, it is easily seen that
the current $K\in D^*(V\times U)$ is exactly the object ${\rm
d}S_M$ in \cite[Theorem 8.1.5]{Hoe}. Then by that theorem, the
wave front set of $K$ is equal to \begin{displaymath} {\rm
WF}(K)=\{(v,u,\xi,-f^*(\xi))\in T_\R^*V\times T_\R^*U \mid f(u)=v,
\xi\neq0\}.\end{displaymath}

Let $S$ be a conical subset of $T_\R^*U_0$, we fix some notations
as follows.
\begin{align*}
{\rm WF}(K)_V=&\{(v,\xi)\in T_\R^*V_0 \mid \exists u\in U,
(v,u,\xi,0)\in {\rm WF}(K)\}\\
{\rm WF}'(K)_U=&\{(u,\eta)\in T_\R^*U_0 \mid \exists v\in V,
(v,u,0,-\eta)\in {\rm WF}(K)\}\\
{\rm WF}'(K)_V\circ S=&\{(v,\xi)\in T_\R^*V_0 \mid \exists
(u,\eta)\in S, (v,u,\xi,-\eta)\in {\rm WF}(K)\}.
\end{align*}

\begin{thm}\label{405}
Let notations and assumptions be as above. Assume that $\omega$ is
a current over $U$ whose wave front set is contained in $S$ with
$S\cap {\rm WF}'(K)_U=\emptyset$, then the wave front set of
$f_*\omega$ is contained in \begin{displaymath} S':={\rm
WF}(K)_V\cup{\rm WF}'(K)\circ S. \end{displaymath}
\end{thm}
\begin{proof}
This follows from \cite[Theorem 8.2.12 and 8.2.13]{Hoe}.
\end{proof}

\begin{rem}\label{406}
(i) In our situation, the condition $S\cap {\rm
WF}'(K)_U=\emptyset$ is always satisfied because by definition we
have ${\rm WF}'(K)_U=\emptyset$.

(ii). In our situation, $S'$ is always equal to ${\rm WF}'(K)\circ
S$ because ${\rm WF}(K)_V=\emptyset$.

(iii). If $S$ is the empty set, then $S'$ is also empty. This is
compatible with the push-forward of smooth forms.

(iv). Assume that the restriction of $f$ to a closed submanifold
$W$ is also smooth. Denote by $N_{U/W}$ the normal bundle of $W$
in $U$. If $S=N_{U/W,\R}^\vee\setminus\{0\}$, then $S'=\emptyset$.
\end{rem}

We now turn to the arithmetic case. Let $X$, $Y$ be two
$\mu_n$-equivariant arithmetic schemes with smooth generic fibres,
and let $f: X\rightarrow Y$ be an equivariant morphism over $D$
which is smooth on the complex numbers. Fix a
$\mu_n(\C)$-invariant K\"{a}hler metric on $X(\C)$ so that we get
a K\"{a}hler fibration with respect to the holomorphic submersion
$f_\C: X(\C)\rightarrow Y(\C)$. Let $\overline{E}$ be an
$f$-acyclic $\mu_n$-equivariant hermitian sheaf on $X$, we know
that the direct image $f_*E$ is locally free on $Y(\C)$ and it can
be endowed with a natural equivariant structure and the
$L^2$-metric. Let $\widehat{G_0}^{\rm ac}(X,\mu_n,S)$ be the group
generated by $f$-acyclic equivariant hermitian sheaves on $X$ and
the elements of $\widetilde{\mathcal{U}}(X_{\mu_n},S)$, with the
same relations as in Definition~\ref{403}. A theorem of Quillen
(cf. \cite[Cor.3 P. 111]{Qu1}) for the algebraic analogs of these
groups implies that the natural map $\widehat{G_0}^{\rm
ac}(X,\mu_n,S)\to \widehat{G_0}(X,\mu_n,S)$ is an isomorphism. So
the following definition does make sense.

\begin{defn}\label{407}
Let notations and assumptions be as above. The push-forward
morphism $f_*:\widehat{G_0}(X,\mu_n,S)\rightarrow
\widehat{G_0}(Y,\mu_n,S')$ is defined in the following way.

(i). For every $f$-acyclic $\mu_n$-equivariant hermitian sheaf
$\overline{E}$ on $X$,
$f_*\overline{E}=(f_*E,f_*h^E)-T_g(\omega^X,h^E)$.

(ii). For every element $\alpha\in
\widetilde{\mathcal{U}}(X_{\mu_n},S)$,
$f_*\alpha=\int_{X_g/{Y_g}}{\rm Td}_g(Tf,h^{Tf})\alpha\in
\widetilde{\mathcal{U}}(Y_{\mu_n},S')$.
\end{defn}

\begin{rem}\label{tang}
If $Y$ is regular, by Remark~\ref{404} (ii) we know that
$\widehat{K_0}(Y,\mu_n,S')$ is naturally isomorphic to
$\widehat{G_0}(Y,\mu_n,S')$ so that $(f_*E,f_*h^E)$ admits a
finite equivariant hermitian resolution; if the morphism $f$ is
flat and $Y$ is reduced, then $(f_*E,f_*h^E)$ is locally free when
$E$ is so. Therefore in both two cases above, one can also define
a reasonable push-forward morphism
$f_*:\widehat{K_0}(X,\mu_n,S)\rightarrow
\widehat{K_0}(Y,\mu_n,S')$.
\end{rem}

\begin{thm}\label{408}
The push-forward morphism $f_*$ is a well-defined group
homomorphism.
\end{thm}
\begin{proof}
The argument is the same as in the proof of \cite[Theorem
6.2]{T1}.
\end{proof}

\begin{lem}\label{409}(Projection formula)
For any elements $y\in \widehat{K_0}(Y,\mu_n,\emptyset)$ and $x\in
\widehat{G_0}(X,\mu_n,S)$, the identity $f_*(f^*y\cdot x)=y\cdot
f_*x$ holds in $\widehat{G_0}(Y,\mu_n,S')$.
\end{lem}
\begin{proof}
Assume that $y=\overline{E}$ is an equivariant hermitian vector
bundle and $x=\overline{F}$ is an $f$-acyclic equivariant
hermitian sheaf, then $f^*y\cdot x=f^*\overline{E}\otimes
\overline{F}$. By projection formula for direct images and the
definition of $L^2$-metric, we know that
$f_*(f^*\overline{E}\otimes \overline{F})$ is isometric to
$\overline{E}\otimes f_*\overline{F}$. Moreover, concerning the
analytic torsion form, we have $T_g(\omega^X,h^{f^*E\otimes
F})={\rm ch}_g(\overline{E})T_g(\omega^X,h^F)$. So the projection
formula $f_*(f^*y\cdot x)=y\cdot f_*x$ holds in this case.

Assume that $y=\overline{E}$ is an equivariant hermitian vector
bundle and $x=\alpha$ is represented by some singular current. We
write $f_g^*$ and ${f_g}_*$ for the pull-back and push-forward of
currents respectively, then
\begin{align*}
f_*(f^*y\cdot x)=&f_*(f_g^*{\rm
ch}_g(\overline{E})\alpha)={f_g}_*(f_g^*{\rm
ch}_g(\overline{E})\alpha{\rm Td}_g(\overline{Tf}))\\
=&{\rm
ch}_g(\overline{E}){f_g}_*(\alpha{\rm Td}_g(\overline{Tf}))\\
=&{\rm ch}_g(\overline{E})\int_{X_g/{Y_g}}\alpha{\rm
Td}_g(\overline{Tf})=y\cdot f_*x.
\end{align*} Here we have
used an extension of projection formula of smooth forms
$p_*(p^*\alpha_1\wedge\alpha_2)=\alpha_1\wedge p_*\alpha_2$ (cf.
\cite[Prop. IX p. 303]{GHV}) to the case where the second variable
$\alpha_2$ is replaced by a singular current. The fact that this
extension is valid follows from the definition of $p_*$ and the
definition of the product of smooth form and singular current.

Assume that $y=\beta$ is represented by some smooth form and
$x=\overline{E}$ is an $f$-acyclic hermitian sheaf, then
\begin{align*}
f_*(f^*y\cdot x)=&f_*(f_g^*(\beta){\rm
ch}_g(\overline{F}))={f_g}_*(f_g^*(\beta){\rm
ch}_g(\overline{F}){\rm Td}_g(\overline{Tf}))\\
=&\beta{f_g}_*({\rm
ch}_g(\overline{E}){\rm
Td}_g(\overline{Tf}))\\
=&\beta\int_{X_g/{Y_g}}{\rm ch}_g(\overline{E}){\rm
Td}_g(\overline{Tf})=\beta({\rm ch}_g(\overline{f_*F})-{\rm
dd}^cT_g(\omega^X,h^F))
\end{align*}
which is exactly $y\cdot f_*x$.

Finally, assume that $y=\beta$ is represented by some smooth form
and $x=\alpha$ is represented by some singular current, then
\begin{align*}
f_*(f^*y\cdot x)=&f_*(f_g^*(\beta){\rm
dd}^c\alpha)={f_g}_*(f_g^*(\beta){\rm dd}^c\alpha{\rm
Td}_g(\overline{Tf}))\\
=&\beta{\rm dd}^c{f_g}_*(\alpha{\rm
Td}_g(\overline{Tf}))
\end{align*} which is also equal to $y\cdot f_*x$.

Since $f_*$ and $f^*$ are both group homomorphisms, we may
conclude the projection formula by linear extension.
\end{proof}

\begin{rem}\label{410}
Lemma~\ref{409} implies that $f_*$ is a homomorphism of
$R(\mu_n)$-modules, and hence it induces a push-forward morphism
after taking localization.
\end{rem}

To end this section, we recall an important lemma which will be
used frequently in our later arguments.

\begin{lem}\label{411}(\cite[Lemma 4.5]{KR1})
Let $X$ be a regular $\mu_n$-equivariant arithmetic scheme and let
$\overline{E}$ be an equivariant hermitian vector bundle on
$X_{\mu_n}$ such that $\overline{E}_{\mu_n}=0$. Then the element
$\lambda_{-1}(\overline{E})$ is invertible in
$\widehat{K_0}(X_{\mu_n},\mu_n,S)_{\rho}$.
\end{lem}

\section{Arithmetic concentration theorem}
In this section, we shall prove the arithmetic concentration
theorem which is an analog of Thomason's result in Arakelov
geometry. Let $X$ be a $\mu_n$-equivariant arithmetic scheme with
smooth generic fibre, we consider a special closed immersion $i:
X_{\mu_n}\hookrightarrow X$ where $X_{\mu_n}$ is the fixed point
subscheme of $X$. We shall first construct a well-defined group
homomorphism $i_*$ between equivariant arithmetic $G_0$-groups as
in the algebraic case. To construct $i_*$, some analytic datum,
which is the equivariant Bott-Chern singular current, should be
involved. Precisely speaking, let $\overline{\eta}$ be a
$\mu_n$-equivariant hermitian sheaf on $X_{\mu_n}$ and let
$\overline{\xi}.$ be a bounded complex of $\mu_n$-equivariant
hermitian sheaves which provides a resolution of
$i_*\overline{\eta}$ on $X$. Such a resolution always exists since
the generic fibre of $X$ is supposed to be smooth. Then we may
have an equivariant Bott-Chern singular current
$T_g(\overline{\xi}.)\in \widetilde{\mathcal{U}}(X_{\mu_n})$. Note
that on the complex numbers the $0$-degree part of the normal
bundle $N:=N_{X/{X_g}}$ vanishes (cf. \cite[Prop. 2.12]{KR1}) so
that the wave front set of $T_g(\overline{\xi}.)$ is the empty
set. This fact means that the following definition does make
sense.

\begin{defn}\label{501}
Let notations and assumptions be as above. The embedding morphism
\begin{displaymath}
i_*:\quad \widehat{G_0}(X_{\mu_n},\mu_n,S)\rightarrow
\widehat{G_0}(X,\mu_n,S) \end{displaymath} is defined in the
following way.

(i). For every $\mu_n$-equivariant hermitian sheaf
$\overline{\eta}$ on $X_{\mu_n}$, suppose that $\overline{\xi}.$
is a resolution of $i_*\overline{\eta}$ on $X$ whose metrics
satisfy Bismut assumption (A),
$i_*[\overline{\eta}]=\sum_k(-1)^k[\overline{\xi}_k]+T_g(\overline{\xi.})$.

(ii). For every $\alpha\in \widetilde{\mathcal{U}}(X_{\mu_n},S)$,
$i_*\alpha=\alpha{\rm Td}_g^{-1}(\overline{N})$.
\end{defn}

\begin{thm}\label{502}
The embedding morphism $i_*$ is a well-defined group homomorphism.
\end{thm}
\begin{proof}
The argument is the same as in the proof of \cite[Theorem
5.2]{T1}.
\end{proof}

\begin{lem}\label{506}(Projection formula)
For any elements $x\in \widehat{K_0}(X,\mu_n,\emptyset)$ and $y\in
\widehat{G_0}(X_{\mu_n},\mu_n,S)$, the identity $i_*(i^*x\cdot
y)=x\cdot i_*y$ holds in $\widehat{G_0}(X,\mu_n,S)$.
\end{lem}
\begin{proof}
Assume that $x=\overline{E}$ is an equivariant hermitian vector
bundle and $y=\overline{F}$ is an equivariant hermitian sheaf. Let
$\overline{\xi.}$ be a resolution of $i_*\overline{F}$ on $X$,
then $\overline{E}\otimes \overline{\xi.}$ provides a resolution
of $i_*(i^*\overline{E}\otimes \overline{F})$. By definition we
have \begin{displaymath} i_*(i^*x\cdot
y)=\sum(-1)^k[\overline{\xi} _k\otimes \overline{E}]+{\rm
ch}_g(\overline{E})T_g(\overline{\xi.}) \end{displaymath} which is
exactly $x\cdot i_*y$. Assume that $x=\alpha$ is represented by
some smooth form and $y=\overline{F}$ is an equivariant hermitian
sheaf. Again let $\overline{\xi.}$ be a resolution of
$i_*\overline{F}$ on $X$, then \begin{displaymath} i_*(i^*x\cdot
y)=\alpha{\rm Td}_g^{-1}(\overline{N}_{X/{X_g}}){\rm
ch}_g(\overline{F})=\alpha[{\rm
dd}^cT_g(\overline{\xi.})+\sum(-1)^k{\rm ch}_g(\overline{\xi}_k)]
\end{displaymath} which is exactly $x\cdot i_*y$. Now assume that
$x=\overline{E}$ is an equivariant hermitian vector bundle and
$y=\alpha$ is represented by some singular current, then
\begin{displaymath}
i_*(i^*x\cdot y)=i_*({\rm ch}_g(\overline{E})\alpha)={\rm
ch}_g(\overline{E})\alpha{\rm Td}_g^{-1}(\overline{N}_{X/{X_g}})
\end{displaymath} which is exactly $x\cdot i_*y$. Finally, if $x$
is represented by some smooth form and $y$ is represented by some
singular current then their product is well-defined and
$i_*(i^*x\cdot y)$ is obviously equal to $x\cdot i_*y$. Note that
$i_*$ and $i^*$ are group homomorphisms, so we may conclude the
projection formula from its correctness on generators. This
completes the proof.
\end{proof}

\begin{rem}\label{507}
Lemma~\ref{506} implies that $i_*$ is even a homomorphism of
$R(\mu_n)$-modules so that it induces a homomorphism between
arithmetic $G_0$-groups after taking localization.
\end{rem}

With Remark~\ref{507}, we may formulate the arithmetic
concentration theorem as follows.

\begin{thm}\label{503}
The embedding morphism $i_*:
\widehat{K_0}(X_{\mu_n},\mu_n,S)_{\rho}\rightarrow
\widehat{K_0}(X,\mu_n,S)_{\rho}$ is an isomorphism if $X$ is
regular. In this case, the inverse morphism of $i_*$ is given by
$\lambda_{-1}^{-1}(\overline{N}_{X/{X_{\mu_n}}}^\vee)\cdot i^*$
where $N_{X/{X_{\mu_n}}}$ is the normal bundle of $i(X_{\mu_n})$
in $X$.
\end{thm}

Before we give the proof of this concentration theorem, we recall
a crucial lemma as follows.

\begin{lem}\label{504}
Let $\overline{\eta}$ be an equivariant hermitian vector bundle on
$X_{\mu_n}$. Assume that $\overline{\xi.}$ is an equivariant
hermitian resolution of $i_*\overline{\eta}$ on $X$ whose metrics
satisfy Bismut assumption (A). Then the equality
\begin{displaymath}
\lambda_{-1}(\overline{N}_{X/{X_{\mu_n}}}^\vee)\cdot
\overline{\eta}-\sum_j(-1)^ji^*(\overline{\xi}_j)=T_g(\overline{\xi}.)
\end{displaymath} holds in the group
$\widehat{K_0}(X_{\mu_n},\mu_n,S)$.
\end{lem}
\begin{proof}
This is \cite[Lemma 5.13]{T1}.
\end{proof}

\begin{proof}(of Theorem~\ref{503})
Denote by $U$ the complement of $X_{\mu_n}$ in $X$, then $j:
U\hookrightarrow X$ is a $\mu_n$-equivariant open subscheme of $X$
whose fixed point set is empty. We consider the following double
complex \begin{displaymath} \xymatrix{
\widetilde{\mathcal{U}}(X_{\mu_n},S)_{\rho} \ar[r]^-{i_*}
\ar[d]^{a} & \widetilde{\mathcal{U}}(X_{\mu_n},S)_{\rho}
\ar[r]^-{j^*} \ar[d]^{a} &
\widetilde{\mathcal{U}}(U_{\mu_n},\emptyset)_{\rho} \ar[r] \ar[d]^{a} & 0 \\
\widehat{K_0}(X_{\mu_n},\mu_n,S)_{\rho} \ar[r]^-{i_*} \ar[d]^\pi &
\widehat{K_0}(X,\mu_n,S)_{\rho} \ar[r]^-{j^*} \ar[d]^\pi &
\widehat{K_0}(U,\mu_n,\emptyset)_{\rho} \ar[r] \ar[d]^\pi & 0 \\
K_0(X_{\mu_n},\mu_n)_{\rho} \ar[r]^-{i_*} \ar[d] &
K_0(X,\mu_n)_{\rho} \ar[r]^-{j^*} \ar[d] & K_0(U,\mu_n)_{\rho}
\ar[r] \ar[d] & 0 \\
0 & 0 & 0 & } \end{displaymath} whose first and second columns are
both exact sequences according to Remark~\ref{404} (iii). For the
third column, $K_0(U,\mu_n)_{\rho}$ is equal to $0$ by
\cite[(2.1.3)]{Th},
$\widetilde{\mathcal{U}}(U_{\mu_n},\emptyset)_{\rho}$ is also
equal to $0$ since $U_{\mu_n}$ is empty. Then from
Remark~\ref{404} (iii) we know that
$\widehat{K_0}(U,\mu_n,\emptyset)_{\rho}$ is equal to $0$. We
claim that $i_*:
\widehat{K_0}(X_{\mu_n},\mu_n,S)_{\rho}\rightarrow
\widehat{K_0}(X,\mu_n,S)_{\rho}$ is surjective. Indeed, for any
element $x\in \widehat{K_0}(X,\mu_n,S)_{\rho}$ we may find an
element $y\in \widehat{K_0}(X_{\mu_n},\mu_n,S)_{\rho}$ such that
$i_*\pi(y)=\pi(x)$ because the third line is exact. This means
$x-i_*(y)$ is in the kernel of $\pi$, so there exists an element
$\alpha\in \widetilde{\mathcal{U}}(X_{\mu_n},S)_{\rho}$ such that
$\alpha=x-i_*(y)$ in $\widehat{K_0}(X,\mu_n,S)_{\rho}$. Set
$\beta=\alpha{\rm Td}_g(\overline{N})$, we get
$i_*(y+\beta)=i_*(y)+\alpha=x$ in
$\widehat{K_0}(X,\mu_n,S)_{\rho}$. Hence, $i_*$ is surjective.

We now prove that the embedding morphism $i_*:
\widehat{K_0}(X_{\mu_n},\mu_n,S)_{\rho}\rightarrow \widehat{K
_0}(X,\mu_n,S)_{\rho}$ is really an isomorphism by constructing
its inverse morphism. Let $\omega$ be an element in
$\widetilde{\mathcal{U}}(X_{\mu_n},S)$, by definition we have
\begin{align*}
\lambda_{-1}^{-1}(\overline{N}_{X/{X_{\mu_n}}}^\vee)\cdot
i^*i_*(\omega)=&\lambda_{-1}^{-1}(\overline{N}_{X/{X_{\mu_n}}}^\vee)\cdot
\omega{\rm Td}_g^{-1}(\overline{N}_{X/{X_g}})\\
=&{\rm
ch}_g(\lambda_{-1}^{-1}(\overline{N}_{X/{X_g}}^\vee))\omega{\rm
Td}_g^{-1}(\overline{N}_{X/{X_g}})\\
=&\omega.
\end{align*}
Let $\overline{\eta}$ be an equivariant hermitian vector bundle on
$X_{\mu_n}$ and assume that $\overline{\xi.}$ is an equivariant
hermitian resolution of $i_*\overline{\eta}$ on $X$ whose metrics
satisfy Bismut assumption (A), then by the definition of the
embedding morphism $i_*$ and Lemma~\ref{504} we have
\begin{align*}
\lambda_{-1}^{-1}(\overline{N}_{X/{X_{\mu_n}}}^\vee)\cdot
i^*i_*(\overline{\eta})=&\lambda_{-1}^{-1}(\overline{N}_{X/{X_{\mu_n}}}^\vee)\cdot
i^*(\sum_k(-1)^k\overline{\xi}_k+T_g(\overline{\xi.}))=\overline{\eta}.
\end{align*}
So the inverse morphism of $i_*$ is of the form
$\lambda_{-1}^{-1}(\overline{N}_{X/{X_{\mu_n}}}^\vee)\cdot i^*$
and we are done.
\end{proof}

\section{A Lefschetz fixed point formula for singular arithmetic schemes
with smooth generic fibres}

\subsection{The statement}
We formulate in this subsection the statement of our main theorem,
a singular Lefschetz fixed point formula for equivariant
arithmetic schemes with smooth generic fibres. Its proof will be
given in next two subsections. Let $f: X\to Y$ be a
$\mu_n$-equivariant morphism between two arithmetic schemes with
smooth generic fibres, which is smooth on the complex numbers.
This morphism $f$ is automatically projective and hence proper,
according to the definition of equivariant arithmetic scheme.
Suppose that $f$ factors through some regular equivariant
arithmetic scheme $Z$. More precisely, our assumption is that
there exist an equivariant closed immersion $i: X\hookrightarrow
Z$ and an equivariant morphism $h: Z\to Y$ such that $f=h\circ i$
and $h$ is also smooth on the complex numbers. Moreover, we shall
assume that the $\mu_n$-action on $Y$ is trivial.

Let $\eta$ be an equivariant coherent sheaf on $X$, then there
exists a bounded complex of equivariant vector bundles which
provides a resolution of $i_*\eta$ on $Z$ because $Z$ is regular.
Since any two equivariant resolution of $i_*\eta$ can be dominated
by a third one, the symbol ${\rm
Tor}_{\mathcal{O}_Z}^k(i_*\eta,\mathcal{O}_{Z_{\mu_n}})$ does make
sense.

We choose arbitrary $\mu_n$-invariant K\"{a}hler forms $\omega^Z$
and $\omega^X$ on $Z(\C)$ and $X(\C)$ respectively, the K\"{a}hler
form $\omega^X$ is not necessarily the K\"{a}hler form induced by
$\omega^Z$. The K\"{a}hler form on $X(\C)$ induced by $\omega^Z$
will be denoted by $\omega^Z_X$. Denote by $N$ the normal bundle
of $i_\C(X(\C))$ in $Z(\C)$, we endow it with the quotient metric
provided that $TX(\C)$ carries the K\"{a}hler metric corresponding
to $\omega^Z_X$. Let $\overline{F}$ be the non-zero degree part of
$\overline{N}$, then by \cite[Exp. VII, Lem. 2.4 and Prop.
2.5]{GBI} for any equivariant hermitian sheaf $\overline{\eta}$ on
$X$ there exists a canonical isomorphism on $X_g$
\begin{displaymath}
{\rm
Tor}_{\mathcal{O}_Z}^k(i_*\eta,\mathcal{O}_{Z_{\mu_n}})_\C\cong
\wedge^kF^\vee\otimes \eta_\C\mid_{X_g} \end{displaymath} which is
equivariant. This means we may endow ${\rm
Tor}_{\mathcal{O}_Z}^k(i_*\eta,\mathcal{O}_{Z_{\mu_n}})_\C$ with a
hermitian metric induced by the metrics on $F$ and $\eta$ so that
it becomes an equivariant hermitian sheaf on $X_{\mu_n}$.
Moreover, we know that the hermitian vector bundle $\overline{F}$
fits the following exact sequence
\begin{displaymath}
(\overline{\mathcal{F}},\omega^X):\quad 0\to
\overline{N}_{X/{X_g}}\to \overline{N}_{Z/{Z_g}}\to
\overline{F}\to 0 \end{displaymath} where $N_{Z/{Z_g}}$ admits the
quotient metric associated to $\omega^Z$ and $N_{X/{X_g}}$ admits
the quotient metric associated to $\omega^X$. Similarly, we shall
denote by $(\overline{\mathcal{F}},\omega^Z_X)$ the hermitian
exact sequence $\overline{\mathcal{F}}$ whose metric on
$N_{X/{X_g}}$ is induced by $\omega^Z_X$.

The push-forward homomorphism from the arithmetic $G_0$-group
$\widehat{G_0}(X,\mu_n,\emptyset)$ to
$\widehat{G_0}(Y,\mu_n,\emptyset)$ with respect to the K\"{a}hler
form $\omega^X$ is denoted by $f_*$ as usual. The push-forward
homomorphism from $\widehat{G_0}(X_{\mu_n},\mu_n,\emptyset)$ to
$\widehat{G_0}(Y,\mu_n,\emptyset)$ with respect to the K\"{a}hler
form $\omega^Z_X$ will be denoted by ${f_{\mu_n}^Z}_*$.

Write $\widetilde{{\rm Td}}(Tf_g,\omega^X,\omega^Z_X)$ for the
secondary characteristic class of the exact sequence
\begin{displaymath}
\xymatrix{0 \ar[r] & (Tf_g,\omega^X) \ar[r]^-{{\rm Id}} &
(Tf_g,\omega^Z_X) \ar[r] & 0 \ar[r] & 0 } \end{displaymath} where
the middle term carries the metric induced by $\omega^Z_X$ and the
sub term carries the metric induced by $\omega^X$. Then the
singular Lefschetz fixed point formula for equivariant arithmetic
schemes with smooth generic fibres can be formulated as follows.

\begin{thm}\label{601}
Let notations and assumptions be as above. Then for any
equivariant hermitian sheaf $\overline{\eta}$ on $X$, the equality
\begin{align*}
f_*(\overline{\eta})=&{f_{\mu_n}^Z}_*(i_{\mu_n}^*(\lambda_{-1}^{-1}(\overline{N}^\vee_{Z/{Z_{\mu_n}}}))\cdot\sum_k(-1)^k{\rm
Tor}_{\mathcal{O}_Z}^k(i_*\overline{\eta},\overline{\mathcal{O}}_{Z_{\mu_n}}))\\
&+\int_{X_g/Y}{\rm Td}(Tf_g,\omega^X){\rm
ch}_g(\overline{\eta})\widetilde{{\rm
Td}}_g(\overline{\mathcal{F}},\omega^X){\rm
Td}_g^{-1}(\overline{F})\\
&-\int_{X_g/Y}{\rm Td}_g(Tf){\rm
ch}_g(\eta)R_g(N_{X/{X_g}})\\
&+\int_{X_g/Y}\widetilde{{\rm Td}}(Tf_g,\omega^X,\omega^Z_X){\rm
ch}_g(\overline{\eta}){\rm Td}_g(\overline{N}_{Z/{Z_g}}){\rm
Td}_g^{-1}(\overline{F})
\end{align*}
holds in the group $\widehat{G_0}(Y,\mu_n,\emptyset)_{\rho}$.
\end{thm}

\begin{rem}\label{602}
This arithmetic Lefschetz fixed point formula was inspired by
\cite[Th\'{e}or\`{e}me 3.5]{Th}.
\end{rem}

\subsection{Equivariant arithmetic $G_0$-theoretic vanishing
theorem} The central actor in the proof of Theorem~\ref{601} is
the following vanishing theorem in equivariant arithmetic
$G_0$-theory, which can be viewed as a translation of
Theorem~\ref{301}.

\begin{thm}\label{603}
Let notations and assumptions be as in last subsection. Let
$\overline{\eta}$ be an equivariant hermitian sheaf on $X$, and
let
\begin{displaymath}
\overline{\Psi}:\quad 0\to \overline{\xi}_m\to \cdots\to
\overline{\xi}_1\to \overline{\xi}_0\to i_*\overline{\eta}\to 0
\end{displaymath} be a resolution of $i_*\overline{\eta}$ by
equivariant hermitian vector bundles on $Z$. Denote by
${h_{\mu_n}}_*$ the push-forward homomorphism from
$\widehat{K_0}(Z_{\mu_n},\mu_n,N^\vee_{g,\R}\setminus\{0\})$ to
$\widehat{G_0}(Y,\mu_n,\emptyset)$ with respect to the K\"{a}hler
form $\omega^Z$. Then the formula
\begin{align*}
&{f_{\mu_n}^Z}_*(\sum(-1)^k{\rm
Tor}^k_{\mathcal{O}_Z}(i_*\overline{\eta},\overline{\mathcal{O}}_{Z_{\mu_n}}))-{h_{\mu_n}}_*(\sum(-1)^k(\overline{\xi}_k\mid_{Z_{\mu_n}}))\\
=&\int_{Z_g/Y}T_g(\overline{\xi.}){\rm
Td}(\overline{Th_g})+\int_{X_g/Y}{\rm Td}(Tf_g){\rm
Td}_g^{-1}(F){\rm
ch}_g(\eta)R(N_g)\\
&+\int_{X_g/Y}{\rm ch}_g(\overline{\eta}){\rm
Td}_g^{-1}(\overline{N})\widetilde{{\rm
Td}}((Tf_g,\omega^Z_X),\overline{Th_g}\mid_{X_g})
\end{align*}
holds in $\widehat{G_0}(Y,\mu_n,\emptyset)$.
\end{thm}
\begin{proof}
Following the same arguments given in the proof of
Lemma~\ref{309}, we may show that the deformation to the normal
cone $W(i)$ admits an equivariant hermitian very ample invertible
sheaf $\overline{\mathcal{L}}$ which is relative to the morphism
$l: W(i)\to Y$. By Theorem~\ref{301} and the fact that
$\mathcal{L}$ is very ample, we conclude that there exists an
integer $k_0>0$ such that for $n\geq k_0$, $\mathcal{L}^{\otimes
n}$ is $l$-acyclic and $\delta(\overline{\Psi}(n)_\C)=0$. Then $l$
factors through an equivariant projective space bundle
$\mathbb{P}(\mathcal{E}^\vee)$ where $\mathcal{E}$ is locally free
of rank $r+1$ on $Y$ and $l_*\mathcal{L}^{\otimes k_0}$ is an
equivariant quotient of $\mathcal{E}$. Denote by $p:
\mathbb{P}(\mathcal{E}^\vee)\to Y$ the canonical projection. On
$P:=\mathbb{P}(\mathcal{E}^\vee)$, we have a canonical exact
sequence
\begin{displaymath}
\mathcal{H}:\quad 0\to \mathcal{O}_P\to
p^*(\mathcal{E}^\vee)(1)\to \cdots \to
p^*(\wedge^{r+1}\mathcal{E}^\vee)(r+1)\to 0. \end{displaymath}
Restricting this sequence to $Z$, we obtain an exact sequence of
exact sequences \begin{displaymath} 0\to \Psi\to \Psi\otimes
h^*(\mathcal{E}^\vee)(1)\to \cdots \to \Psi\otimes
h^*(\wedge^{r+1}\mathcal{E}^\vee)(r+1)\to 0. \end{displaymath}
Endow $\mathcal{E}$ with any $\mu_n(\C)$-invariant hermitian
metric. We claim that the assumption that Theorem~\ref{603} holds
for $\overline{\Psi}\otimes
h^*(\wedge^n\overline{\mathcal{E}}^\vee)(n)$ with $n\geq1$ implies
that it holds for $\overline{\Psi}$. In fact, since $\mathcal{H}$
is an exact sequence of flat modules, for any $k\geq0$ we have the
following exact sequence on $X_{\mu_n}$
\begin{align*}
0\to {\rm
Tor}^k_{\mathcal{O}_Z}(i_*\overline{\eta},\overline{\mathcal{O}}_{Z_{\mu_n}})\to&
{\rm
Tor}^k_{\mathcal{O}_Z}(i_*\overline{\eta},\overline{\mathcal{O}}_{Z_{\mu_n}})\otimes
f_{\mu_n}^*(\overline{\mathcal{E}}^\vee)(1)\to \cdots \\
\to& {\rm
Tor}^k_{\mathcal{O}_Z}(i_*\overline{\eta},\overline{\mathcal{O}}_{Z_{\mu_n}})\otimes
f_{\mu_n}^*(\wedge^{r+1}\overline{\mathcal{E}}^\vee)(r+1)\to 0.
\end{align*}
We compute
\begin{align*}
&{f^Z_{\mu_n}}_*({\rm
Tor}^k_{\mathcal{O}_Z}(i_*\overline{\eta},\overline{\mathcal{O}}_{Z_{\mu_n}}))\\
=&{f^Z_{\mu_n}}_*(-\sum_{j=1}^{r+1}(-1)^j{\rm
Tor}^k_{\mathcal{O}_Z}(i_*\overline{\eta},\overline{\mathcal{O}}_{Z_{\mu_n}})\otimes
f_{\mu_n}^*(\wedge^j\overline{\mathcal{E}}^\vee)(j))\\
&+\int_{X_g/Y}{\rm Td}(Tf_g,\omega^Z_X){\rm
ch}_g(\wedge^k\overline{F}^\vee){\rm
ch}_g(\overline{\eta})(-1)^{r+1}\widetilde{{\rm
ch}}_g(\overline{\mathcal{H}})
\end{align*} and
\begin{align*}
&\sum_{k=0}^m(-1)^k{h_{\mu_n}}_*(\overline{\xi}_k\mid_{Z_{\mu_n}})\\
=&\sum_{k=0}^m(-1)^k{h_{\mu_n}}_*(-\sum_{j=1}^{r+1}(-1)^j\overline{\xi}_k\mid_{Z_{\mu_n}}\otimes
h_{\mu_n}^*(\wedge^j\overline{\mathcal{E}}^\vee)(j))\\
&+\sum_{k=0}^m(-1)^k\int_{Z_g/Y}{\rm Td}(\overline{Th_g}){\rm
ch}_g(\overline{\xi}_k)(-1)^{r+1}\widetilde{{\rm
ch}}_g(\overline{\mathcal{H}}).
\end{align*}
Moreover, we have
\begin{align*}
&\int_{X_g/Y}{\rm Td}(Tf_g){\rm ch}_g({\rm
Tor}^k_{\mathcal{O}_Z}(i_*\eta,\mathcal{O}_{Z_{\mu_n}}))R(N_g)\\
=&\int_{X_g/Y}-\sum_{j=1}^{r+1}(-1)^j {\rm ch}_g({\rm
Tor}^k_{\mathcal{O}_Z}(i_*\eta,\mathcal{O}_{Z_{\mu_n}})\otimes
f_{\mu_n}^*(\wedge^j\mathcal{E}^\vee)(j))R(N_g){\rm Td}(Tf_g)
\end{align*}
and
\begin{align*}
&\int_{Z_g/Y}T_g(\overline{\xi.}){\rm
Td}(\overline{Th_g})\\
=&\int_{Z_g/Y}{\rm
Td}(\overline{Th_g})\{\delta_{X_g}{\rm
Td}_g^{-1}(\overline{N}){\rm
ch}_g(\overline{\eta})(-1)^{r+1}\widetilde{{\rm
ch}}_g(\overline{\mathcal{H}})\\
&-\sum_{k=0}^m(-1)^k{\rm
ch}_g(\overline{\xi}_k)(-1)^{r+1}\widetilde{{\rm
ch}}_g(\overline{\mathcal{H}})-\sum_{j=1}^{r+1}(-1)^jT_g(\overline{\xi.}){\rm
ch}_g(h_{\mu_n}^*(\wedge^j\overline{\mathcal{E}}^\vee)(j))\}
\end{align*}
by the double complex formula of equivariant Bott-Chern singular
currents. At last, we also have
\begin{align*}
&\int_{X_g/Y}{\rm ch}_g(\overline{\eta}){\rm
Td}_g^{-1}(\overline{N})\widetilde{{\rm
Td}}((Tf_g,\omega^Z_X),\overline{Th_g}\mid_{X_g})\\
=&\int_{X_g/Y}\{{\rm dd}^c(-1)^{r+1}\widetilde{{\rm
ch}}_g(\overline{\mathcal{H}}){\rm
ch}_g(\overline{\eta})-\sum_{j=1}^{r+1}(-1)^j{\rm
ch}_g(\overline{\eta}\otimes
f^*(\wedge^j\overline{\mathcal{E}}^\vee)(j))\}\\
&\qquad\qquad\qquad\qquad\qquad\qquad\qquad\quad\cdot{\rm
Td}_g^{-1}(\overline{N})\widetilde{{\rm
Td}}((Tf_g,\omega^Z_X),\overline{Th_g}\mid_{X_g})\\
=&-\int_{X_g/Y}(-1)^{r+1}\widetilde{{\rm
ch}}_g(\overline{\mathcal{H}}){\rm ch}_g(\overline{\eta})\cdot\{{\rm
Td}_g^{-1}(\overline{N}){\rm Td}(\overline{Th_g})\\
&\qquad\qquad\qquad\qquad\qquad\qquad\qquad\quad-{\rm
Td}(Tf_g,\omega^Z_X){\rm
Td}_g^{-1}(\overline{F})\}\\
&-\int_{X_g/Y}\sum_{j=1}^{r+1}(-1)^j{\rm
ch}_g(\overline{\eta}\otimes
f^*(\wedge^j\overline{\mathcal{E}}^\vee)(j)){\rm
Td}_g^{-1}(\overline{N})\widetilde{{\rm
Td}}((Tf_g,\omega^Z_X),\overline{Th_g}\mid_{X_g}).
\end{align*}
Gathering all these computations above and using our assumption,
we get
\begin{align*}
&{f^Z_{\mu_n}}_*(\sum(-1)^k{\rm
Tor}^k_{\mathcal{O}_Z}(i_*\overline{\eta},\overline{\mathcal{O}}_{Z_{\mu_n}}))-{h_{\mu_n}}_*(\sum(-1)^k\overline{\xi}_k\mid_{Z_{\mu_n}})\\
&-\int_{Z_g/Y}T_g(\overline{\xi.}){\rm
Td}(\overline{Th_g})-\int_{X_g/Y}{\rm Td}(Tf_g){\rm Td}_g^{-1}(F){\rm
ch}_g(\eta)R(N_g)\\
&\qquad\qquad\qquad\quad-\int_{X_g/Y}{\rm ch}_g(\overline{\eta}){\rm
Td}_g^{-1}(\overline{N})\widetilde{{\rm
Td}}((Tf_g,\omega^Z_X),\overline{Th_g}\mid_{X_g})\\
=&\int_{X_g/Y}{\rm Td}(Tf_g,\omega^Z_X){\rm
Td}_g^{-1}(\overline{F}){\rm
ch}_g(\overline{\eta})(-1)^{r+1}\widetilde{{\rm
ch}}_g(\overline{\mathcal{H}})\\
&-\sum_{k=0}^m(-1)^k\int_{Z_g/Y}{\rm Td}(\overline{Th_g}){\rm
ch}_g(\overline{\xi}_k)(-1)^{r+1}\widetilde{{\rm
ch}}_g(\overline{\mathcal{H}})\\
&-\int_{X_g/Y}{\rm Td}(\overline{Th_g}){\rm
Td}_g^{-1}(\overline{N}){\rm
ch}_g(\overline{\eta})(-1)^{r+1}\widetilde{{\rm
ch}}_g(\overline{\mathcal{H}})\\
&+\sum_{k=0}^m(-1)^k\int_{Z_g/Y}{\rm Td}(\overline{Th_g}){\rm
ch}_g(\overline{\xi}_k)(-1)^{r+1}\widetilde{{\rm
ch}}_g(\overline{\mathcal{H}})\\
&+\int_{X_g/Y}(-1)^{r+1}\widetilde{{\rm
ch}}_g(\overline{\mathcal{H}}){\rm ch}_g(\overline{\eta})\{{\rm
Td}_g^{-1}(\overline{N}){\rm Td}(\overline{Th_g})\\
&\qquad\qquad\qquad\qquad\qquad\qquad\qquad\qquad\quad-{\rm
Td}(Tf_g,\omega^Z_X){\rm Td}_g^{-1}(\overline{F})\}
\end{align*}
which vanishes. This ends the proof of our claim.

By the construction of the projective space bundle $P$, we have
already known that $\delta(\overline{\Psi}(n)_\C)$ vanishes from
$n=1$ to $n=r+1$. Moreover, according to the projection formula of
higher direct images, the operation of tensoring with the element
$l^*(\wedge^n\overline{\mathcal{E}}^\vee)$ doesn't change the
property of $l$-acyclicity. Hence we also have
$\delta(\overline{\Psi}\otimes
h^*(\wedge^n\overline{\mathcal{E}}^\vee)(n)_\C)=0$. By the
generating relations and the definition of push-forward morphisms
of arithmetic $G_0$-groups, this is equivalent to say that
Theorem~\ref{603} holds for $\overline{\Psi}\otimes
h^*(\wedge^n\overline{\mathcal{E}}^\vee)(n)$. Therefore the
equality in the statement of this theorem follows from our claim
before.
\end{proof}

\begin{cor}\label{604}
Let notations and assumptions be as in Theorem~\ref{603}, and let
$x$ be any element in $\widehat{K_0}(Z,\mu_n,\emptyset)_{\rho}$.
Then the formula
\begin{align*}
&{f_{\mu_n}^Z}_*(i^*x\mid_{X_{\mu_n}}\cdot\sum(-1)^k{\rm
Tor}^k_{\mathcal{O}_Z}(i_*\overline{\eta},\overline{\mathcal{O}}_{Z_{\mu_n}}))\\
&\qquad\qquad\qquad\qquad\qquad-{h_{\mu_n}}_*(x\mid_{Z_{\mu_n}}\cdot\sum(-1)^k(\overline{\xi}_k\mid_{Z_{\mu_n}}))\\
=&\int_{Z_g/Y}T_g(\overline{\xi.}){\rm Td}(\overline{Th_g}){\rm
ch}_g(x)\\
&+\int_{X_g/Y}{\rm Td}(Tf_g){\rm Td}_g^{-1}(F){\rm
ch}_g(\eta)R(N_g){\rm ch}_g(i^*x)\\
&+\int_{X_g/Y}{\rm ch}_g(\overline{\eta}){\rm
Td}_g^{-1}(\overline{N}){\rm ch}_g(i^*x)\widetilde{{\rm
Td}}((Tf_g,\omega^Z_X),\overline{Th_g}\mid_{X_g})\\
\end{align*}
holds in $\widehat{G_0}(Y,\mu_n,\emptyset)_{\rho}$.
\end{cor}
\begin{proof}
If $x=\overline{E}$ is an equivariant hermitian vector bundle on
$Z$, then $\overline{\xi.}\otimes \overline{E}$ provides a
resolution of $i_*(\overline{\eta}\otimes i^*\overline{E})$. Hence
the formula follows from Theorem~\ref{603} in this case. If
$x=\alpha$ is represented by some smooth form, then
\begin{align*}
&{f_{\mu_n}^Z}_*(i^*x\mid_{X_{\mu_n}}\cdot\sum(-1)^k{\rm
Tor}^k_{\mathcal{O}_Z}(i_*\overline{\eta},\overline{\mathcal{O}}_{Z_{\mu_n}}))\\
=&\int_{Z_g/Y}{\rm Td}(Tf_g,\omega^Z_X){\rm
Td}_g^{-1}(\overline{F}){\rm
ch}_g(\overline{\eta})\delta_{X_g}\alpha
\end{align*}
and
\begin{displaymath}
{h_{\mu_n}}_*(x\mid_{Z_{\mu_n}}\cdot\sum(-1)^k(\overline{\xi}_k\mid_{Z_{\mu_n}}))=\int_{Z_g/Y}{\rm
Td}(\overline{Th_g})\alpha\sum(-1)^k{\rm ch}_g(\overline{\xi}_k).
\end{displaymath} Moreover, by the definition of ${\rm ch}_g(x)$ we have
\begin{align*}
\int_{Z_g/Y}T_g(\overline{\xi.}){\rm Td}(\overline{Th_g}){\rm
ch}_g(x)=&\int_{Z_g/Y}{\rm ch}_g(\overline{\eta}){\rm
Td}_g^{-1}(\overline{N})\delta_{X_g}{\rm
Td}(\overline{Th_g})\alpha\\
&-\int_{Z_g/Y}\sum(-1)^k{\rm ch}_g(\overline{\xi}_k){\rm
Td}(\overline{Th_g})\alpha
\end{align*} and
\begin{displaymath}
\int_{X_g/Y}{\rm Td}(Tf_g){\rm Td}_g^{-1}(F){\rm
ch}_g(\eta)R(N_g){\rm ch}_g(i^*x)=0. \end{displaymath} Finally,
using the definition of $\widetilde{{\rm Td}}$ we compute
\begin{align*}
&\int_{X_g/Y}{\rm ch}_g(\overline{\eta}){\rm
Td}_g^{-1}(\overline{N}){\rm ch}_g(i^*x)\widetilde{{\rm
Td}}((Tf_g,\omega^Z_X),\overline{Th_g}\mid_{X_g})\\
=&\int_{Z_g/Y}{\rm Td}(Tf_g,\omega^Z_X){\rm
Td}_g^{-1}(\overline{F}){\rm
ch}_g(\overline{\eta})\delta_{X_g}\alpha\\
&\qquad\qquad\qquad\qquad\qquad-\int_{Z_g/Y}{\rm
ch}_g(\overline{\eta}){\rm
Td}_g^{-1}(\overline{N})\delta_{X_g}{\rm
Td}(\overline{Th_g})\alpha.
\end{align*} Gathering all
computations above, we know that the formula still holds for $x$
which is represented by smooth form. Since both two sides are
additive, we are done.
\end{proof}

\begin{cor}\label{605}
Let notations and assumptions be as in Theorem~\ref{603}, and let
$y$ be any element in
$\widehat{K_0}(Z_{\mu_n},\mu_n,\emptyset)_{\rho}$. Then the
formula
\begin{align*}
&{f_{\mu_n}^Z}_*(i_{\mu_n}^*y\cdot\sum(-1)^k{\rm
Tor}^k_{\mathcal{O}_Z}(i_*\overline{\eta},\overline{\mathcal{O}}_{Z_{\mu_n}}))-{h_{\mu_n}}_*(y\cdot\sum(-1)^k(\overline{\xi}_k\mid_{Z_{\mu_n}}))\\
=&\int_{Z_g/Y}T_g(\overline{\xi.}){\rm Td}(\overline{Th_g}){\rm
ch}_g(y)\\
&+\int_{X_g/Y}{\rm Td}(Tf_g){\rm Td}_g^{-1}(F){\rm
ch}_g(\eta)R(N_g){\rm ch}_g(i_{\mu_n}^*y)\\
&+\int_{X_g/Y}{\rm ch}_g(\overline{\eta}){\rm
Td}_g^{-1}(\overline{N}){\rm ch}_g(i_{\mu_n}^*y)\widetilde{{\rm
Td}}((Tf_g,\omega^Z_X),\overline{Th_g}\mid_{X_g})
\end{align*}
holds in $\widehat{G_0}(Y,\mu_n,\emptyset)_{\rho}$.
\end{cor}
\begin{proof}
Provided Corollary~\ref{604}, it is enough to prove that for any
$y\in \widehat{K_0}(Z_{\mu_n},\mu_n,\emptyset)_{\rho}$ there
exists an element $x\in \widehat{K_0}(Z,\mu_n,\emptyset)_{\rho}$
such that $i_Z^*x=y$. Here $i_Z$ stands for the inclusion
$Z_{\mu_n}\hookrightarrow Z$. Actually, set
$x={i_Z}_*(\lambda_{-1}^{-1}(\overline{N}^\vee_{Z/{Z_{\mu_n}}})\cdot
y)$, we have
\begin{displaymath}
i_Z^*x=i_Z^*{i_Z}_*(\lambda_{-1}^{-1}(\overline{N}^\vee_{Z/{Z_{\mu_n}}})\cdot
y)=\lambda_{-1}(\overline{N}^\vee_{Z/{Z_{\mu_n}}})\cdot\lambda_{-1}^{-1}(\overline{N}^\vee_{Z/{Z_{\mu_n}}})\cdot
y=y. \end{displaymath} This follows from our arithmetic
concentration theorem.
\end{proof}

\subsection{Proof of the fixed point formula}
In this subsection, we provide a complete proof of
Theorem~\ref{601} the singular Lefschetz fixed point formula.
Before that, we need to translate Bismut-Ma's immersion formula to
an arithmetic $G_0$-theoretic version. That's the following.

\begin{thm}\label{606}
Let notations and assumptions be as in Section 6.1. Assume that
$\overline{\eta}$ is an equivariant hermitian sheaf on $X$ and
$\overline{\xi.}$ is a bounded complex of equivariant hermitian
vector bundles providing a resolution of $i_*\overline{\eta}$ on
$Z$ whose metrics satisfy Bismut assumption (A). Then the equality
\begin{align*}
f^Z_*(\overline{\eta})-\sum_{j=0}^m(-1)^jh_*(\overline{\xi}_j)=&\int_{X_g/Y}{\rm
ch}_g(\eta)R_g(N){\rm
Td}_g(Tf)\\
&+\int_{Z_g/Y}T_g(\overline{\xi}.){\rm
Td}_g(\overline{Th})\\
&+\int_{X_g/Y}{\rm ch}_g(\overline{\eta})\widetilde{{\rm
Td}}_g((Tf,\omega^Z_X),\overline{Th}\mid_X){\rm
Td}_g^{-1}(\overline{N})
\end{align*}
holds in $\widehat{G_0}(Y,\mu_n,\emptyset)$.
\end{thm}
\begin{proof}
We first suppose that $\eta$ and $\xi.$ are all acyclic, then the
verification follows rather directly from the generating relations
of arithmetic $G_0$-theory. In fact
\begin{align*}
f^Z_*(\overline{\eta})-\sum_{j=0}^m(-1)^jh_*(\overline{\xi}_j)=&
\overline{f_*\eta}-T_g(\omega^Z_X,h^{\eta})-(\sum_{j=0}^m(-1)^j(\overline{h_*{\xi_j}}-T_g(\omega^Z,h^{\xi_j})))\\
=&\widetilde{{\rm
ch}}_g(h_*\overline{\Xi})-T_g(\omega^Z_X,h^{\eta})+\sum_{j=0}^m(-1)^jT_g(\omega^Z,h^{\xi_j}).
\end{align*}
And the right-hand side of the last equality is exactly the
left-hand side of Bismut-Ma's immersion formula. We emphasize
again that to simplify the right-hand side of Bismut-Ma's
immersion formula, we have used an Atiyah-Segal-Singer type
formula for immersion \begin{displaymath} {i_g}_*({\rm
Td}_g^{-1}(N){\rm ch}_g(x))={\rm ch}_g(i_*(x)). \end{displaymath}
To remove the condition of acyclicity, one can use the argument
which is essentially the same as in the proof of
Theorem~\ref{603}. Since it doesn't use any new techniques, we
omit it here.
\end{proof}

\begin{defn}\label{607}
The inclusion $i: X\hookrightarrow Z$ induces an embedding
morphism \begin{displaymath} i_*:\quad
\widehat{G_0}(X,\mu_n,\emptyset)\rightarrow
\widehat{K_0}(Z,\mu_n,N^\vee_{g,\R}\setminus\{0\})
\end{displaymath} which is defined as follows.

(i). For every $\mu_n$-equivariant hermitian sheaf
$\overline{\eta}$ on $X$, suppose that $\overline{\xi.}$ is a
resolution of $i_*\overline{\eta}$ on $Z$ whose metrics satisfy
Bismut assumption (A),
$i_*[\overline{\eta}]=\sum_k(-1)^k[\overline{\xi}_k]+T_g(\overline{\xi.})$.

(ii). For every $\alpha\in
\widetilde{\mathcal{U}}(X_{\mu_n},\emptyset)$,
$i_*\alpha=\alpha{\rm Td}_g^{-1}(\overline{N})\delta_{X_g}$.
\end{defn}

\begin{rem}\label{608}
Similar to Theorem~\ref{502} and Lemma~\ref{506}, one can prove
that the embedding morphism is a well-defined homomorphism of
$R(\mu_n)$-modules.
\end{rem}

\begin{proof}(of Theorem~\ref{601})
We first prove that this fixed point formula holds when $\omega^X$
is equal to $\omega^Z_X$, namely the K\"{a}hler metric on $X(\C)$
is induced by the K\"{a}hler metric on $Z(\C)$. By
Theorem~\ref{606} and Definition~\ref{607}, we have the following
equality
\begin{align*}
f^Z_*(\overline{\eta})=&h_*i_*(\overline{\eta})+\int_{X_g/Y}{\rm
ch}_g(\eta)R_g(N){\rm Td}_g(Tf)\\
&+\int_{X_g/Y}{\rm
ch}_g(\overline{\eta})\widetilde{{\rm
Td}}_g((Tf,\omega^Z_X),\overline{Th}\mid_X){\rm
Td}_g^{-1}(\overline{N})
\end{align*}
which holds in $\widehat{G_0}(Y,\mu_n,\emptyset)$. Now we claim
that for any element $y\in
\widehat{K_0}(Z_{\mu_n},\mu_n,N^\vee_{g,\R}\setminus\{0\})_{\rho}$,
we have
\begin{displaymath}
{h_{\mu_n}}_*(y)-h_*{i_Z}_*(y)={h_{\mu_n}}_*(y\cdot
R_g(N_{Z/{Z_{\mu_n}}})). \end{displaymath} Since all morphisms are
homomorphisms of $R(\mu_n)$-modules, we can only consider the
generators of
$\widehat{K_0}(Z_{\mu_n},\mu_n,N^\vee_{g,\R}\setminus\{0\})$.
Indeed, by applying Theorem~\ref{606} to the closed immersion
$i_Z$, for any equivariant hermitian vector bundle $\overline{E}$
on $Z_{\mu_n}$ we have
\begin{align*}
{h_{\mu_n}}_*(\overline{E})-h_*{i_Z}_*(\overline{E})=&\int_{Z_g/Y}{\rm
ch}_g(E)R_g(N_{Z/{Z_{\mu_n}}}){\rm Td}_g(Th_g)\\
=&{h_{\mu_n}}_*({\rm
ch}_g(E)R_g(N_{Z/{Z_{\mu_n}}})).
\end{align*}
The first equality holds because the exact sequence
\begin{displaymath}
0\rightarrow \overline{Th_g}\rightarrow
\overline{Th}\mid_{Z_g}\rightarrow
\overline{N}_{Z/{Z_g}}\rightarrow 0 \end{displaymath} is
orthogonally split on $Z_g$ so that $\widetilde{{\rm
Td}}_g(\overline{Th_g},\overline{Th}\mid_{Z_g})=0$. The second
equality follows from \cite[Lemma 7.3]{KR1} and the fact that
${\rm ch}_g(E)R_g(N_{Z/{Z_{\mu_n}}})$ is ${\rm dd}^c$-closed. On
the other hand, let $\alpha$ be an element in
$\widetilde{\mathcal{U}}(Z_{\mu_n},N^\vee_{g,\R}\setminus\{0\})$,
we have
\begin{align*}
{h_{\mu_n}}_*(\alpha)-h_*{i_Z}_*(\alpha)=&\int_{Z_g/Y}\alpha{\rm
Td}_g(\overline{Th_g})-\int_{Z_g/Y}\alpha{\rm
Td}_g^{-1}(\overline{N}_{Z/{Z_{\mu_n}}}){\rm
Td}_g(\overline{Th})\\
=&\int_{Z_g/Y}\alpha{\rm
Td}_g^{-1}(\overline{N}_{Z/{Z_{\mu_n}}}){\rm dd}^c\widetilde{{\rm
Td}}_g(\overline{Th_g},\overline{Th}\mid_{Z_g})=0.
\end{align*}
Combing these two computations, we get our claim by linear
extension.

Now using arithmetic concentration theorem, we compute
\begin{align*}
h_*i_*(\overline{\eta})=&h_*{i_Z}_*{i_Z}_*^{-1}i_*(\overline{\eta})\\
=&{h_{\mu_n}}_*({i_Z}_*^{-1}i_*(\overline{\eta})\cdot(1-R_g(N_{Z/{Z_{\mu_n}}})))\\
=&{h_{\mu_n}}_*(\lambda_{-1}^{-1}(\overline{N}^\vee_{Z/{Z_{\mu_n}}})\cdot{i_Z}^*i_*(\overline{\eta})\cdot(1-R_g(N_{Z/{Z_{\mu_n}}})))\\
=&{h_{\mu_n}}_*(\lambda_{-1}^{-1}(\overline{N}^\vee_{Z/{Z_{\mu_n}}})\cdot\{\sum_k(-1)^k(\overline{\xi}_k\mid_{Z_{\mu_n}})\\
&\qquad\qquad\qquad\qquad\qquad\qquad+T_g(\overline{\xi.})\}\cdot(1-R_g(N_{Z/{Z_{\mu_n}}})))\\
=&{h_{\mu_n}}_*(\lambda_{-1}^{-1}(\overline{N}^\vee_{Z/{Z_{\mu_n}}})\cdot\sum_k(-1)^k(\overline{\xi}_k\mid_{Z_{\mu_n}}))\\
&+{h_{\mu_n}}_*({\rm
Td}_g(\overline{N}_{Z/{Z_{\mu_n}}})T_g(\overline{\xi.}))\\
&-{h_{\mu_n}}_*({\rm Td}_g(N_{Z/{Z_{\mu_n}}})\sum_k(-1)^k{\rm
ch}_g(\xi_k)R_g(N_{Z/{Z_{\mu_n}}})).
\end{align*} According to
Corollary~\ref{605}, by setting
$y=\lambda_{-1}^{-1}(\overline{N}^\vee_{Z/{Z_{\mu_n}}})$, we
compute
\begin{align*}
&{h_{\mu_n}}_*(\lambda_{-1}^{-1}(\overline{N}^\vee_{Z/{Z_{\mu_n}}})\cdot\sum_k(-1)^k(\overline{\xi}_k\mid_{Z_{\mu_n}}))\\
=&{f_{\mu_n}^Z}_*(i_{\mu_n}^*(\lambda_{-1}^{-1}(\overline{N}^\vee_{Z/{Z_{\mu_n}}}))\cdot\sum(-1)^k{\rm
Tor}^k_{\mathcal{O}_Z}(i_*\overline{\eta},\overline{\mathcal{O}}_{Z_{\mu_n}}))\\
&-\int_{Z_g/Y}T_g(\overline{\xi.}){\rm Td}(\overline{Th_g}){\rm
Td}_g(\overline{N}_{Z/{Z_g}})\\
&-\int_{X_g/Y}{\rm Td}(Tf_g){\rm
Td}_g^{-1}(F){\rm ch}_g(\eta)R(N_g){\rm
Td}_g(N_{Z/{Z_g}})\\
&-\int_{X_g/Y}{\rm ch}_g(\overline{\eta}){\rm
Td}_g^{-1}(\overline{N}){\rm
Td}_g(\overline{N}_{Z/{Z_g}})\widetilde{{\rm
Td}}((Tf_g,\omega^Z_X),\overline{Th_g}\mid_{X_g})\\
=&{f_{\mu_n}^Z}_*(i_{\mu_n}^*(\lambda_{-1}^{-1}(\overline{N}^\vee_{Z/{Z_{\mu_n}}}))\cdot\sum(-1)^k{\rm
Tor}^k_{\mathcal{O}_Z}(i_*\overline{\eta},\overline{\mathcal{O}}_{Z_{\mu_n}}))\\
&-\int_{Z_g/Y}T_g(\overline{\xi.}){\rm
Td}_g(\overline{Th})-\int_{X_g/Y}{\rm Td}_g(Tf){\rm
ch}_g(\eta)R(N_g)\\
&-\int_{X_g/Y}{\rm ch}_g(\overline{\eta}){\rm
Td}_g^{-1}(\overline{N}){\rm
Td}_g(\overline{N}_{Z/{Z_g}})\widetilde{{\rm
Td}}((Tf_g,\omega^Z_X),\overline{Th_g}\mid_{X_g}).
\end{align*}
Here we have used various relations of character forms or
characteristic classes arising from the following double complex
\begin{displaymath}
\xymatrix{ & 0 \ar[d] & 0 \ar[d] & 0 \ar[d] & \\
0 \ar[r] & (Tf_g,\omega^Z_X) \ar[r] \ar[d] & \overline{Th_g}
\ar[d]
\ar[r] & \overline{N}_g \ar[r] \ar[d] & 0 \\
0 \ar[r] & (Tf,\omega^Z_X) \ar[r] \ar[d] & \overline{Th} \ar[d]
\ar[r] & \overline{N} \ar[r] \ar[d] & 0 \\
0 \ar[r] & (N_{X/{X_g}},\omega^Z_X) \ar[r] \ar[d] &
\overline{N}_{Z/{Z_g}} \ar[r] \ar[d] & \overline{F} \ar[d] \ar[r]
& 0 \\
& 0 & 0  & 0 &} \end{displaymath} whose columns are all
orthogonally split. Also, for this double complex, one may use
Example~\ref{s204} (iv) to compute that
\begin{align}
\widetilde{{\rm
Td}}_g((Tf,\omega^Z_X),\overline{Th}\mid_X)=&\widetilde{{\rm
Td}}_g(\overline{\mathcal{F}},\omega^Z_X){\rm
Td}(\overline{N}_g){\rm Td}(Tf_g,\omega^Z_X)\notag\\
&+\widetilde{{\rm
Td}}((Tf_g,\omega^Z_X),\overline{Th_g}\mid_{X_g}){\rm
Td}_g(\overline{N}_{Z/{Z_g}}). \label{eq10}
\end{align}
We deduce from (\ref{eq10}) that
\begin{align}
&\int_{X_g/Y}{\rm ch}_g(\overline{\eta})\widetilde{{\rm
Td}}_g((Tf,\omega^Z_X),\overline{Th}\mid_X){\rm
Td}_g^{-1}(\overline{N})\notag\\
=&\int_{X_g/Y}{\rm
ch}_g(\overline{\eta})\widetilde{{\rm
Td}}((Tf_g,\omega^Z_X),\overline{Th_g}\mid_{X_g}){\rm
Td}_g^{-1}(\overline{N}){\rm
Td}_g(\overline{N}_{Z/{Z_g}})\notag\\
&+\int_{X_g/Y}{\rm
ch}_g(\overline{\eta})\widetilde{{\rm
Td}}_g(\overline{\mathcal{F}}){\rm Td}_g^{-1}(\overline{F}){\rm
Td}(Tf_g,\omega^Z_X).\label{eq11}
\end{align}

Moreover, we have
\begin{align}
{h_{\mu_n}}_*({\rm
Td}_g(\overline{N}_{Z/{Z_{\mu_n}}})T_g(\overline{\xi.}))=&\int_{Z_g/Y}T_g(\overline{\xi.}){\rm
Td}(\overline{Th_g}){\rm
Td}_g(\overline{N}_{Z/{Z_g}})\notag\\
=&\int_{Z_g/Y}T_g(\overline{\xi.}){\rm
Td}_g(\overline{Th})\label{eq12}
\end{align} and
\begin{align}
&{h_{\mu_n}}_*({\rm Td}_g(N_{Z/{Z_{\mu_n}}})\sum_k(-1)^k{\rm
ch}_g(\xi_k)R_g(N_{Z/{Z_{\mu_n}}}))\notag\\
=&\int_{Z_g/Y}{\rm Td}_g(N_{Z/{Z_{\mu_n}}})\delta_{X_g}{\rm
ch}_g(\eta){\rm Td}_g^{-1}(N)R_g(N_{Z/{Z_g}}){\rm
Td}(Th_g)\notag\\
=&\int_{X_g/Y}{\rm Td}_g(N_{X/{X_g}}){\rm Td}_g(F){\rm
ch}_g(\eta){\rm Td}_g^{-1}(N)\notag\\
&\qquad\cdot[R_g(N_{X/{X_g}})+R_g(N)-R(N_g)]{\rm
Td}(Tf_g){\rm Td}(N_g)\notag\\
=&\int_{X_g/Y}{\rm Td}_g(Tf){\rm
ch}_g(\eta)[R_g(N_{X/{X_g}})+R_g(N)-R(N_g)].\label{eq13}
\end{align}

Gathering (\ref{eq11}), (\ref{eq12}) and (\ref{eq13}) we finally
get
\begin{align*}
f^Z_*(\overline{\eta})=&{f_{\mu_n}^Z}_*(i_{\mu_n}^*(\lambda_{-1}^{-1}(\overline{N}^\vee_{Z/{Z_{\mu_n}}}))\cdot\sum_k(-1)^k{\rm
Tor}_{\mathcal{O}_Z}^k(i_*\overline{\eta},\overline{\mathcal{O}}_{Z_{\mu_n}}))\\
&+\int_{X_g/Y}{\rm Td}(Tf_g,\omega^Z_X){\rm
ch}_g(\overline{\eta})\widetilde{{\rm
Td}}_g(\overline{\mathcal{F}},\omega^Z_X){\rm
Td}_g^{-1}(\overline{F})\\
&-\int_{X_g/Y}{\rm Td}_g(Tf){\rm
ch}_g(\eta)R_g(N_{X/{X_g}})
\end{align*}
which completes the
proof of Theorem~\ref{601} in the case where the K\"{a}hler metric
on $X(\C)$ is induced by the K\"{a}hler metric on $Z(\C)$.

In general, in analogy with the notation $\widetilde{{\rm
Td}}(Tf_g,\omega^X,\omega^Z_X)$, we write $\widetilde{{\rm
Td}}_g(N_{X/{X_g}},\omega^X,\omega^Z_X)$ for the secondary
characteristic class of the exact sequence
\begin{displaymath}
\xymatrix{0 \ar[r] & N_{X/{X_g}} \ar[r]^-{{\rm Id}} & N_{X/{X_g}}
\ar[r] & 0 \ar[r] & 0 }\end{displaymath} where the middle term
carries the metric induced by $\omega^Z_X$ and the sub term
carries the metric induced by $\omega^X$. Similarly, we have the
notation $\widetilde{{\rm Td}}_g(Tf,\omega^X,\omega^Z_X)$. Then by
applying the argument in the proof of (\ref{eq10}) to the double
complex
\begin{displaymath}
\xymatrix{ & 0 \ar[d] & 0 \ar[d] & 0 \ar[d] & \\
0 \ar[r] & (Tf_g,\omega^X) \ar[r] \ar[d] & (Tf_g,\omega^Z_X)
\ar[d]
\ar[r] & 0 \ar[r] \ar[d] & 0 \\
0 \ar[r] & (Tf,\omega^X) \ar[r] \ar[d] & (Tf,\omega^Z_X) \ar[d]
\ar[r] & 0 \ar[r] \ar[d] & 0 \\
0 \ar[r] & (N_{X/{X_g}},\omega^X) \ar[r] \ar[d] &
(N_{X/{X_g}},\omega^Z_X) \ar[r] \ar[d] & 0 \ar[d] \ar[r]
& 0 \\
& 0 & 0  & 0 &} \end{displaymath} We get
\begin{align*}
\widetilde{{\rm
Td}}_g(Tf,\omega^X,\omega^Z_X)=&\widetilde{{\rm
Td}}(Tf_g,\omega^X,\omega^Z_X){\rm
Td}_g(N_{X/{X_g}},\omega^Z_X)\\
&+\widetilde{{\rm
Td}}_g(N_{X/{X_g}},\omega^X,\omega^Z_X){\rm Td}(Tf_g,\omega^X).
\end{align*}
Moreover, by \cite[Proposition 2.8]{T2}, we obtain that
\begin{displaymath}
\widetilde{{\rm
Td}}_g(\overline{\mathcal{F}},\omega^X)=\widetilde{{\rm
Td}}_g(\overline{\mathcal{F}},\omega^Z_X)+\widetilde{{\rm
Td}}_g(N_{X/{X_g}},\omega^X,\omega^Z_X){\rm Td}_g(\overline{F}).
\end{displaymath}

With these two comparison formulae, we can compute
\begin{align*}
&\int_{X_g/Y}{\rm Td}(Tf_g,\omega^X){\rm
ch}_g(\overline{\eta})\widetilde{{\rm
Td}}_g(\overline{\mathcal{F}},\omega^X){\rm
Td}_g^{-1}(\overline{F})\\
&\qquad\qquad\qquad-\int_{X_g/Y}{\rm
Td}(Tf_g,\omega^Z_X){\rm ch}_g(\overline{\eta})\widetilde{{\rm
Td}}_g(\overline{\mathcal{F}},\omega^Z_X){\rm
Td}_g^{-1}(\overline{F})\\
=&\int_{X_g/Y}{\rm Td}(Tf_g,\omega^X){\rm
ch}_g(\overline{\eta})\widetilde{{\rm
Td}}_g(\overline{\mathcal{F}},\omega^X){\rm
Td}_g^{-1}(\overline{F})\\
&-\int_{X_g/Y}{\rm Td}(Tf_g,\omega^X){\rm
ch}_g(\overline{\eta})\widetilde{{\rm
Td}}_g(\overline{\mathcal{F}},\omega^Z_X){\rm
Td}_g^{-1}(\overline{F})\\
&+\int_{X_g/Y}{\rm Td}(Tf_g,\omega^X){\rm
ch}_g(\overline{\eta})\widetilde{{\rm
Td}}_g(\overline{\mathcal{F}},\omega^Z_X){\rm
Td}_g^{-1}(\overline{F})\\
&-\int_{X_g/Y}{\rm Td}(Tf_g,\omega^Z_X){\rm
ch}_g(\overline{\eta})\widetilde{{\rm
Td}}_g(\overline{\mathcal{F}},\omega^Z_X){\rm
Td}_g^{-1}(\overline{F})\\
=&\int_{X_g/Y}\widetilde{{\rm Td}}(Tf_g,\omega^X,\omega^Z_X){\rm
ch}_g(\overline{\eta})\\
&\quad\cdot[{\rm Td}_g(N_{X/{X_g}},\omega^Z_X){\rm
Td}_g(\overline{F})-{\rm Td}_g(\overline{N}_{Z/{Z_g}})]{\rm
Td}_g^{-1}(\overline{F})\\
&+\int_{X_g/Y}{\rm Td}(Tf_g,\omega^X){\rm
ch}_g(\overline{\eta})\widetilde{{\rm
Td}}_g(N_{X/{X_g}},\omega^X,\omega^Z_X)\\
=&\int_{X_g/Y}{\rm ch}_g(\overline{\eta})\widetilde{{\rm
Td}}_g(Tf,\omega^X,\omega^Z_X)\\
&-\int_{X_g/Y}\widetilde{{\rm
Td}}(Tf_g,\omega^X,\omega^Z_X){\rm ch}_g(\overline{\eta}){\rm
Td}_g(\overline{N}_{Z/{Z_g}}){\rm Td}_g^{-1}(\overline{F}).
\end{align*}

At last, using \cite[Lemma 7.3]{KR1}, we get the equality
\begin{displaymath}
f_*(\overline{\eta})-f^Z_*(\overline{\eta})=\int_{X_g/Y}{\rm
ch}_g(\overline{\eta})\widetilde{{\rm
Td}}_g(Tf,\omega^X,\omega^Z_X). \end{displaymath} Together with
the fact that the other two terms have nothing to do with the
choice of the metric $\omega^X$, we finally obtain that
\begin{align*}
f_*(\overline{\eta})=&{f_{\mu_n}^Z}_*(i_{\mu_n}^*(\lambda_{-1}^{-1}(\overline{N}^\vee_{Z/{Z_{\mu_n}}}))\cdot\sum_k(-1)^k{\rm
Tor}_{\mathcal{O}_Z}^k(i_*\overline{\eta},\overline{\mathcal{O}}_{Z_{\mu_n}}))\\
&+\int_{X_g/Y}{\rm Td}(Tf_g,\omega^X){\rm
ch}_g(\overline{\eta})\widetilde{{\rm
Td}}_g(\overline{\mathcal{F}},\omega^X){\rm
Td}_g^{-1}(\overline{F})\\
&-\int_{X_g/Y}{\rm Td}_g(Tf){\rm
ch}_g(\eta)R_g(N_{X/{X_g}})\\
&+\int_{X_g/Y}\widetilde{{\rm Td}}(Tf_g,\omega^X,\omega^Z_X){\rm
ch}_g(\overline{\eta}){\rm Td}_g(\overline{N}_{Z/{Z_g}}){\rm
Td}_g^{-1}(\overline{F})
\end{align*} which ends the proof of
Theorem~\ref{601}.
\end{proof}

\begin{rem}\label{609}
Let $Y$ be an affine arithmetic scheme ${\rm Spec}(D)$, and choose
$\omega^X$ to be the induced K\"{a}hler form $\omega^Z_X$. Then
the formula in Theorem~\ref{601} is the content of
\cite[Conjecture 5.1]{MR}.
\end{rem}

\Addresses
\end{document}